\documentclass[11pt, dvipsnames]{amsart}

\PassOptionsToPackage{dvipsnames}{xcolor}
\usepackage{macros}
\usepackage{hyperref} 
\hypersetup{
  colorlinks = true,
  linkcolor = Fuchsia,
  urlcolor = Fuchsia,
  citecolor = ForestGreen,
  linkbordercolor = {white}}
\usepackage{newpxtext} \usepackage{mathpazo} \usepackage[euler-digits]{eulervm}
\renewcommand{\F}{\mathcal{F}}

\newcommand{\frH}{\mathfrak{H}}
\newcommand{\frJ}{\mathfrak{J}}
\newcommand{\frD}{\mathfrak{D}}
\newcommand{\cC}{\mathcal{C}}
\newcommand{\frP}{\mathfrak{P}}

\usepackage[alphabetic,initials]{amsrefs}

\makeatletter
\let\oldabs\abs
\def\abs{\@ifstar{\oldabs}{\oldabs*}}
\makeatother

\newcommand{\todo}[1]{\textcolor{red}{Todo: #1}}

\DeclareMathOperator{\Jac}{\mathrm{Jac}}
\DeclareMathOperator{\Nm}{\mathrm{Nm}}
\DeclareMathOperator{\Cl}{\mathrm{Cl}}
\DeclareMathOperator{\Ann}{\mathrm{Ann}}

\newcommand{\X}{\mathcal{X}}

\usepackage{fancyhdr}
\begin{document}

\title[Obstructions to unirationality for product-quotient surfaces over $\ol{\FF}_p$]{Obstructions to unirationality for \\ product-quotient surfaces over $\ol{\FF}_p$}
\author{Benjamin Church} 
\address{Department of Mathematics, Stanford University, 
380 Jane Stanford Way, CA 94086} 
\email{{\tt bvchurch@stanford.edu}}

\thanks{}
\date{\today}


\begin{abstract}
We construct a surface over $\ol{\FF}_p$ with $\pi_1^{\et}(X) = 1$ that is supersingular -- in the sense that $H^2_{\et}(X, \Q_{\ell}(1))$ is spanned by algebraic cycles -- but is not unirational. This provides a counterexample to a 1977 conjecture of Shioda. To achieve this, we produce new obstructions to unirationality for product-quotient surfaces.
\end{abstract}

\maketitle
\makeatletter
\newcommand\@dotsep{4.5}
\def\@tocline#1#2#3#4#5#6#7{\relax
  \ifnum #1>\c@tocdepth 
  \else
    \par \addpenalty\@secpenalty\addvspace{#2}%
    \begingroup \hyphenpenalty\@M
    \@ifempty{#4}{%
      \@tempdima\csname r@tocindent\number#1\endcsname\relax
    }{%
      \@tempdima#4\relax
    }%
    \parindent\z@ \leftskip#3\relax
    \advance\leftskip\@tempdima\relax
    \rightskip\@pnumwidth plus1em \parfillskip-\@pnumwidth
    #5\leavevmode\hskip-\@tempdima #6\relax
    \leaders\hbox{$\m@th
      \mkern \@dotsep mu\hbox{.}\mkern \@dotsep mu$}\hfill
    \hbox to\@pnumwidth{\@tocpagenum{#7}}\par
    \nobreak
    \endgroup
  \fi}
\def\l@section{\@tocline{1}{0pt}{1pc}{}{\bfseries}}
\def\l@subsection{\@tocline{2}{0pt}{25pt}{5pc}{}}
\makeatother


\tableofcontents

\section{Introduction}

Let $X$ be a smooth projective variety over $\ol{\FF}_p$. In analogy with the case of elliptic curves, we say that $X$ is \textit{supersingular} if the $q$-Frobenius action on $H^i_{\et}(X, \Q_{\ell})$, for some $\ell \neq p$, has all its eigenvalues equal to $q^{i/2}$ times a root of unity for each $i$. The Tate conjecture predicts, when odd cohomology vanishes, supersingularity is equivalent to \textit{Shioda supersingularity}: that $H^{2i}_{\et}(X, \Q_{\ell}(i))$ is spanned by algebraic cycles.
\par 
In 1977, Shioda made the following conjecture predicting when a surface over $\ol{\FF}_p$ is unirational \cite{Shioda:unirational}.

\noindent
\textbf{Conjecture (Shioda).} \emph{Let $X$ be a smooth projective surface over $\ol{\FF}_p$ with $\pi_1^{\et}(X) = 1$. Then $X$ is unirational if and only if it is Shioda supersingular.}

In this paper, we construct a counterexample to Shioda's conjecture. The examples, whose invariants are given in \S\ref{section:case_D7}--\ref{section:case_A4}, will be the minimal resolution $X$ of certain diagonal quotients $(C \times C)/G$ where $C$ is a Hurwitz curve of genus $14$ with automorphism group $\PSL_2(\FF_{13})$ and $G$ is any copy of the groups $D_{6}, D_{7}$ or $A_4$ inside $\PSL_2(\FF_{13})$. These surfaces are defined over the number field $K_+ = \Q(\cos{\frac{2 \pi}{7}})$ and their reduction at infinitely many places will violate Shioda's conjecture. The first few such characteristics are $1091, 2339, 6551, 7643, \dots$ (see Appendix~\ref{appendix:computations}).   We will actually conclude the following stronger result about curves covering the reduction of $X$.

\begin{Lthm} \label{intro:thm:example}
Let $X$ be one of the surfaces described above. For each prime $\p \subset K_+$ of good reduction, its reduction $X_{\p}$ is not covered by rational or elliptic curves. However, $X$ is simply-connected and $X_{\p}$ is Shioda supersingular for infinitely many $\p$. 
\end{Lthm}

First, we demonstrate new methods to prove that a surface in positive characteristic is not unirational or even covered by elliptic curves. This uses the strategy of Bogomolov's proof \cite{Bog} of finiteness of rational and elliptic curves on a general type surface satisfying $c_1^2 > c_2$ along with an argument by infinite descent. This strategy leads to the following criterion for product-quotient surfaces, those surfaces birational to a diagonal quotient $(C_1 \times C_2) / G$ where $C_i$ are smooth projective curves equipped with faithful $G$-actions.

\begin{Lthm}[Theorem~\ref{thm:obstruction} + Theorem~\ref{thm:obstruction_elliptic}] \label{intro:thm:criterion}
Let $\pi : X \to (C_1 \times C_2)/G$ be a product-quotient surface defined over $\FF_p$. If $X$ carries a symmetric $2m$-form $\omega \in H^0(X, \Sym^{2m} \Omega_X)$ that is ``diagonal'', i.e.\ locally $\omega = f(z_1, z_2) \d{z_1}^m \d{z_2}^m$ (see Definition~\ref{def:diagonal}), then $X$ has only finitely many rational curves of fixed degree. If the vanishing locus of $\omega$ contains a big divisor then there are only finitely many elliptic curves on $X$ of fixed degree. 
\end{Lthm}

The reason for the assumption that $C_1, C_2$ and the $G$-actions are defined over $\FF_p$ is so that we do not need to worry about ``asymmetric Frobenius twists'' of $X$ (see the discussion before Theorem~\ref{thm:obstruction}). In the main text, we will remove this unnecessary assumption at the cost of additional notation and strengthening the hypothesis by requiring symmetric forms on these related ``asymmetric twists''. This more general statement is necessary in the proof of Theorem~\ref{intro:thm:example}.
\par 
Once we know that certain symmetric forms provide an obstruction to unirationality in positive characteristic, we can view Shioda's conjecture as a claim about the geography of complex surfaces using spreading out techniques. In this way, the problem of constructing examples turns into primarily one of complex geometry. Indeed, Shioda would predict that the following three properties are in tension:
\begin{enumerate}
    \item $X$ is simply-connected
    \item $X$ carries a diagonal (see Definition~\ref{def:diagonal}) symmetric form $\omega \in H^0(X, \Sym^{2m}{\Omega_X})$
    \item $X$ has primes of supersingular reduction which, via the philosophy of the Mumford-Tate conjecture, we consider as a constraint on its Hodge structure.
\end{enumerate}
Product-quotient surfaces are, luckily, the unique case where the author knows enough tools to check everything above. Still, in searching for examples, one runs into an apparent dilemma: to show (3) one wants $G$ to be very large to control the Galois representations of $C_1, C_2$ but the larger $G$ is, the more difficult (2) is to achieve. The trick needed to have our cake and eat it too is to consider curves $C_i$ with extremely large automorphism groups but to choose $G$ to be only a modest subgroup of $\Aut(C_i)$ just big enough that $C_i / G \cong \P^1$. This can be viewed as choosing a highly symmetric special point on the Hurwitz space of $G$-covers $C \to \P^1$ so that it has infinitely many primes of supersingular reduction although most other $G$-covers likely do not. A reasonable choice is to consider \textit{Hurwitz curves}, those achieving the Hurwitz bound $\# \Aut(C) = 84(g(C)-1)$. The first time the numerics are in our favor occurs at genus $14$ with the so-called ``first Hurwitz triplet'' consisting of the three isomorphism classes of complex genus $14$ curves with $\Aut(C) \cong \PSL_2(\FF_{13})$ with $84 \cdot (14 - 1) = 1092$ automorphisms. Exploiting these symmetries, we can actually find an example breaking this tension predicted by Shioda's conjecture.
\par 
Bogomolov's method requires the existence of a symmetric differential form (hence why it is usually stated only for surfaces satisfying $c_1^2 > c_2$). Subsequently, many authors sought to generalize Bogomolov's arguments, replacing the symmetric form with a section of a higher jet bundle leading to, for example, the proof of Kobayashi's conjecture for hypersurfaces of very large degree. Likewise, in upcoming work, the author will extend the techniques of this manuscript to employ higher jets as unirationality obstructions. 

\textbf{Conventions.} A \textit{variety} is a geometrically integral separated scheme of finite type over a field. A resolution of surface singularities $\pi : \tilde{S} \to S$ is \textit{minimal} if there is no $(-1)$-curve (possibly defined over $\bar{k}$) contracted by $\pi$. When we say a surface $X$ is \textit{minimal}, we mean it is smooth projective and contains no $(-1)$-curve (possibly defined over $\bar{k}$). When $X$ is smooth projective of nonnegative Kodaira dimension, minimality is equivalent to the condition the condition that $K_X$ is nef. The group $D_{n}$ is the dihedral group with $2n$ elements. For example, $D_{6}$ is the symmetries of the regular hexagon.

\textbf{Acknowledgments.} I thank Daniel Litt, Ravi Vakil, John Voight, Johan de Jong, H\'{e}l\`{e}ne Esnault, Noam Elkies, Peter Sarnak, Nathan Chen, and Anh~Duc~Vo for numerous enlightening conversations and comments on this manuscript. The appendices were made possible by the detailed advice of John Voight. 
Special thanks are due to the Department of Mathematics at Harvard University for its hospitality during the 2024--2025 academic year. During the preparation of this article, the author was partially supported by an NSF Graduate Research Fellowship under grant DMS-2103099 DGE-2146755. It would be impossible to record the special debt this article, and its author, owe to Ming Jing.

\section{Rational curves on product-quotient surfaces} \label{sec:obstruction}

In this section, we will use the existence of a particular type of symmetric differential form -- which we refer to as \textit{diagonal} -- to obstruct unirationality. In particular, we will actually show there are only finitely many rational curves of bounded degree. The idea is inspired by techniques used to establish the hyperbolicity of complex varieties, particularly Bogomolov's proof of boundedness for fixed genus curves on surfaces of general type satisfying $c_1^2 > c_2$. In the particular case of rational curves, Bogomolov uses the fact that any symmetric differential must vanish when restricted to the curve and hence the curve lifts ``holonomically'' to the vanishing locus $Z$ inside $\P_X(\Omega_X)$ of the symmetric form. The fact that the curve on $Z$ arose as a holonomic lift turns out to imply it lies tangent to a particular ``holonomic'' foliation on $Z$. Bogomolov concludes his argument by showing (or invoking a theorem of Jouanolou) that algebraic leaves of a foliation form a bounded family. Up until the last step, the argument works perfectly well in any characteristic. However, the final step fails miserably for foliations in characteristic $p$ because there are far ``too many'' curves tangent to a $p$-closed foliation. Indeed, curves tangent to the foliation are exactly those pulled back from Ekedahl's inseparable quotient by the foliation (see \S\ref{section:Ekedahl}). Often, e.g.\ if the Ekedahl quotient is not of general type, these curves form unbounded families. 
\par 
Our replacement for the final step uses the special ``diagonal'' property of the chosen symmetric differential. This will mean that there is self-similarity between the problem of finding rational curves on $X$ and on the quotients by the induced foliations. This sets up an argument by infinite descent. 

\begin{defn} \label{def:diagonal}
Let $C_1$ and $C_2$ be smooth projective curves. The decomposition
\[ \Omega_{C_1 \times C_2} = \pi_1^* \Omega_{C_1} \oplus \pi_2^* \Omega_{C_2} \]
induces a map
\[ \omega_{C_1 \times C_2} = \pi_1^* \Omega_{C_1} \ot \pi_2^* \Omega_{C_2} \embed \Sym^{2}{\Omega_{C_1 \times C_2}} \]
realizing $\omega_{C_1 \times C_2}$ as a direct summand. We call a symmetric $2m$-form \textit{diagonal} if it arises from a section of $\omega_{C_1 \times C_2}^{\ot m}$ along the $m$-fold power of this map.
\end{defn}

\begin{defn}
Let $f : C_1 \times C_2 \rat X$ be a generically \etale dominant rational from a product of smooth curves to a smooth surface $X$. Call a section $\omega \in H^0(X, \Sym^{2m}{\Omega_X})$ \textit{diagonal} if $f^* \omega$ is diagonal.
\end{defn}

Let $\pi : X \to (C_1 \times C_2)/G$ be the minimal resolution of a product-quotient surface over a perfect field $k$. Denote by $E$ the exceptional divisor of $\pi$. We also need to consider associated ``asymmetrically Frobenius twisted'' surfaces, the minimal resolution \[ \pi^{a,b} : X^{a,b} \to (C_1^{(p^a)} \times C_2^{(p^b)})/G \] 
of the diagonal quotient of the twisted curves with their induced $G$-actions. Here $C^{(p^a)}$ denotes the $a$-th Frobenius twist of $C$ meaning the pullback of $C \to \Spec{k}$ along $\Frob^a : \Spec{k} \to \Spec{k}$ with its functorially induced $G$-action. Note that $X^{a,a} = X^{(p^a)}$. When $C_1$, $C_2$, and the $G$-actions are defined over $\FF_p$, then actually $X$ and $X^{a,b}$ are isomorphic because $C$ and $C^{(p)}$ are $G$-equivariantly isomorphic. Usually $X^{a,b}$ has very similar geometry to $X$ but they are not always isomorphic or even have the same numerical invariants (see Example~\ref{example:twists_have_different_numerics}). Because $G$ acts diagonally, the maps $F_1 : C_1 \times C_2 \to C_1^{(p)} \times C_2$ and $F_2 : C_1 \times C_2 \to C_1 \times C_2^{(p)}$ applying relative Frobenius in the first or second coordinate respectively are $G$-equivariant. Therefore, they induce maps $F_1 : X^{a,b} \to X^{a+1, b}$ and $F_2 : X^{a,b} \to X^{a,b+1}$. 

\begin{theorem} \label{thm:obstruction}
If, for all $a,b \ge 0$, the surface $X^{a,b}$ carries a diagonal symmetric form, then $X$ has only finitely many rational curves of fixed degree. 
\end{theorem}

In particular, this sequence of diagonal symmetric forms give an obstruction to unirationality of $X$. Any section of a tensor power of $\Omega_X$ (or of any Schur functor) obstructs separable unirationality via pulling back to $\P^2$. The difficulty is in ruling out inseparable maps $\P^2 \rat X$. Instead of worrying about this problem, we study the rational curves individually. This has the advantage that, up to reparameterization, we may assume the curves are generically immersed and hence restriction of forms works as in characteristic zero. 

\begin{example}
It is important to note that the existence of some (not necessarily diagonal) symmetric form is not sufficient to obstruct unirationality. For example, Mumford  \cite{pathologies} constructed Enriques surfaces $X$ in characteristic $2$ carrying a $1$-form $\omega$. These $X$ are all unirational, in fact an Enriques surface in characteristic $2$ has a $1$-form if and only if it is unirational by \cite[Proposition 1.3.8]{enriques} and \cite{enriques_unirational}. In this case, the foliation $\F$ associated to $\omega$ defines a quotient map $X \to X / \F$ with $X / \F$ either rational or a unirational K3. We learn that all but finitely many rational curves on $X$ descend to $X / \F$, but of course there are many such curves. 
\end{example}

\subsection{Ekedahl's theory of positive characteristic foliations} \label{section:Ekedahl}

We first review the theory of foliations, what Ekedahl calls $1$-foliations, in positive characteristic developed by Ekedahl \cite{Ek87}.

\begin{defn}
Let $X$ be a smooth variety. A \textit{foliation} $\F$ is a subsheaf $\F \subset \T_X$ of the tangent bundle $\T_X$ of $X$ which is
\begin{enumerate}
    \item \textit{integrable} meaning $[\F, \F] \subset \F$ under the Lie bracket
    \item \textit{saturated} meaning $\T_X / \F$ is torsion free.
\end{enumerate}
\end{defn}

\begin{defn}
Let $\F \subset \T_X$ be a foliation on $X$. We say a subvariety $Z \embed X$ is \textit{tangent} to $\F$  if the map $\T_Z \to \T_X |_Z$ factors through $\F_Z \subset \T_X|_Z$: the saturation of $(\F|_Z)^{\vee \vee} \embed \T_X|_Z$. If additionally $\T_Z \to \T_X|_Z$ is identified with $\F_Z$ we say that $Z$ \textit{is a leaf} of $\F$.
\end{defn}

Note that the property of being tangent depends only on behavior along a dense open set (e.g.\ along the smooth locus of $Z$). From now on in this section, let $k$ be a perfect field of characteristic $p$ and all varieties will be varieties over $k$. Here we let $\Frob_X : X \to X$ denote the ($k$-semilinear) absolute Frobenius and $F_{X/k} : X \to X^{(p)}$ denote the ($k$-linear) relative Frobenius.

\begin{defn}
The $p$-curvature of a foliation $\F$ on $X$ is the $\struct{X}$-linear map
\[ \psi_p : \Frob_X^* \F \to \T_X / \F \quad \partial \mapsto \partial^{[p]} \mod \F \]
where $\partial^{[p]}$ is the derivation obtained by iterating $\partial$ as a differential operator $p$-times. We say that $\F$ is \textit{$p$-closed} if $\psi_p = 0$ or equivalently if $\F^{[p]} \subset \F$. 
\end{defn}

\begin{defn}
Let $f : X \rat Y$ be a dominant rational map of varieties. We say that $f$ is \textit{purely inseparable} if the map of function fields $k(Y) \to k(X)$ is purely inseparable. If, moreover, $f$ is generically finite, then $k(X)^{p^h} \subset k(X)$ for some $h$. The smallest such $h$ is called the \textit{height} of $f$. Note that if $f$ has height $h$ then by definition there is a factorization $g \circ f : X \rat Y \rat X^{(p^h)}$ of the $h$-iterated relative Frobenius $F_{X/k}^h : X \to X$. 
\end{defn}

\begin{theorem}[Ekedahl]
Let $X$ be a smooth variety. There is an equivalence between the following two posets
\begin{enumerate}
    \item $p$-closed foliations $\F \subset \T_X$
    \item finite height-$1$ purely inseparable morphisms $f : X \to Y$ with $Y$ normal
\end{enumerate}
constructed as follows: given a $p$-closed foliation $\F \subset \T_X$ we define $f : X \to X / \F$ where $f$ is the identity of the underlying topological spaces and $\cO_{X/\F} = \Ann_{\cO_X}(\F)$ where explicitly
\[ \struct{X/\F}(U) := \{ s \in \struct{X}(U) \mid \forall \partial \in \F(U) : \partial s = 0 \}. \]
Therefore, $f$ is the map $X \to \mathbf{Spec}_X(\Ann_{\cO_X}(\F))$.
In reverse, given $f : X \to Y$ the $p$-closed foliation is the subsheaf $\T_{X/Y} \subset \T_X$ of relative tangent fields. Furthermore, the correspondence has the following properties:
\begin{enumerate}
    \item $Y$ is smooth iff $\F$ is \textit{nonsingular} ($\F \subset \T_X$ is a subbundle) iff $f$ is flat
    \item there is an exact sequence of $\cO_X$-modules
    \[ 0 \to \F \to \T_X \to f^* \T_Y \to \Frob_X^* \F \to 0 \]
    which implies that
    \[ K_X = f^* K_Y + (p-1) c_1(\F) \]
    \item the map $f^* \T_Y \to \Frob_X^* \F$ is the pullback of a map $\T_Y \to g^* \sigma^* \F$ where $g : Y \to X^{(p)}$ is such that $g \circ f = F_X$ and $\sigma : X^{(p)} \to X$ is the canonical $k$-semilinear isomorphism. Furthermore, the kernel of $\T_Y \to g^* \sigma^* \F$ is a $p$-closed foliation defining $g$. 
\end{enumerate}
\end{theorem}

We will use Ekedahl's constructions to understand the leaves of a foliation $\F$ in characteristic $p$. When $\F$ is $p$-closed, we get the following useful characterization of the leaves of $\F$.

\begin{lemma} \label{lem:tangent_leaf_pullback}
Let $\F$ be a $p$-closed foliation on a smooth variety $X$. Let $Z \subset X$ be a subvariety of dimension $\rank{\F}$ not contained in the singular locus of $\F$. Then $Z$ arises as the (generically reduced) pullback of a subvariety $Z' \subset X / \F$ if and only if $Z$ is a leaf of $\F$.    
\end{lemma}

\begin{proof}
There is always a commutative diagram of relative Frobenii,
\begin{center}
    \begin{tikzcd}
    Z \arrow[r] \arrow[d] & Z^{(p)} \arrow[d]
    \\
    X \arrow[r] & X^{(p)}
    \end{tikzcd}
\end{center}
This is not Cartesian: the fiber product is rather the extension of $Z$ by nilpotents along its conormal sheaf. If $Z$ is a leaf, the map $\T_Z \to \T_X|_Z$ is identified with (the reflexive hull) of $\F_X|_Z \to \T_X$ therefore there is a compatible diagram of quotients by foliations
\begin{center}
    \begin{tikzcd}
     Z \arrow[d] \arrow[r] & Z^{(p)} \arrow[d] \arrow[r, equals] & Z^{(p)} \arrow[d]
     \\
     X \arrow[r] & X / \F \arrow[r] & X^{(p)}
    \end{tikzcd}
\end{center}
so the unique map $Z \to Z^{(p)} \times_{X / \F} X$ must be an isomorphism at the generic point by degree considerations. 
\par 
Since $Z$ is not contained in the singular locus of $\F$, we may shrink to assume that $Y = X / \F$ is smooth and $f : X \to X / \F = Y$ is flat. Now, let $Z' \subset X / \F$ be a subvariety and consider the fiber product $F := Z' \times_{X / \F} X$. By base change, the relative cotangent bundle $\Omega_{F/Z'}$ is the pullback of $\Omega_{X/Y} \cong \F^{\vee}$ where $f : X \to X / \F$ is the quotient map. Hence, if $F$ is generically reduced, i.e., the case when $Z$ is the preimage of $Z'$, then $Z \to Z'$ must have zero differential. Therefore, from the diagram
\begin{center}
    \begin{tikzcd}
    & & \T_Z \arrow[r, "0"] \arrow[d] \arrow[ld, dashed] & \T_{Z'}|_Z \arrow[d]
    \\
    0 \arrow[r] & (\F|_Z)^{\vee \vee} \arrow[r] & \arrow[r] \T_X |_Z \arrow[r] & \T_{X / \F}|_Z
    \end{tikzcd}
\end{center}
we see that $\T_Z \to \T_X|_Z$ factors through $(\F|_Z)^{\vee \vee} \to \T_X|_Z$. 
\end{proof}

\subsection{Bogomolov's strategy for boundedness of curves in characteristic zero} \label{section:bogomolov_strategy}

Before we prove the theorem, let us recall Bogomolov's construction in detail. Let $X$ be a smooth projective surface. For any generically immersive map $f : C \to X$ from a smooth projective curve to $X$, there is a unique ``holonomic'' lift $t_f : C \to \P_X(\Omega_X)$ defined by the map of bundles
\[ f^* : f^* \Omega_X \onto \im{f^*} \subset \omega_C \]
where $\im{f^*}$ is torsion-free of rank $1$ and hence a line bundle. This defines $t_f$ in such a way that $t_f^* \cO_{\P}(1) = \im{f^*}$. At the points $t \in C$ where $f$ is an immersion, the map $t_f$ is defined by $s \mapsto (f(s), \d{f}_s(\partial_s))$, viewing $\P := \P(\Omega_X)$ as the projectivization of the physical tangent bundle. Since $C$ is a smooth projective curve and $\P_X(\Omega_X)$ is proper, the extension over the points where $f$ is not an immersion is unique. 
\par 
Now let $\omega \in H^0(X, \Sym^{m}{\Omega_X})$ be a nonzero symmetric form. We can view
\[ \omega \in H^0(X, \Sym^{m}{\Omega_X}) = H^0(\P, \cO_{\P}(m)) \]
as a section of the line bundle $\cO_{\P}(m)$ hence its vanishing locus on $\P$ is a divisor $Z$. This represents the locus $(x, v) \in \P$ where $\omega_x(v) = 0$. Since $\P^1$ carries no symmetric differential forms, for any map $f : \P^1 \to X$ we must have $f^* \omega = 0$. By definition, this means $t_f : \P^1 \to X$ factors through $Z$. Similarly, for an elliptic curve $C$ and a map $f : C \to X$, either $f^* \omega$ is zero identically or is zero nowhere. Hence, if $V(\omega) = \{x \in X \mid \omega_x = 0\}$ contains a big divisor then at most finitely many curves may not meet $V(\omega)$. Except for a finite number of exceptions, for each elliptic curve, $f^* \omega$ is therefore zero. 
\par 
Since $\dim{X} = 2$, the divisor $Z$ is a union of surfaces. Let $S \to Z$ be a resolutions of singularities of a component of $Z$ dominating $X$. Except for finitely many rational or elliptic curves -- namely those which equal the image in $X$ of a component of $Z$ or which do not meet $V(\omega)$ -- all rational or elliptic curves on $X$ lift to a map $\wt{f} : C \to S$. Let $g : S \to X$ be the induced morphism. The quotient map $\pi^* \Omega_X \onto \cO_{\P}(1)$ defines a foliation on each component of $Z$. 

\begin{defn}
The \textit{induced holonomic foliation} on $S$ is defined as follows. Consider the pushout diagram,
\begin{center}
    \begin{tikzcd}
        g^* \Omega_X \arrow[r, two heads] \arrow[d, hook] & \iota^* \cO_{\P}(1) \arrow[d, hook]
        \\
        \Omega_{S} \arrow[r, two heads] & \pushout \cQ 
    \end{tikzcd}
\end{center}
where $\iota : S \to \P$ is the obvious map. Since $\Omega_{S} \onto \cQ$ is surjective, its dual $\F \embed \T_{S}$, with $\F = \cQ^{\vee}$, is a saturated reflexive subsheaf of generic rank $1$. Therefore, $[\F, \F] \subset \F$ because $\F$ has rank $1$ so it defines a foliation.  
\end{defn}

\begin{lemma} \label{lemma:holonomic_leaf}
Let $f : C \to X$ be a map from a smooth proper curve $C$ such that $f^* \omega = 0$. Then the lift $\wt{f} : C \to S$ is a leaf of $\F$.    
\end{lemma}

\begin{proof}
The lifted map $t_f : C \to \P$ is induced by the factorization 
\[ f^* \Omega_X \to t_f^* \cO_{\P}(1) \to \omega_C \]
using that $t_f$ factors through $S$ the following diagram shows that $\wt{f}^* \Omega_{S} \to \omega_C$ factors through $\wt{f}^* \Omega_{S} \onto \wt{f}^* \cQ$,
\begin{center}
    \begin{tikzcd}
    \wt{f}^* g^* \Omega_X \arrow[r] \arrow[d] & \wt{f}^* \iota^* \cO_{\P}(1) \arrow[d] \arrow[rdd, bend left]
    \\
    \wt{f}^* \Omega_{S} \arrow[r] \arrow[rrd, bend right] & \wt{f}^* \cQ \arrow[rd, dashed]
    \\
    & & \omega_C
    \end{tikzcd}
\end{center}
\end{proof}

From here, Bogomolov concludes by showing that each foliation can have at most finitely many rational leaves. This requires analytic input that fails in positive characteristic. In the subsequent section, we will replace this step with a detailed study of the successive Ekedahl quotients by these induced foliations.

\subsection{Proof of Theorem~\ref{thm:obstruction}}

The idea of the proof is as follows. For a rational curve $h : C \to X$ consider its lift $t_h :  C \to \P_X(\Omega_X)$. Because $h^* \omega = 0$, the lift lands inside the vanishing locus $Z$ of $\omega$ viewed as a section of $\cO_{\P}(2m)$. Either $h$ is contained in a component that does not dominate $X$, which comprises a finite number of curves that can be set aside, or it lies tangent to the holonomic foliation. We will identify these holonomic foliations with those induced on $X$ by the projections $\pi_i : C_1 \times C_2 \to C_i$ through the quotient structure. So $h$ must be tangent to, for example, the foliation induced by $\pi_2$. Applying Lemma~\ref{lem:tangent_leaf_pullback}, it arises as the pullback of $h' : C' \to X'$ along the quotient map $X \to X'$. This step is captured in Figure~\ref{fig:descent_step}. It turns out that $X'$ is also a product-quotient surface, birationally equivalent to $X^{1,0}$. We assume that $X^{1,0}$ also carries a diagonal symmetric form so we can run the argument again. Since $\deg{\pi_1 \circ h}$ or $\deg{\pi_2 \circ h}$ drops at each stage, this sets up a contradiction by infinite descent. 

\begin{figure} \label{fig:descent_step}
    \centering
    \includegraphics[trim={5cm 0 0 0},clip, angle=90, width=1\linewidth]{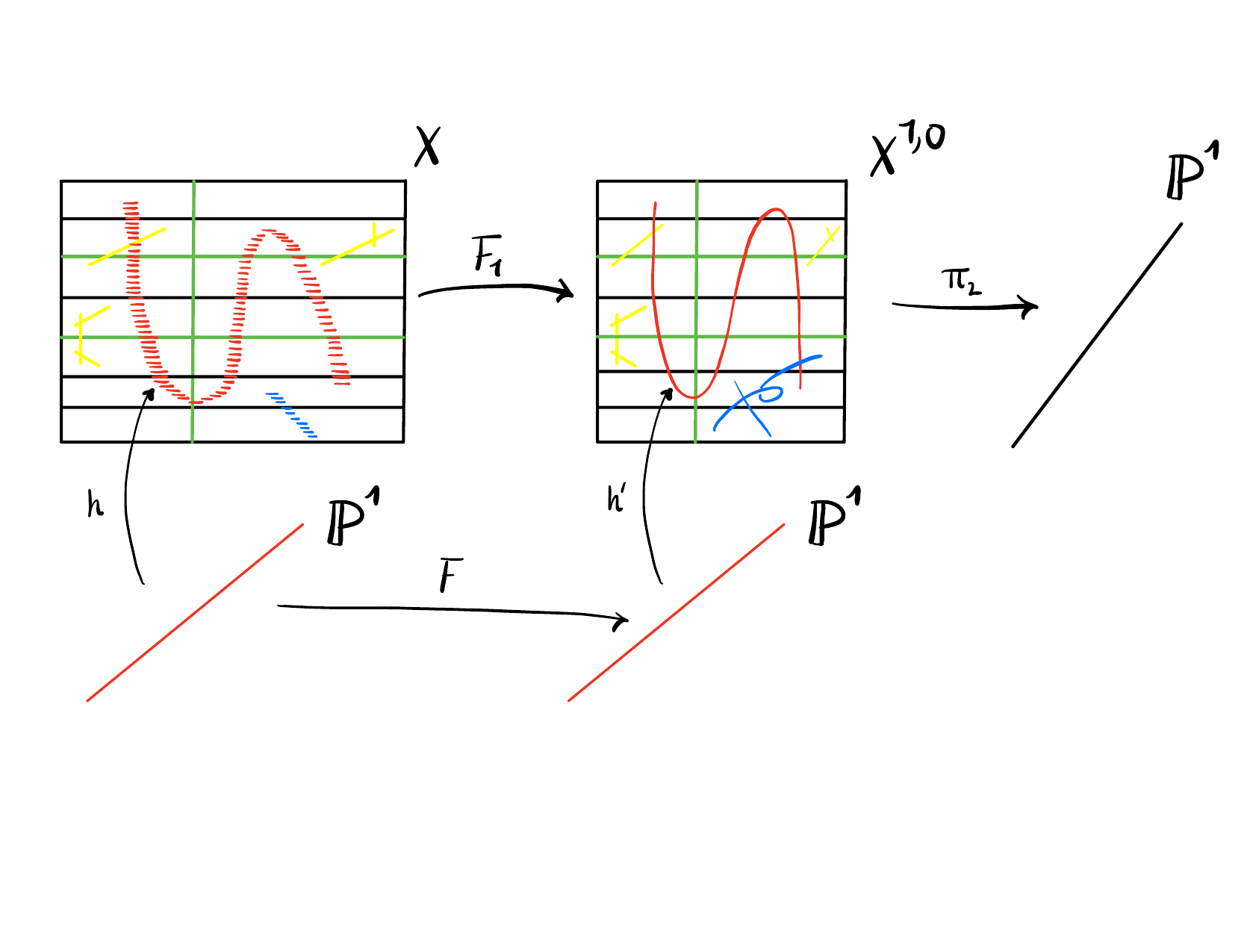}
    \caption{Descent of a rational curve (red) along $X \to X' = X^{1,0}$. Beginning with a rational curve $h : \P^1 \to X$ lifting to a curve on $S$ tangent to the foliation $\F$ -- represented here by the curve being formed from segments tangent to the fibers of $\pi_2$ -- it is the pullback of a rational curve $h'$ on $X'$ with $\deg{\pi_2 \circ h} = p \deg{\pi_2 \circ h'}$.   Exceptional divisors of the resolution (curves of type (1) in yellow) and fibers of either projection that are rational (curves of type (2) in green) are pulled back to themselves. However, rational components of $V(\omega^{1,0})$ (curves of type (3) in blue) have their pullback under partial Frobenius is again rational of $p$-times the degree. Such a pullback is denoted on $X$ (blue, marked as $p$-th order tangent to the fibers of $\pi_2$).}
    \label{fig:enter-label}
\end{figure}

\begin{proof}[Proof of Theorem]
We proceed by induction on the degree of $q : C_1 \times C_2 \rat X$ (i.e., the order of $G$). For $\deg{q} = 1$, the conclusion is obvious. 

Now choose diagonal symmetric forms $\omega^{a,b} \in H^0(X^{a,b}, \Sym^{2m}{\Omega_{X^{a,b}}})$ on each surface $X^{a,b}$. Denote by $\omega := \omega^{0,0}$ the form on $X$. Let $Z = Z(\omega) \subset \P_X(\Omega_X)$ be the vanishing locus of $\omega$ viewed as a section of $\cO_{\P}(2m)$. Denote by $V(\omega) \subset X$ the vanishing locus of $\omega$ viewed as a section of $\Sym^{2m}{\Omega_X}$. We will first set aside a finite collection of rational curves on each $X^{a,b}$ grouped into the following types:
\begin{enumerate}
    \item exceptional curves of the resolution $\pi^{a,b}$
    \item images under $q$ of the fibers of the projection maps $\pi_1$ and $\pi_2$ on $C_1 \times C_2$
    \item images of components of $Z$ not dominating $X$ 
\end{enumerate}
A curve is labeled uniquely by the first applicable type in the list (i.e., curves of type (3) are by definition not type (1) or (2)). Note that there are only finitely many curves of type (2) since otherwise $q^* \omega$ would have to be zero, but $q$ is assumed to be \etale. We run the construction of \S\ref{section:bogomolov_strategy} on $\P_X(\Omega_X)$ compatibly with the covering product surface. This gives a diagram,
\begin{center}
    \begin{tikzcd}
        Z(q^* \omega) \arrow[d, hook] \arrow[r, dashed] & Z(\omega) \arrow[d, hook]
        \\
        \P(\Omega_{C_1 \times C_2}) \arrow[d] \arrow[r, dashed] & \P_X(\Omega_X) \arrow[d]
        \\
        C_1 \times C_2 \arrow[r, "f", dashed] & X 
    \end{tikzcd}
\end{center}
We are only interested in components of $Z$ which dominate $X$ since there are at most finitely many curves that lift to the other components (since their image is a union of finitely many curves in $X$). Each of those components of $Z$ is dominated by a corresponding component of $Z(f^* \omega)$ dominating $C_1 \times C_2$. However, $q^* \omega$ is a diagonal form, meaning in local coordinates it has the form $f(z_1, z_2) \d{z_1}^m \d{z_2}^m$. Hence the components of $Z(q^* \omega)$ are those lying over components of the zero locus of $q^* \omega$ on $C_1 \times C_2$ and copies of the two sections $C_1 \times C_2 \to \P(\Omega_{C_1 \times C_2})$ corresponding to the quotients $\Omega_{C_1 \times C_2} \onto \pi_i^* \Omega_{C_i}$. Therefore, the resolution $g : S \to X$ of each component of $Z$ that dominates $X$ fits into a diagram
\[ C_1 \times C_2 \rat S \to X \]
factoring the rational map $q$. Moreover, the induced holonomic foliations are compatible so $\F \subset \T_S$ pulls back to one of the standard foliations $\pi_i^* \T_{C_i} \subset \T_{C_1 \times C_2}$. In particular, $\F$ is $p$-closed since this can be checked \etale-locally. By Galois theory, $C_1 \times C_2 \rat S$ makes $S$ birational to a product-quotient surface for some subgroup $H \subset G$. Since $q$ is generically \etale, so is $S \to X$ hence $g^* \omega$ is a global diagonal symmetric form on $S$. If $g : S \to X$ is not birational, then the degree of $C_1 \times C_2 \rat S$ is strictly less than $\deg{q}$ so the induction hypothesis applies, proving the conclusion for $S$. Since all rational curves on $X$ not of type (1) - (3) lift along one of the finitely many dominant maps $g : S \to X$, we have reduced to showing there are finitely many rational curves lifting to $g : S \to X$ for $g$ birational.
\par 
We may assume that the foliation $\F$ on $S$ pulls back to the foliation $\pi_1^* \T_{C_1} \subset \T_{C_1 \times C_2}$ on $C_1 \times C_2$ defined by the fibration $\pi_2 : C_1 \times C_2 \to C_2$. The other case is completely symmetric. The quotient of $C_1 \times C_2$ by this foliation is the geometric Frobenius on $C_1$ i.e., the map $F_1 : C_1 \times C_2 \to C_1^{(p)} \times C_2$ applying geometric Frobenius in the first coordinate. Because $G$ acts diagonally, the map $F_1 : C_1 \times C_2 \to C_1^{(p)} \times C_2$ is $G$-equivariant so it produces a map $X \to X'$ of the minimal resolutions of the quotients. In our previous notation, $X' = X^{1,0}$. Because the partial Frobenii are $G$-equivariant, there is a diagram,
\begin{center}
    \begin{tikzcd}
        C_1 \times C_2 \arrow[r, "F_1"] \arrow[d, dashed, "q"] & C_1^{(p)} \times C_2 \arrow[d, dashed, "f'"] \arrow[r, "F_2"] & C_1^{(p)} \times C_2^{(p)} \arrow[d, dashed, "q^{(p)}"]
        \\
        X \arrow[r] & X' \arrow[r, dashed] & X^{(p)}
        \\
        S \arrow[u] \arrow[r] & S / \F \arrow[u, dashed] \arrow[r] & S^{(p)} \arrow[u]
    \end{tikzcd}
\end{center}
Each of the top squares is generically Cartesian because $q$ and $q'$ are generically \etale. 
The map $S \to X \rat X'$ must factor through $S / \F$ because its tangent map kills $\F \subset \T_{S}$. This is because $X \rat X'$ is \etale-locally identified with $F_1$ and $\F$ pulls back to the foliation $\pi_1^* \T_{C_1} \subset \T_{C_1 \times C_2}$ on $C_1 \times C_2$ which induces $F_1$. Now, since both $S \to S / \F$ and $S \to X \rat X'$ have degree $p$, we conclude that $S / \F \rat X'$ is birational. 
\par 
This is relevant for the following reason. Let $h : C \to X$ be a rational curve not of type (1). Since $h^* \omega = 0$, the lift $t_h : C \to \P_X(\Omega_X)$ factors through one of the components of $Z$. As before, we restrict our attention to those curves $h$ factoring through $S \to X$. By Lemma~\ref{lemma:holonomic_leaf}, $\tilde{h} : C \to S$ is a leaf of $\F$ and therefore arises as the pullback of a curve $h' : C' \to X'$ via an application of Lemma~\ref{lem:tangent_leaf_pullback}. Because $C \to C'$ is a dominant map of curves, $C'$ is also rational.
\par 
Starting with $h : C \to X$, we have produced a ``simpler'' curve $h' : C' \to X'$ on a similar surface $X'$. Precisely, $\deg{\pi_2 \circ h'} = \frac{1}{p} \cdot \deg{\pi_2 \circ h}$ because the diagram
\begin{center}
\begin{tikzcd}
X \arrow[r] \arrow[r, bend left, "\pi_2"] & X' \arrow[r, "\pi_2"] & C_2 
\\
C \arrow[u, "h"] \arrow[r, "F"] & C' \arrow[u, "h'"]
\end{tikzcd}
\end{center}
commutes (see also Figure~\ref{fig:descent_step}). Intuitively, this is because, for each point of intersection between $h'$ and a fiber, the corresponding intersection point of $h$ on $X$ has multiplicity $p$ because the curve is made $p$-th order tangent to the fibers. Furthermore, $X' = X^{1,0}$ also carries a diagonal symmetric form by assumption, so we may repeat this process with $h'$. Restoring the symmetry of the situation, at each step either the sum $\deg{\pi_1 \circ h} + \deg{\pi_2 \circ h}$ drops or one of $\deg{\pi_i \circ h} = 0$. meaning $h$ lives inside a fiber of one of the projections. Those rational curves are already accounted for in type (1) or (2). Otherwise, as we iterate this process, the degree against one of the $\pi_i$ must decrease so eventually $h^{(i)} : C^{(i)} \to X^{a,b}$ must terminate at one of the finitely many rational curves already set aside where the process cannot continue. 
\par 
Let us take stock of the possibilities for rational curves on $X$ not currently ruled out. There are the curves of type (1) - (3) set aside earlier and all rational curves which can produced by pulling back these curves a finite number of times along a sequence of partial Frobenii. Notice that only curves of type (3) may proliferate after applying Frobenii since fibers of $\pi_i$ pull back to fibers and exceptional curves pull back to exceptional curves. Since the remaining curves have positive degree against both projections $\pi_1$ and $\pi_2$, the partial Frobenius pullback increases their degree against the ample divisor $F_1 + F_2$. Since there are only finitely many curves of type (1), and at each stage we apply either $F_1^{-1}$ or $F_2^{-1}$, there are only finitely many rational curves of each fixed degree on $X$. 
\end{proof}

\begin{rmk} \label{rmk:counting}
Explicitly, every rational curve on $X$ is either type (1) - (3) or is the pullback of a curve of type (1) on $X^{a,b}$. 
Note that, $F_1 \circ F_2 = F_X$ is simply the relative Frobenius of $X$ so any curve pulled back along both partial Frobenii becomes non-reduced. Alternatively, one can say that no reduced curve can be tangent to both foliations simultaneously. Therefore, in total the rational curves on $X$ are 
\[ \{ \text{type (1) - (3) on } X \} \cup \bigcup_{a \ge 1} (F_1^a)^{-1} \{ \text{type (1) on } X^{a,0} \} \cup \bigcup_{b \ge 1} (F_2^b)^{-1} \{ \text{type (1) on } X^{0,b} \} \]
so this analysis produces an upper bound of 
\[ \sum_{a \le \log_p{d}} \# \{ \text{type (3) on } X^{a,0}\} + \sum_{b \le \log_p{d}} \# \{ \text{type (3) on } X^{0,b}\} \]
for the number of rational curves on $X$ not of type (1) or (2) of degree $\le d$. Notice moreover, that any curve of type (3) on $X'$ not tangent to $F_2$ does pull back along $F_1$ to a reduced rational curve on $X$ with $p$ times the degree (it remains rational because $F_1^{-1}(C) \to C$ factors Frobenius $F_C : C^{(1/p)} \to C$ since $F_X^{-1}(C)$ is a non-reduced copy of $C^{(1/p)}$). If $X$ is defined over $\ol{\FF}_p$, then only finitely many $X^{a,b}$ are non-isomorphic so the collection of curves $\{ \text{type (3) on } X^{a,} \}$ has uniformly bounded size. In this case, if any such rational curve of type (3) exists, then $X$ genuinely does contain infinitely many rational curves although those curves will have exponentially growing degree. 
\end{rmk}

\subsection{Elliptic curves}

A similar argument applies to elliptic curves as well. 

\begin{theorem} \label{thm:obstruction_elliptic}
Suppose, for all $a,b \ge 0$, the surface $X^{a,b}$ carries a diagonal symmetric form whose vanishing locus in $X^{a,b}$ contains a big divisor. Then $X$ has only finitely many rational curves of fixed degree. If the vanishing loci of these diagonal forms contain a big divisor then there are only finitely many elliptic curves on $X$ of fixed degree. 
\end{theorem}

\begin{proof}
All the induction and descent steps are identical. However, for an elliptic curve $h : C \to X$, it is not automatic that $h^* \omega = 0$. What is true is that if $h^* \omega$ has a zero then $h^* \omega = 0$ identically. To force the lift $t_h : C \to \P_X(\Omega_X)$ to factor through $Z(\omega)$, we will want $h(C)$ to meet $V(\omega) \subset X$ the locus where $\omega$ is zero as a section of $\Sym^{2m} \Omega_X$. This is how the extra assumption is used. If $V(\omega)$ contains a big divisor $D \subset X$ then all but finitely many curves must meet $V(\omega)$. Therefore, we must amending the types of curves on $X^{a,b}$ set aside at the beginning. Indeed, consider the finitely many elliptic curves on $X^{a,b}$ grouped into the following types:
\begin{enumerate}
    \item exceptional curves of the resolution $\pi^{a,b}$
    \item images under $q$ of the fibers of the projection maps $\pi_i : C_1 \times C_2 \to C_i$
    \item images of components of $Z$ not dominating $X$ 
    \item curves not meeting $V(\omega)$.
\end{enumerate}
Now the argument proceeds exactly as before by forcing any curve not of type (1) - (4) above to arise as the pullback along a sequence of partial Frobenii. Now both curves of type (3) and (4) may proliferate under these operations. However, since the sum of their degrees increases at each stage by at least a factor of $p$, there can only be finitely many rational and elliptic curves in each fixed degree. 
\end{proof}

The explicit description of curves on $X$ immediately gives the following corollary.

\begin{theorem}
Suppose on each $X^{a,b}$ the diagonal symmetric forms span a line bundle  $\L \embed \Sym^{2m} \Omega_{X^{a,b}}$ whose stable base locus $\BB(\L) := \bigcap_{m \ge 0} B(\L^m)$ contains no rational (resp.\ elliptic) curves outside the exceptional locus of $\pi^{a,b}$, then $X$ contains only finitely many rational (resp.\ elliptic) curves.    
\end{theorem}

\begin{proof}
Since the argument can be run for any diagonal symmetric form $\omega^{a,b}$ we can restrict further types (3) (resp.\ types (3) and (4)) to only include curves in the stable base locus formed by intersecting $V(\omega^{a,b})$ running over all diagonal symmetric forms. 
The assumption means there are no rational (resp.\ elliptic) curves of type (3) (resp.\ types (3) or (4)) on any $X^{a,b}$. Hence the descent process must terminate in curves of type (1) or (2) but these are mapped to themselves under partial Frobenius preimage so $X$ has only the finite number of rational (resp.\ elliptic) curves of types (1) and (2). 
\end{proof}

In remark~\ref{rmk:counting}, we saw this result is almost sharp. If the vanishing locus of the symmetric form $\omega^{a,b}$ contains a rational (resp.\ elliptic) curve (not of type (1) or (2)) then its partial Frobenius iterates (under whichever of $F_1$ or $F_2$ it is not tangent to) form a sequence of rational (resp.\ elliptic) curves of unbounded degree. 
\par 
In characteristic zero, the same methods apply much more directly since curves tangent to the standard foliations must actually be fibers. In that case, we obtain finiteness without disallowing curves of type (3) hence recovering the main results of  \cite{lang_product_quotient} by essentially the same argument.

\section{Invariants of product-quotient surfaces}

Throughout this section, fix a finite group $G$ and two smooth projective curves $C_1, C_2$ defined on complex numbers, each equipped with a faithful action of the finite group $G$. Let $Y = (C_1 \times C_2) / G$ be the quotient by the diagonal action and $X \to Y$ the minimal resolution of the isolated cyclic quotient singularities of $X$. Write $Y^{\circ}$ for the smooth locus of $Y$, which is exactly the locus over which $C_1 \times C_2 \to Y$ is \etale. We also assume that $C_i / G \cong \P^1$; otherwise, the fundamental group of $X$ will always be infinite. These varieties fit into a diagram
\begin{center}
    \begin{tikzcd}
        & & & & X \arrow[d, "\varphi"] \arrow[rdd, bend left, "\sigma_2"] \arrow[ldd, bend right, "\sigma_1"']
        \\
        & C_1 \times C_2 \arrow[ld, "\pi_1"] \arrow[rd, "\pi_2"'] \arrow[rrr, "q", crossing over] & & & Y := (C_1 \times C_2)/G \arrow[rd, "\pi_2"'] \arrow[ld, "\pi_1"]
        \\
        C_1 \arrow[rrr, bend right] & & C_2 \arrow[rrr, bend right, crossing over] & \P^1 & & \P^1
    \end{tikzcd}
\end{center}

In this section, we compile computations of the various invariants of the surface $X$ in terms of the data $G, C_1, C_2$ and the actions. We state these results over the complex numbers for simplicity and because the proofs, especially those related to the fundamental group, rely on topological methods. Later, we will use base change and spreading out techniques to apply these results in arithmetic settings of interest.

\subsection{Finite quotient singularities} \label{section:quotient_sing}

\begin{defn}
A surface singularity is called \textit{finite quotient} if it is analytically isomorphic to the singularity produced as the quotient $\A^2 / G$ where $G \subset \GL_2(\CC)$ is a finite group. The singularity is called \textit{cyclic} if $G$ can be taken to be a cyclic group. In that case, the action is always equivalent to one generated by the diagonal matrix
\[ 
\begin{pmatrix}
    \zeta_n & 0
    \\
    0 & \zeta_n^a
\end{pmatrix} 
\]
with $\zeta_n := e^{2 \pi i / n}$ and $\gcd(a,n) = 1$. We call this a $\frac{1}{n} (1,a)$ singularity.
\end{defn}

\begin{rmk}
Note that among these, the \textit{canonical} or \textit{ADE} surface singularities are those for which we can choose $G \subset \SL_2(\CC)$. Among these, the cyclic canonical singularities are exactly those of type $\frac{1}{n+1}(1,n)$ called type $A_n$.
\end{rmk}

Cyclic quotient singularities are always klt, hence rational $\Q$-Gorenstein. For the reader's convenience, we recall some properties of the well-known Hirzebruch--Jung resolution, but the interested reader might also consult \cite[\S2.6]{Fulton:toric_varieties}.
The minimal resolution of a quotient singularity has, as its exceptional divisor, a tree of rational curves. For the cyclic quotient singularities, this tree is a chain of rational curves with self-intersections determined by the arithmetic of the fraction $\frac{n}{a}$. Let $\dbrac{b_1, b_2, \cdots, b_{\ell}}$ denote the continued fraction representation of $\frac{n}{a}$ using the convention
\[ \dbrac{b_1, b_2, \cdots, b_{\ell}} := b_1 - \frac{1}{b_2 - \frac{1}{b_3 - \cdots}} = \frac{n}{a}. \]

\begin{prop}
The exceptional divisor of the minimal resolution of a $\frac{1}{n} (1,a)$ singularity is a chain $Z_1 +  \dots + Z_{\ell}$ of smooth rational curves $Z_i$ ordered such that $Z_i \cdot Z_{i+1} = 1$ and $Z_i \cdot Z_j = 0$ for $|i - j| > 1$ and with self-intersections $Z_i^2 = -b_i$ where $b_1, \dots, b_{\ell}$ are the numbers appearing in the continued fraction representation $\frac{n}{a} = \dbrac{b_1, b_2, \cdots, b_{\ell}}$. Therefore, the dual graph of $E$ is \begin{center}
\begin{tikzpicture}
    \node[circle, fill=black, inner sep=2pt] (n1) at (0,0) {};
    \node[circle, fill=black, inner sep=2pt] (n2) at (2,0) {};
    \node[circle, fill=black, inner sep=2pt] (n3) at (5,0) {};
    \node[circle, fill=black, inner sep=2pt] (n4) at (7,0) {};
    
    \draw[thick] (n1) -- (n2);
    \draw[thick] (n3) -- (n4);
    
    \draw[thick, dashed] (n2) -- (n3);
    
    \node[above] at (n1.north) {$-b_1$};
    \node[above] at (n2.north) {$-b_2$};
    \node[above] at (n3.north) {$-b_{k-1}$};
    \node[above] at (n4.north) {$-b_k$};
\end{tikzpicture}
\end{center}
\end{prop}

\subsection{Chern numbers}

Denote the genera $g(C_i)$ by $g_i$. Following \cite[\S1]{bauer:classification}, we refer to the multi-set of \textit{types} $\tfrac{1}{n} (1,a)$ (where the types are symbols subject to the relation $\tfrac{1}{n}(1, a) = \tfrac{1}{n}(1, a')$ iff $a' \equiv a^{\pm 1} \mod n$) of singularities on $Y$ as the \textit{basket} $\cB(Y)$ of singularities. Furthermore, we define the following numerical invariants of a cyclic quotient singularity.

\begin{defn} (cf.\ \cite[Definition~1.4]{bauer:classification})
Let $\frac{n}{a} = \dbrac{b_1, \dots, b_{\ell}}$ be the continued fraction representation and $0 < a' < n$ the inverse of $a$ in $(\Z / n \Z)^{\times}$. Then define,
\begin{enumerate}
    \item $k(\tfrac{n}{a}) := -2 + \frac{2 + a + a'}{n} + \sum_i (b_i - 2)$
    \item $e(\tfrac{n}{a}) := \ell + 1 - \frac{1}{n}$
    \item $B(\tfrac{n}{a}) := 2 e(\tfrac{n}{a}) + k(\tfrac{n}{a})$
\end{enumerate}
Let $\cB$ be a basket of singularities. Then we define the invariants of the basket by summation,
\[ k(\cB) := \sum_{x \in \cB} k(x) \quad e(\cB) := \sum_{x \in \cB} e(x) \quad B(\cB) := \sum_{x \in \cB} B(x) = 3 k(\cB) + k(\cB) \]
\end{defn}

\begin{prop} \cite[Proposition 1.5]{bauer:classification} \label{prop:chern_numerics}
In the notation of this section
\[ K_X^2 = \frac{8(g_1 - 1)(g_2 - 1)}{\# G} - k(\mathcal{B}) \quad \text{ and } \quad c_2(X) = \frac{4(g_1 - 1)(g_2 - 1)}{\# G} + e(\mathcal{B}) \]
\end{prop} 

\subsection{The fundamental group}

In this section, we suppose that the curve $C$ along with $G$-action satisfying $C / G = \P^1$ is presented, via the Riemann existence theorem, as the topological covering space associated to a map $\varphi : \pi_1(\P^1 \sm \{ p_1, \dots, p_r \}) \onto G$ sending the simple loop around $p_i$ to the monodromy $g_i = \varphi(\gamma_i)$ around $p_i$. Call the sequence of elements $g_1, \dots, g_r$ a \textit{spherical} list of generators, meaning a list that generates and satisfies $g_1 \dots g_r = 1$.  Write $m_i$ for the order of the monodromy $g_i = \varphi_i(\gamma_i)$ around $p_i$. It is more convenient to consider the finitely presented group:

\begin{defn}
Define the following group
\[ \TT(m_1, \dots, m_r) = \left< c_1, \dots, c_r \mid c_1^{m_1} = 1, \dots, c_r^{m_r} = 1, c_1 \cdots c_r = 1 \right> \]
\end{defn}

\begin{rmk}
$\TT(m_1, \dots, m_r)$ is the fundamental group of the orbifold curve $[C / G]$ which is $\P^1$ with a cyclic orbifold stabilizer at $p_i$ of order $m_i$. 
\end{rmk}

The cover $C \to \P^1$ can be interpreted as the \etale cover of orbifolds $C \to [C / G]$ and is determined by $\varphi : \TT(m_1, \dots, m_r) \onto G$ sending $c_i$ to the monodromy around $p_i$. In what follows, we have two curves $C_1, C_2$ each presented as a $G$-cover of $\P^1$. Let $m_1, \dots, m_r$ be the orders of monodromy for $C_1 \to \P^1$ and let $n_1, \dots, n_s$ be the orders of monodromy for $C_2 \to \P^1$ and let $\varphi_1 : \TT(m_1, \dots, m_r) \onto G$ and $\varphi_2 : \TT(n_1, \dots, n_s) \onto G$ be the classifying maps defined by spherical generators $g_1, \dots, g_r \in G$ and $h_1, \dots, h_s \in G$ respectively.

\begin{lemma}
The fundamental group of $Y^{\circ}$ fits into a diagram of exact sequences sequences
\begin{center}
    \begin{tikzcd}
     1 \arrow[r] & \pi_1(C_1) \times \pi_1(C_2) \arrow[r] \arrow[d, equals] & \pi_1(Y^{\circ}) \arrow[r] \arrow[d, hook] \pullback & G \arrow[d, hook, "\Delta"] \arrow[r] & 1
     \\
     1 \arrow[r] & \pi_1(C_1) \times \pi_1(C_2) \arrow[r] & \TT(m_1, \dots, m_r) \times \TT(n_1, \dots, n_s) \arrow[r] & G \times G \arrow[r] & 1
    \end{tikzcd}
\end{center}
\end{lemma}

\begin{proof}
This arises from the diagram of \etale covers of orbifolds,
\begin{center}
\begin{tikzcd}
(C_1 \times C_2)^{\circ} \arrow[r, hook] \arrow[d, "\et"] & C_1 \times C_2 \arrow[r, equals] \arrow[d, "\et"] & C_1 \times C_2 \arrow[d, "\et"]
\\
Y^{\circ} \arrow[r] & \left[(C_1 \times C_2) / G \right] \arrow[r] & \left[ C_1 / G \right] \times \left[ C_2 / G \right]
\end{tikzcd}
\end{center}
where the left downward map is a $G$-cover and the right is a $(G \times G)$-cover compatible with the diagonal inclusion $\Delta : G \embed (G \times G)$. Furthermore, since the inclusion of a codimension $\ge 2$ subset of a smooth variety is an isomorphism on $\pi_1$, we conclude that the natural map
\[ \pi_1(Y^{\circ}) \to \pi_1([(C_1 \times C_2) / G]) \]
is an isomorphism because $\pi_1((C_1 \times C_2)^{\circ}) \iso \pi_1(C_1 \times C_2)$ is an isomorphism and the fibration exact sequences form a commutative diagram.
\end{proof}

\begin{lemma} \label{lemma:fund_group_torsion}
$\pi_1(X) = \pi_1(Y) = \pi_1(Y^{\circ}) / \left< \mathrm{tors} \right>$ where $\left< \mathrm{tors} \right>$ means the normal subgroup generated by elements of finite order. 
\end{lemma}

\begin{proof}
That the map $\pi_1(X) \to \pi_1(Y)$ is an isomorphism is well-known for rational surface singularities. Since $Y$ has klt singularities, it also follows from \cite[Theorem 7.8]{Kol93}. The second statement follows from the main result of \cites{Arm65, Arm68}; see also \cite[Proposition 3.4]{bauer:fundamental_groups} 
\end{proof}

In the remainder of the section, we prove a sufficient condition to check that a product-quotient surface with $C_1 = C_2$ (really, we just need that they have the same monodromy representation over $\P^1$ although the branch points may be different) is simply-connected. This will be useful in examples.
\par
From now on, we are working with $(g_1, \dots, g_r) = (h_1, \dots, h_s)$. Recall that $E := \pi_1(Y^\circ)$ is the subgroup of $\TT(m_1, \cdots, m_r)^2$ on which the two projections to $G$ agree. We want conditions on $(g_1, \dots, g_r)$ for which $E$ is normally generated by torsion elements.

\begin{defn}
We say that a list of spherical generators $(g_1, \dots, g_r)$ of a finite group $G$ \textit{extends a good presentation} if there exists a finite presentation
\[ G \cong \left< a_1, \dots, a_s \mid R \right> \]
where $R$ is a list of words in $a_1, \dots, a_s$ so that the following hold:
\begin{enumerate}
    \item $\{ a_1, \dots, a_s \} \subset \{ g_1, \dots, g_r \}$ under the isomorphism
    \item each $g_i = h_i a_{j(i)}^{\ell_i} h_i^{-1}$ for some integers $j(i), \ell_i$ and $h_i \in G$ 
    \item each word $w \in R$ is conjugate to one of the following forms
    \begin{enumerate}
        \item the words $w = a_i^{e_1} h(\ul{a}) a_j^{e_2} h(\ul{a})^{-1}$ where $h(\ul{a})$ is any word in the variables $a_1, \dots, a_s$ and $e_1, e_2$ are integers
        \item for any choice $\tilde{h}_i(\ul{a})$ of a word evaluating to $h_i$ (for some choice of $h_i$ as in (2)), the words
        \[ \prod_{i = 1}^r \tilde{h}_i(\ul{a}) a_{j(i)}^{\ell_i} \tilde{h}_i(\ul{a})^{-1} \]
    \end{enumerate}
\end{enumerate}
\end{defn}

\begin{prop} \label{prop:extends_good_presentation_normally_generated}
Suppose that $(g_1, \dots, g_r)$ are spherical generators of $G$ extending a good presentation. Then $E$ is normally generated by torsion, hence $\pi_1(X) = 1$ for the corresponding surface.  
\end{prop}

\begin{proof}
We need to show that every element of $E$ is conjugate to one that can be written as a product of elements of finite order. First, we reduce to the subgroup $\{(1, k) \}_{k \in \ker{\varphi}}$. For any $(w_1, w_2) \in E$, towards induction on the word length of $w_1$, write $w_1 
= c_1 w_1'$ then $(w_1, w_2) = (c_1, c_1) \cdot (w_1', c_1^{-1} w_2)$ with both terms inside $E$. Since each $c_i \in \TT$ has finite order, the first term has finite order. Repeating this process we reduce to terms of the form $(1, k)$ with $k \in \ker{\varphi}$.
\par 
Now we use the good presentation extending $g_1, \dots, g_r$. Since $\varphi : c_i \mapsto g_i$, the kernel is normally generated\footnote{One might worry that $\ker{\varphi}$ being normally generated in $\TT$ by some set is not the same as saying that $\{(1,k)\}_{k \in \ker{\varphi}}$ is normally generated in $E$ by the same set. There is no issue because if $k = w k' w^{-1}$ for $w \in \TT$ then $(1, k) = (w, w) \cdot (1, k') \cdot (w^{-1}, w^{-1})$ and each term is in $E$.} by the relations between the $g_i$ in $G$. A generating set is $R$ plus the relations $g_i = \tilde{h}_i(\ul{a}) a_{j(i)}^{\ell_i} \tilde{h}_i(\ul{a})^{-1}$. It suffices to show that each of these relations is in the normal subgroup generated by elements of finite order in $E$. The relations in (2) give elements $(c_i, \tilde{h}_i(\ul{c}) c_{j(i)}^{\ell_i} \tilde{h}_i(\ul{c})^{-1}) \in E$. But this element has finite order $\lcm(m_i, m_{j(i)})$ since $c_i$ has order $m_i$ and the second term is conjugate to $c_{j(i)}$ which has order $m_{j(i)}$. Similarly, the words of type (a) give elements of finite order $(c_i^{-e_1}, h(\ul{c}) c_j^{e_2} h(\ul{c})^{-1})$. Finally, the words of type (b) use the relation $c_1 \cdots c_r = 1$. Using the elements of finite order 
\[ (c_i, \tilde{h}_i(\ul{c}) c_{j(i)}^{\ell_i} \tilde{h}_i(\ul{c})^{-1}) \in E \] 
and taking the product over $i$ we get $(1, w) \in E$ is a product of finite-order elements. 
\end{proof}

\begin{example}
Let $G = C_n$ be a cyclic group and $(g_1, \dots, g_r)$ a list of spherical generators and suppose that $g_1$ generates. Then taking $a_1 = g_1$ and $R = \{ a_1^n \}$ gives a good presentation since each $g_i = a_1^{e_i}$ and the relation $a_1^n$ is of the form (a).    
\end{example}

The following example will be used in the ``$D_7$ case'' of Theorem~\ref{intro:thm:example}.

\begin{example} \label{example:odd_dihedral}
Let $G = D_{n} = \left< r,s \mid r^n, s^2, (rs)^2 \right>$ be the dihedral group of order $2n$ with $n$ odd. Suppose we take a list of spherical generators $g_1, \dots, g_r$ so that $r$ and $s$ appear among the $g_i$. Up to Hurwitz moves and relabeling $r$ and $s$ by an automorphism, we may assume $g_1 = r$ and $g_2 = s$. Then $a_1 = r$ and $a_2 = s$ makes the above presentation good. 
Indeed, because $n$ is odd, every element is conjugate to $s$ or a power of $r$, so (2) is satisfied. Finally, the relations are all of the form $a_i^{e_i} = 1$, which are of type (a), except the relation $rsrs = 1$ which is actually a conjugation relation $r s r s^{-1}$ since $s = s^{-1}$ and hence also of type (a). 
\end{example}

Finally, we give an example showing the use of the ``product relation'' in a good presentation.

\begin{example} \label{example:reflection_dihedral}
Let $G = D_{n}$ using the Coxeter presentation 
\[ D_{n} = \left< f_1, f_2 \mid f_1^2, f_2^2, (f_1 f_2)^n \right>. \] 
Consider the set of spherical generators $g_1, \dots, g_{2n}$ where 
\[ g_i = 
\begin{cases}
f_1 & i \text{ odd}
\\
f_2 & i \text{ even}
\end{cases}
\]
Then we have $g_i = f_{j(i)}$ for some $j(i)$, so (2) is satisfied. Furthermore, the relations in the presentation are either $f_i^2 = 1$, which is type (a), or $(f_1 f_2)^n = 1$, which is type (b) using $\wt{h}_1 = 1$ and $\ell_i = 1$.
\end{example}

\subsection{Computing the basket of singularities} \label{section:basket}

Here we describe an algorithm for computing the basket of singularities $\cB$ of a surface $Y = (C_1 \times C_2) / G$ where the curves $C_i$ are presented as $G$-covers $f_i : C_i \to \P^1$ prescribed by spherical generators $g_1, \dots, g_r$ and $h_1, \dots, h_s$. This is described in \S1.2 of \cite{bauer:classification}; in particular, see \cite[Proposition 1.17]{bauer:classification} for additional details. In \S\ref{section:examples}, we used the \textsc{Magma} implementation of this algorithm included in \cite{bauer:classification} to actually perform calculations of the basket and its invariants. 
\par 
The fixed points of the actions $G \acts C_i$ lie in the fibers of $f_i : C_i \to \P^1$ over the ramification points $p_1, \dots, p_r$ and $p_1', \dots, p_s'$ respectively. The fiber $(f_1 \times f_2)^{-1}(p_i \times p_j')$ is identified with $G / \left< g_i \right> \times G / \left< h_j \right>$. Therefore, the fiber of the stack $[C_1 \times C_2 / G] \to C_1/G \times C_2/G = \P^1 \times \P^1$ over $(p_i, p_j')$ is identified with the groupoid
\[ [(G/\left< g_i \right> \times G/\left<h_j\right>) / G] \cong \bigsqcup_{r \in G / \left<g_i, h_j \right>} B (\left< g_i \right> \cap r \left< h_j \right> r^{-1})  \]
so these groups $\left< g_i \right> \cap r \left< h_j \right> r^{-1}$ form the stabilizers of the $[G : \left< g_i, h_j \right>]$ singular points of $Y$ in the fiber over $p_i \times p_j'$. What remains is to compute the type of the singularities corresponding to the above stabilizers. To do this, we need a local model. Let $(x, y) \in C_1 \times C_2$ be the point in the fiber over $p_i \times p_j'$ corresponding to $(1,1) \in G/\left< g_i \right> \times G / \left< h_j \right>$. This means $g_i$ fixes $x$ and $h_j$ fixes $y$. Since the monodromy element $g_i$ acts by lifting the simple clockwise closed loop downstairs and $f_i$, locally near $x$, takes the form  $z \mapsto z^{\ord(g_i)}$ we see that $g_i$ acts by $e^{\frac{2 \pi i}{\ord(g_i)}}$ on $T_x C_1$. Similarly, $h_j$ acts by $e^{\frac{2 \pi i}{\ord(h_j)}}$ on $T_y C_2$. Now the point $(1,r) \in G/\left< g_i \right> \times G / \left< h_j \right>$ corresponding to $(x, r \cdot y) \in C_1 \times C_2$ has stabilizer generated by $g_i^{\gamma}$ where $\gamma$ is the smallest positive integer such that $g_i^{\gamma} \in r \left< h_j \right> r^{-1}$. Write $g_i^{\gamma} = r h_j^{\delta} r^{-1}$ for some $\delta \le \ord(h_j)$. Therefore $g_i^{\gamma}$ acts on $T_x C_1 \oplus T_{r \cdot y} C_2$ via $(e^{\frac{2 \pi i \gamma}{\ord(g_i)}}, e^{\frac{2 \pi i \delta}{\ord(h_j)}})$. The quotient is a type $\frac{1}{n} (1, a)$ singularity where 
\[ n = \frac{\ord(g_i)}{\gamma} \quad \text{and} \quad a = \frac{\ord(g_i) \delta}{\ord(h_j) \gamma} = \frac{\delta}{\gcd(\delta, \ord(h_j))}. \]

\begin{rmk} \label{rmk:baseket_depends_only_conjugacy_multiset}
Notice that the basket of singularities depends only on the two multisets of conjugacy classes $\Cl(g_1), \dots, \Cl(g_r)$ and $\Cl(h_1), \cdots \Cl(h_s)$, not on the individual elements of either list of spherical generators or the order in which they appear in the list. This is because the singularities over $p_i \times p_j'$ are determined by the set of pairs $(\gamma, \delta)$ appearing in relations of the form $g_i^{\gamma} = r h_j^{\delta} r^{-1}$ as we vary over $r$. Conjugating $g_i$ and $h_j$ by elements $\ell$ and $\ell'$ respectively simply permutes the value of $r$ for each pair $(\gamma, \delta)$ since 
\[ (\ell g_i \ell^{-1})^{\gamma} = \ell g_i^{\gamma} \ell^{-1} = \ell r h_j^{\delta} r^{-1} \ell^{-1} = (\ell r \ell'^{-1}) (\ell' h_j^{\delta} \ell'^{-1}) (\ell r \ell'^{-1})^{-1} \]
so conjugation has the effect of permuting the set of $r$ via $r \mapsto \ell r \ell'^{-1}$ while preserving the multiset of pairs $(\gamma, \delta)$. 
\end{rmk}

\subsection{Supersingularity and \etale cohomology}

One advantage of product-quotient surfaces is that their cohomology is nicely controlled by the cohomology of the curves. In particular, working over an algebraically closed field of characteristic $p > 0$, the following lemma gives a very simple criterion for (Shioda) supersingularity.

\begin{lemma} \label{lem:supersingular_domination}
Let $f : X \rat Y$ be a dominant rational map of smooth proper surfaces. If $X$ is (Shioda) supersingular then so is $Y$.
\end{lemma}

\begin{proof}
Resolving the map $\wt{f} : \wt{X} \to Y$, we get an inclusion $H^i_{\et}(Y, \Q_{\ell}) \to H^i_{\et}(\wt{X}, \Q_{\ell})$. Moreover, the pullback map $H^i_{\et}(X, \Q_{\ell}) \to H^i_{\et}(\wt{X}, \Q_{\ell})$ is an isomorphism for $i \neq 2$ and for $i = 2$ adds a copy of $\Q_{\ell}(-1)$ for each blowup. Hence, assuming $X$ is supersingular, we conclude the same for $\wt{X}$ and also $Y$. If $X$ is Shioda supersingular, then $H^2_{\et}(\wt{X}, \Q_{\ell}(1))$ is spanned by algebraic cycles so, using the push-pull formula, $H^2_{\et}(Y, \Q_{\ell}(1))$ is spanned by the pushforward of the cycle classes on $\wt{X}$ spanning $\im{(H^2_{\et}(Y, \Q_{\ell}) \to H^2_{\et}(\wt{X}, \Q_{\ell}))}$.
\end{proof}

Returning to the case that $X$ is a product-quotient surface defined over a number field $K$, the dominant rational map $C_1 \times C_2 \rat X$ immediately provides the following corollary.

\begin{cor} \label{cor:curves_supersingular}
If $C_1$ and $C_2$ have good supersingular reduction at some prime $\p \subset K$ and $X$ has good reduction at $\p$, then $X_{\p}$ is Shioda supersingular.
\end{cor}

\begin{proof}
The Tate conjecture is known for products of curves by \cite[Theorem 5.2(a)]{Tate}. Hence, by the K\"{u}nneth formula, $C_1 \times C_2$ is Shioda supersingular if $C_1$ and $C_2$ are each supersingular. Then we conclude by Lemma~\ref{lem:supersingular_domination}. 
\end{proof}

In fact, we know the Tate conjecture for $X$ at all primes of good reduction by \cite[Theorem 5.2(b)]{Tate}. 

\section{Existence of symmetric differential forms}

In this section, we largely follow the strategy used in \cite{lang_product_quotient} to construct diagonal symmetric forms which the authors use to prove quasi-hyperbolicity results in characteristic zero. 
\par 
In a different direction, \cite{RR13} developed a general method based on orbifold techniques for producing symmetric forms. In particular, they proved the following.

\begin{theorem} \cite[Theorem 1]{RR13}
Let $\pi : X \to Y$ be the minimal resolution of a surface $Y$ with ADE singularities and suppose that
\[ \tfrac{1}{2} (K_X^2 - c_2(X)) + \sum_{n \ge 1} \left[ (n+1) (a_n + d_n + e_n) - \frac{a_n}{n+1} - \frac{d_n}{4(n-2)} \right] - \frac{e_6}{24} - \frac{e_7}{48} - \frac{e_8}{120} > 0 \]
where $a_n, d_n$, and $e_n$ are the numbers of $A_n, D_n$, and $E_n$ singularities respectively. Then $\Omega_X$ is big. In particular, $X$ carries a nonzero symmetric form.
\end{theorem}

These results are not adequate for our purposes for two reasons. First, their method relies on $Y$ having only canonical singularities (so that $\pi$ is crepant), but the surfaces we will analyze have, for example, $\tfrac{1}{3}(1,1)$ singularities which are not canonical. Second, the symmetric forms produced this way need not be diagonal. 

\subsection{Constructing a bundle of diagonal forms}

Let $E$ be the exceptional divisor for $\pi : X \to Y$. In this section, we construct a map $\gamma : \omega_X(-E) \to \Sym^2 \Omega_X$ so that $q^* \im{\gamma}$ is diagonal. To construct this map, we need a few lemmas.

\begin{lemma} \label{lem:klt_push_pull_canonical}
Let $Y$ be a variety with klt singularities and $\pi : X \to Y$ a resolution of singularities. Then $\im( (\pi^* \pi_* \omega_X)^{\vee \vee} \to \omega_X) = \omega_X(\floor{- \sum a_i E_i})$ where the $E_i$ are the exceptional divisors of $\pi$ and the $a_i$ are the associated discrepancies.
\end{lemma}

\begin{proof}
By \cite[Theorem 3 + Theorem 4]{kovacs_rational}, $Y$ has rational singularities and $\pi_* \omega_X = \omega_Y$ which is a reflexive sheaf. Then we use the defining relation
\[ K_X = \pi^* K_Y + \sum a_i E_i \]
and the fact that  $\pi^{[*]} \struct{Y}(D) = \struct{X}(\floor{\pi^* D})$ where $\pi^{[*]}$ is the reflexive pullback, $D$ is a Weil $\Q$-Cartier divisor, and $f^* D$ is the pullback as a $\Q$-Cartier divisor. To prove this, let $\ell D$ be Cartier. Then any local section $s$ of $\struct{Y}(D)$ has $s^{[\ell]}$ a local section of $\struct{Y}(\ell D)$ and hence $\pi^* s^{[\ell]} = (\pi^{[*]} s)^\ell$ is a section of $\pi^* \struct{Y}(\ell D)$ so $\ell \div(\pi^{[*]} s) + f^* \ell D \ge 0$. Likewise, if $\ord_E(\pi^{[*]} s) \ge \ord_E(\pi^* D) + 1$ then dividing $s$ by a local parameter for $E$, we can ensure that $\struct{Y}(D)$ has local sections generating $\struct{X}(\floor{\pi^* D})$.
\end{proof}

In particular, if $\pi : X \to Y$ is the minimal resolution of a surface with klt singularities, then the discrepancies satisfy $-1 < a_i \le 0$ so we conclude that $(\pi^* \pi_* \omega_X)^{\vee \vee} \to \omega_X$ is an isomorphism. We also need Miyaoka's extension theorem \cite{miyaoka}, which, in the modern version due to Greb, Kebekus, and \Kovacs, says that reflexive tensor-forms on $Y$ extend to logarithmic tensor-forms on $X$. We use the notation $(-)^{[n]}$ and $\Sym^{[n]}(-)$ for the reflexive tensor power and reflexive symmetric power, meaning the reflexive hull of the usual tensor power or symmetric power respectively.

\begin{theorem} \cite[Corollary 3.2]{GKK10} \label{thm:miyaoka_extension}
Let $\pi : X \to Y$ is the resolution of a variety with isolated finite quotient singularities and $E$ be the reduced exceptional divisor. Then 
\[ \Sym^{[m]}\Omega_Y = \pi_* \Sym^m\Omega_X(\log{E}). \]
In particular, $\pi_* \Sym^{m} \Omega_X(\log{E})$ is reflexive and there is a pullback map
\[ H^0(Y, \Sym^{[m]}\Omega_Y) = H^0(Y^{\circ}, \Sym^m \Omega_{Y^\circ}) \to H^0(X, \Sym^{m} \Omega_X(\log{E})) \]
\end{theorem}

To construct a map $\gamma$ we first build a map $\omega_X \to \Sym^2 \Omega_X(\log{E})$ and prove that, when restricted to $\omega_X(-E)$ it factors through $\Sym^2 \Omega_X$. This map is built as follows:
\[ \pi_* \omega_X \iso \omega_Y \to (q_* \omega_{C_1 \times C_2})^G \to (q_* \Sym^2 \Omega_{C_1 \times C_2})^G = \Sym^{[2]} \Omega_Y = \pi_* \Sym^{2} \Omega_X(\log{E}) \]
where the last equality uses the logarithmic extension theorem~\ref{thm:miyaoka_extension}.  Note that the diagonal inclusion map $\omega_{C_1 \times C_2} \embed \Sym^{2} \Omega_{C_1 \times C_2}$ is $G$-equivariant because $G$ acts diagonally on $C_1 \times C_2$. Adjunction produces a map
\[ \alpha : \pi^* \pi_* \omega_X \to \Sym^{2} \Omega_X(\log{E}) \]
but $(\pi^* \pi_* \omega_X)^{\vee \vee} \iso \omega_X$ by the remark after Lemma~\ref{lem:klt_push_pull_canonical} and $\Sym^{2} \Omega_X(\log{E})$ is reflexive, in fact locally free, so $\alpha$ factors through the reflexive hull of the source, giving the desired map
\[ \alpha' : \omega_X \to \Sym^{2} \Omega_X(\log{E}). \]
We want to show that $\alpha'$ lands in honest symmetric forms when restricted to $\omega_X(-E)$. This holds immediately when restricting to $\omega_X(-2E)$ but the refined result is necessary for the numerics of our examples. This is achieved via the subsequent lemma that follows the philosophy of a result due to Greb, Kebekus, \Kovacs, and Peternell \cite{GKKP} establishing the existence of pullback maps $\pi^* : H^0(Y, \Omega_Y^{[p]}) \to H^0(X, \Omega_X^p)$ for a resolution of klt singularities $\pi : X \to Y$. In view of the logarithmic extension theorem~\ref{thm:miyaoka_extension}, this result says exactly that the extended logarithmic form always has ``one fewer pole than expected'' along $E$. Likewise, we conclude a similar ``one fewer pole'' statement for symmetric forms.

\begin{lemma} \label{lem:pole_order}
Let $Y$ be a surface with rational singularities and $\pi : X \to Y$ the minimal resolution with exceptional divisor $E$. Then for any integer $m > 0$,
\[ \pi_* \left( \big[\Sym^m \Omega_X(\log{E}) \big](-E) \right) = \Sym^{[m]} \Omega_Y. \]
\end{lemma}

\begin{proof}
Since these sheaves agree on the smooth locus $Y^{\circ}$, it suffices to prove that any reflexive symmetric form on $Y$ extends to a section of $\big[\Sym^m \Omega_X(\log{E}) \big](-E)$. Using the logarithmic extension theorem~\ref{thm:miyaoka_extension},
there is a diagram
\begin{center}
\begin{tikzcd}
\pi_* \Sym^m \Omega_X(\log{E}) \arrow[r, equals] & \Sym^{[m]} \Omega_Y
\\
\pi_* \left( \big[\Sym^m \Omega_X(\log{E}) \big](-E) \right) \arrow[u, hook] \arrow[ru]
\end{tikzcd}
\end{center}
and we wish to prove the upward inclusion is surjective.
This is local on the normal surface $Y$, so we may shrink so that $Y$ is affine and contains a single isolated singularity. Then we claim 
\[ H^0(X, \big[\Sym^m \Omega_X(\log{E}) \big](-E)) \to H^0(X, \Sym^m \Omega_X(\log{E})) \]
is surjective. It would suffice to establish the following vanishing result:
\[ H^0(E, \Sym^m \Omega_X(\log{E})|_E) = 0. \]
In degree $m = 1$, this is essentially a consequence of the Hodge index theorem (see \cite[\S5]{local_vanishing}). It can also be seen from results in mixed Hodge theory, specifically Steenbrink vanishing: $R^1 \pi_* \Omega_X(\log{E})(-E) = 0$ \cite[Theorem~14.1]{GKKP}, and local vanishing: $R^1 \pi_* \Omega_X(\log{E}) = 0$ \cite[Theorem~B]{local_vanishing}. Indeed, duality gives
\[ \Rbf \pi_* \Omega_X(\log{E})(-E) = \RHom_{\cO_X}(\Rbf \pi_* \Omega_X(\log{E}), \omega_Y). \]
Hence, $\pi_* \Omega_X(\log{E})(-E)$ is reflexive because $\Rbf \pi_* \Omega_X(\log{E}) = \pi_* \Omega_X(\log{E})$ by local vanishing and $\omega_Y$ is reflexive. Then the long exact sequence
\[ H^0(X, \Omega_X(\log{E})(-E)) \iso H^0(X, \Omega_X(\log{E})) \to H^0(E, \Omega_X(\log{E})|_E) \to H^1(X, \Omega_X(\log{E})(-E)) \]
proves $H^0(E, \Omega_X(\log{E})|_E) = 0$ from Steenbrink vanishing.
\par 
Now we use the $m = 1$ case to prove the result for general $m$. The exceptional divisor $E$ is a tree of smooth rational curves $E_i$. On each component there is a sequence
\[ 0 \to \Omega_{E_i}(E' \cdot E_i) \to \Omega_X(\log{E})|_{E_i} \to \cO_{E_i} \to 0  \]
where $E' = E - E_i$. Let $\E_i$ denote the vector bundle $\Omega_X(\log{E})|_{E_i}$ on $E_i$. Note that $\E_i$ is a sum of line bundles $\cO_{\P^1}(a)$ over twists $a$ that are all $\ge 0$ or all $\le 0$. Therefore, $\E_i$ has the remarkable property that
\[ \Sym^m H^0(E_i, \E_i) \to H^0(E_i, \Sym^m \E_i) \]
is surjective for all $m$. Let $W_{ij}$ be the fiber of $\Omega_X(\log{E})$ at $E_i \cap E_j$. Then $H^0(E, \Omega_X(\log{E})|_E)$ is the equalizer of the two maps
\[ \bigoplus_{i} H^0(E_i, \E_i) \rightrightarrows \bigoplus_{i < j} W_{ij} \]
and likewise $H^0(E, \Sym^m \Omega_X(\log{E})|_E)$ is the equalizer of 
\[ \bigoplus_{i} H^0(E_i, \Sym^m \E_i) \rightrightarrows \bigoplus_{i < j} \Sym^m W_{ij}. \]
The noted surjectivity shows it suffices to prove that 
\[ \bigoplus_{i} \Sym^m H^0(E_i, \E_i) \rightrightarrows \bigoplus_{i < j} \Sym^m W_{ij} \]
has a trivial equalizer where the maps are sums of $\Sym^m(-)$ of the maps at level $m = 1$. Now it becomes entirely a question of linear algebra. Given the data $(V_i, W_k, \varphi_{i,k})$ consisting of a list of vector spaces $V_i$ and $W_k$ with maps $\varphi_{i,k} : V_i \to W_k$ if there are no nonzero lists $v_i \in V_i$ so that $\varphi_{i,k}(v_i)$ are equal over all $i,k$ then the same holds for the data $(\Sym^m V_i, \Sym^m W_k, \Sym^m \varphi_{i,k})$ for essentially the same reason that symmetric powers preserve injectivity.
\end{proof}

Our proof is highly particular to the surface case so it would be interesting to study extensions of this result in higher dimension. 

\subsection{Numerical criteria}

In this section, we prove a purely numerical criterion for the existence on $X$ of (many) diagonal symmetric forms. Using the map $\gamma$ of the previous section, it suffices to show that $\omega_X(-E) = \struct{X}(K_X - E)$ has (many) sections. 

\begin{prop} \label{prop:numerical_volume}
Let $\pi : X \to Y$ be the minimal resolution of a klt surface with exceptional divisor $E$. Then 
\[ (K_X - E)^2 = K_X^2 - 3 \sum_i (-E_i^2 - 2) - 2 \# Y^{\mathrm{sing}} \]
Moreover, if $X$ is of general type and $(K_X - E)^2 > 0$, then $K_X - E$ is big.
\end{prop}

\begin{proof}
Expand
\[ (K_X - E)^2 = K_X^2 - 2 K_X \cdot E + E^2 \]
and write $E = E_1 + \cdots + E_r$. Since the support of $E$ is a forest (a union of trees) consisting of smooth rational curves, we get 
\[ E^2 = \sum_i E_i^2 + 2 \sum_{i > j} E_i \cdot E_j \]
but the intersections $E_i \cdot E_j$ correspond to edges of the dual graph $\Gamma$ of $E$. Since $\Gamma$ is a forest, there is exactly one more node (component of $E$) than edge for each connected component. Hence,
\[ E^2 = \sum_i (E_i^2 + 2) - 2 \# \pi_0(\Gamma). \]
Note also that $\pi_0(\Gamma) = \# Y^{\mathrm{sing}}$ since, by Zariski's main theorem, the exceptional divisor above an isolated singularity is connected. Furthermore, the adjunction formula gives that $(K_X + E_i) \cdot E_i = -2$ so 
\[ K_X \cdot E = - \sum_i (E_i^2 + 2) \]
Putting everything together,
\[ (K_X - E)^2 = K_X^2 + 3 \sum_i (E_i^2 + 2) - 2 \# Y^{\mathrm{sing}}. \]
Now, using Hirzebruch--Riemann-Roch, 
\[ h^0(\struct{X}(m(K_X - E)) + h^2(\struct{X}(m(K_X - E))) \ge \frac{1}{2} m^2 (K_X - E)^2 + O(m) \]
but $h^2(\struct{X}(m(K_X - E))) = h^0(\struct{X}(m E - (m-1) K_X)) = 0$ for $m > 1$ because if $H$ is an ample divisor on $Y$, then $\pi^* H \cdot (m E - (m-1) K_X) = -(m-1) H \cdot K_Y < 0$ so $m E - (m-1) K_X$ cannot be effective. Thus, $K_X - E$ is big as long as $(K_X - E)^2 > 0$. 
\end{proof}

\begin{cor}  \label{cor:numerics_of_each_sing}
If $Y$ has $C$ canonical singularities (type $A,D,E$ singularities) and additionally $N_{n,a}$ non-canonical cyclic quotient singularities of type $\frac{1}{n} (1,a)$, then
\[ (K_X - E)^2 = K_X^2 - 2 C - \sum_{n \ge 1} \sum_{a = 1}^{n-2} I(\tfrac{n}{a}) N_{n,a} \]
where, writing $\frac{n}{a} = \dbrac{b_1, \dots, b_{\ell}}$, we define:
\[ D(\tfrac{n}{a}) := 3 \sum_{i=1}^{\ell} (b_i - 2) + 2 \]
In the notation of \cite{bauer:classification}, $D(\tfrac{n}{a}) = 3 k(\tfrac{n}{a}) - 3 \frac{2 + a + a'}{n} + 8$ where $0 \le a' < n$ is the multiplicative inverse of $a$ in $(\Z / n \Z)^{\times}$. 
\end{cor}

\begin{cor} \label{cor:numerical_criterion}
For a product-quotient surface,
\[ (K_X - E)^2 = \frac{8(g_1 - 1)(g_2 - 1)}{\# G} - k(\mathcal{B}) - D(\mathcal{B}) \]
If $X$ is of general type and this number is positive, then $\omega_X(-E)$ is big and $X$ carries a diagonal symmetric form.
\end{cor}

\begin{proof}
Combine Corollary~\ref{cor:numerics_of_each_sing} with Lemma~\ref{lem:pole_order} and Proposition~\ref{prop:chern_numerics}.
\end{proof}

When $Y$ has only ADE singularities (meaning it has only $A_n$ singularities since all the singularities of $Y$ are cyclic quotient), it is interesting to compare this estimate to the estimate of \cite{RR13}. Recall that an $A_n$ singularity is type $\frac{1}{n+1}(1, n)$, so in this case,
\[ k(\cB) = 0, \quad \quad e(\cB) = \sum_{n \ge 1} \frac{n(n+2)}{n+1} a_n, \quad \quad B(\cB) = 2 e(\cB)  \]
where $a_n$ is the number of $A_n$ singularities on $Y$. Therefore \cite{RR13} produces symmetric forms when
\[ 2(K_X^2 - c_2) + \sum_{n \ge 1} \frac{n(n+2)}{n+1} a_n = \frac{8(g_1 - 1)(g_2 - 1)}{\# G} - \sum_{n \ge 1} \frac{n(n+2)}{n+1} a_n  \]
is positive. 
On the other hand, our method produces diagonal symmetric forms when
\[ \frac{4(g_1 - 1)(g_2 - 1)}{\# G} - \sum_{n \ge 1} a_n \]
is positive. A more recent computation in \cite{bruin:an-singularities} improves the result of \cite{RR13} for $A_n$ singularities. They show that $\Omega_X$ is big if 
\[ K_X^2 - c_2(X) + \sum_{n \ge 1} \frac{n(n+2)}{n+1} a_n - (\tfrac{4}{3} \pi^2 - 12) \sum_{n \ge 1} a_n > 0 \]
which, in our case, is the inequality
\[ \frac{4(g_1 - 1)(g_2 - 1)}{\# G} - (\tfrac{4}{3} \pi^2 - 12) \sum_{n \ge 1} a_n > 0. \]
Note that $\tfrac{4}{3} \pi^2 - 12 \approx 1.2 > 1$ so our inequality is more easily satisfied. However, for the particular case of only $A_1$ singularities, \cite{bruin:an-singularities} gives a better inequality:
\[ \frac{4(g_1 - 1)(g_2 - 1)}{\# G} - \frac{66}{108} a_1 > 0. \]

\section{Constructions and examples} \label{section:examples}

In this section, we prove Theorem~\ref{intro:thm:example} by checking that the example described in the introduction satisfies the necessary conditions to apply Theorem~\ref{intro:thm:criterion}. 

\subsection{Hurwitz curves of genus $14$}

A \textit{Hurwitz curve} is any hyperbolic smooth projective curve $C$ over $\CC$ realizing equality in the Hurwitz bound
\[ \# \Aut(C) = 84(g(C) - 1). \]
Any such curve is a \Belyi curve and hence is defined over $\ol{\Q}$. As a complex curve, $C$ can be realized as a quotient of the upper half-plane $\HH$ by a Fuchsian group of finite index in the von Dyck group $\TT(2,3,7)$. This result arises from $2,3,7$ being the unique (up to reordering) sequence of positive integers $e_1, \dots, e_r$ minimizing the sum 
\[ \sum_i \left( 1 - \frac{1}{e_i} \right) \]
over all sequences such that the sum is $ > 2$ and obtaining the minimum value of $2 + \frac{1}{42}$. This means we can always present a Hurwitz curve as a \Belyi map $C \to \P^1$ ramified over $0,1,\infty$ to order $2,3,7$. Because the ramification orders are distinct, given $C$, this map is unique up to isomorphism. Therefore, the data to specify a Hurwitz curve with automorphism group $G$ is a list of spherical generators $g_1, g_2, g_3 \in G$ of orders $2,3,7$ up to simultaneous conjugacy (which allows for isomorphisms not respecting a marked base-point in $X$). A group admitting such generators is called a \textit{Hurwitz group}. Actually, what we have specified is a $G$-\Belyi cover of $\P^1$ meaning along with a particular choice of isomorphism $G \iso \Aut(C)$. To remove this additional choice, we can consider tuples $(g_1, g_2, g_3)$ equivalent up to overall automorphism of $G$ rather than only inner automorphisms.  When $G$ has trivial center, the choice of an isomorphism $\varphi G \to \Aut(C)$ rigidifies the pair $(C, \varphi)$, giving a point on a certain Hurwitz scheme.
\par 
For $g = 14$, there exist exactly three isomorphism classes of Hurwitz curves which are exchanged by $\Gal(\ol{\Q} / \Q)$. In fact, as we will see in Appendix~\ref{appendix:arithmetic}, these three curves arise from a curve $C / K_+$ under the three complex embeddings $K_+ \embed \CC$ of the field $K_+ = \Q(\eta)$ for $\eta = \zeta_7 + \zeta_7^{-1}$. These curves have $\Aut(C) \cong \PSL_2(\FF_{13})$ which is a simple group of order $1092$. However, there are $6$ isomorphism classes of $\PSL_2(\FF_{13})$-\Belyi $(2,3,7)$-ramified covers of $\P^1$ corresponding to the $6$ equivalence classes of $(2,3,7)$-spherical generators up to simultaneous conjugation. Under the unique outer automorphism of $\PSL_2(\FF_{13})$, these $6$ classes form $3$ orbits  corresponding to the three Hurwitz curves of genus $14$, each presented as a $\PSL_2(\FF_{13})$-cover of $\P^1$ in two distinct ways up to $\PSL_2(\FF_{13})$-equivariant isomorphism. In fact, in  Appendix~\ref{appendix:computations}, we show that the stack $\mathcal{H}_{14, \PSL_2(\FF_{13})}$ of pairs $(C, \varphi)$ consisting of a curve $C$ of genus $14$ and an isomorphism $\varphi : \PSL_2(\FF_{13}) \to \Aut(C)$ is, over $\Q$, isomorphic to $\Spec{K}$ where $K = K_+(\sqrt{-3\eta -2})$ is the ray class field of a prime in $\cO_{K_+}$ lying over $13$. This is the degree $6$ number field recorded as \cite[\href{https://www.lmfdb.org/NumberField/6.2.31213.1}{Number field 6.2.31213.1}]{lmfdb} in the LMFDB, and the various embeddings $K \embed \CC$ correspond to the $6$ isomorphism classes of $\PSL_2(\FF_{13})$-\Belyi curves considered above. Note that $K_+ \subset K$ is the subfield consisting of totally real elements, justifying the notation. 
\par 
For example, we can choose a particular set of spherical generators $(g_1, g_2, g_3)$ defining a surjection $\TT(2,3,7) \onto \PSL_2(\FF_{13})$,
\[ g_1 = \begin{pmatrix}
    5 & 3
    \\
    0 & 8
\end{pmatrix}
\quad \quad 
g_2 = \begin{pmatrix}
-1 & 3
\\
4 & 0
\end{pmatrix}
\quad \quad 
g_3 = \begin{pmatrix}
0 & 2
\\
6 & 6
\end{pmatrix}.
\]
This corresponds to the ``refined passport'' \cite[\href{https://www.lmfdb.org/HigherGenus/C/Aut/14.1092-25.0.2-3-7.1}{Higher Genus Family: 14.1092-25.0.2-3-7.1}]{lmfdb} with conjugacy classes $(2,3,5)$ using the labeling scheme of the LMFDB. We call this curve $C^{\nu_0}$, or to be precise $(C, \varphi)^{\nu_0}$, with its $\PSL_2(\FF_{13})$-action, for a certain fixed embedding $\nu_0 : K \embed \CC$. Because in Theorem~\ref{thm:obstruction} it is also necessary to understand partial Frobenius twists, it is insufficient to consider only $C^{\nu_0}$. We also look at the $6$ Galois conjugates corresponding to different embeddings $\nu : K \embed \CC$. These are permuted by $\Gal(K'/\Q) \cong C_2 \times A_4$ where $K'$ is the Galois closure of $K$. Generators for all $6$ isomorphism classes of $\PSL_2(\FF_{13})$-\Belyi covers can be found in \cite[\href{https://www.lmfdb.org/HigherGenus/C/Aut/14.1092-25.0.2-3-7}{Higher Genus Family: 14.1092-25.0.2-3-7}]{lmfdb}.

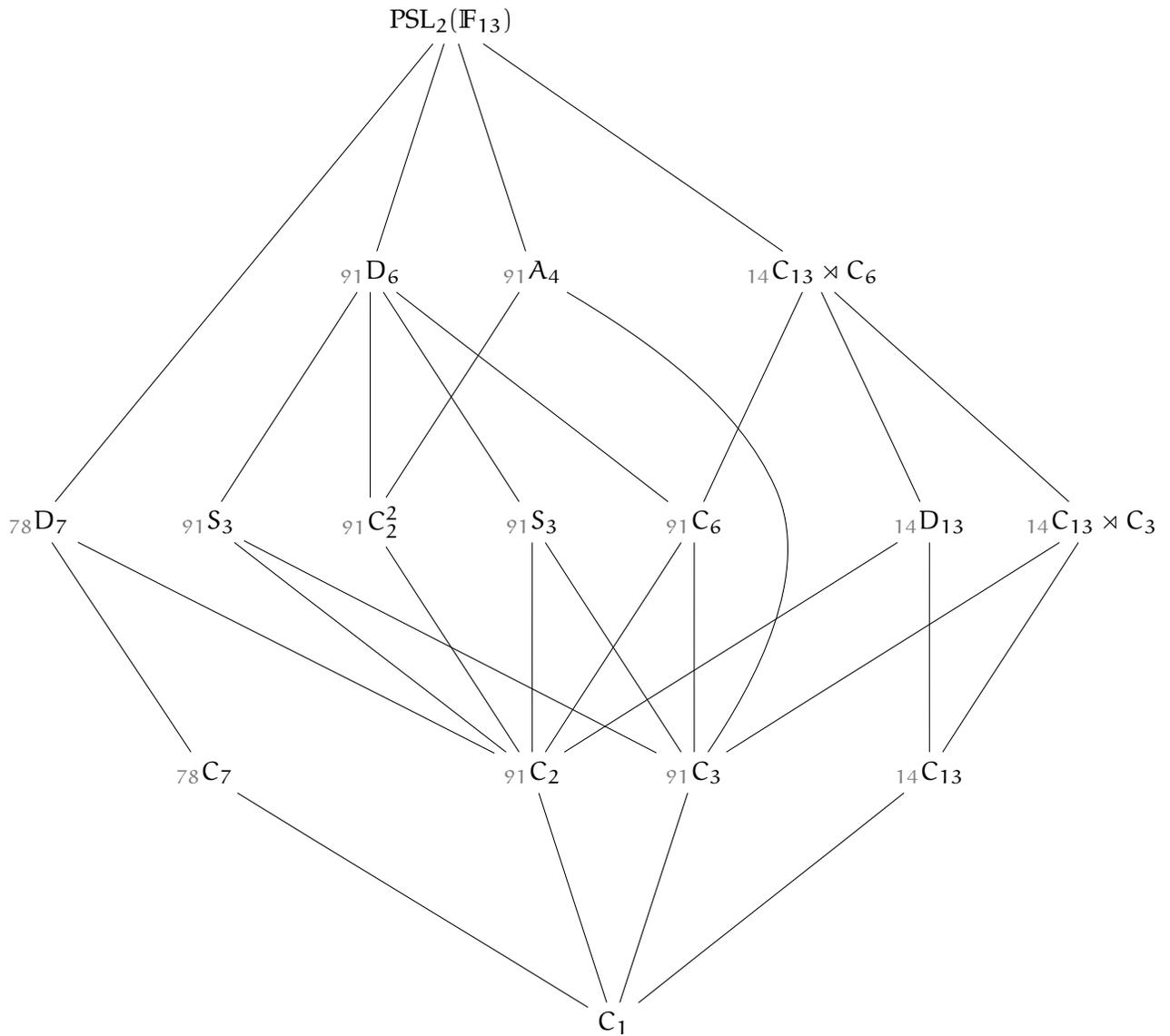
\begin{figure}
\makebox[\textwidth]{
\begin{tikzpicture}[scale = 0.8,line join=bevel]
\node (1) at (334.5bp,18.0bp) [draw,draw=none] {$C_1$};
  \node (2) at (292.5bp,148.0bp) [draw,draw=none] {${}_{\textcolor{gray}{91}} C_2$};
  \node (3) at (376.5bp,148.0bp) [draw,draw=none] {${}_{\textcolor{gray}{91}} C_3$};
  \node (4) at (122.5bp,148.0bp) [draw,draw=none] {${}_{\textcolor{gray}{78}} C_7$};
  \node (5) at (498.5bp,148.0bp) [draw,draw=none] {${}_{\textcolor{gray}{14}} C_{13}$};
  \node (6) at (208.5bp,278.0bp) [draw,draw=none] {${}_{\textcolor{gray}{91}} C_2^2$};
  \node (7) at (124.5bp,278.0bp) [draw,draw=none] {${}_{\textcolor{gray}{91}} S_3$};
  \node (8) at (376.5bp,278.0bp) [draw,draw=none] {${}_{\textcolor{gray}{91}} C_6$};
  \node (9) at (292.5bp,278.0bp) [draw,draw=none] {${}_{\textcolor{gray}{91}} S_3$};
  \node (10) at (36.5bp,278.0bp) [draw,draw=none] {${}_{\textcolor{gray}{78}} D_{7}$};
  \node (11) at (498.5bp,278.0bp) [draw,draw=none] {${}_{\textcolor{gray}{14}} D_{13} $};
  \node (12) at (582.5bp,278.0bp) [draw,draw=none] {${}_{\textcolor{gray}{14}} C_{13}\rtimes C_3$};
  \node (14) at (292.5bp,408.0bp) [draw,draw=none] {${}_{\textcolor{gray}{91}} A_4$};
  \node (13) at (208.5bp,408.0bp) [draw,draw=none] {${}_{\textcolor{gray}{91}} D_{6}$};
  \node (15) at (437.5bp,408.0bp) [draw,draw=none] {${}_{\textcolor{gray}{14}} C_{13}\rtimes C_6$};
  \node (16) at (250.5bp,538.0bp) [draw,draw=none] {$\mathrm{PSL}_2(\mathbb{F}_{13})$};
  \draw [] (1) ..controls (320.82bp,60.692bp) and (306.25bp,105.1bp)  .. (2);
  \draw [] (1) ..controls (348.18bp,60.692bp) and (362.75bp,105.1bp)  .. (3);
  \draw [] (1) ..controls (265.33bp,60.761bp) and (191.49bp,105.35bp)  .. (4);
  \draw [] (1) ..controls (388.01bp,60.761bp) and (445.13bp,105.35bp)  .. (5);
  \draw [] (2) ..controls (265.14bp,190.69bp) and (236.0bp,235.1bp)  .. (6);
  \draw [] (2) ..controls (237.69bp,190.76bp) and (179.17bp,235.35bp)  .. (7);
  \draw [] (2) ..controls (319.86bp,190.69bp) and (349.0bp,235.1bp)  .. (8);
  \draw [] (2) ..controls (292.5bp,190.69bp) and (292.5bp,235.1bp)  .. (9);
  \draw [] (2) ..controls (210.59bp,189.95bp) and (120.23bp,235.13bp)  .. (10);
  \draw [] (2) ..controls (359.71bp,190.76bp) and (431.46bp,235.35bp)  .. (11);
  \draw [] (3) ..controls (295.05bp,190.37bp) and (206.07bp,235.57bp)  .. (7);
  \draw [] (3) ..controls (376.5bp,190.69bp) and (376.5bp,235.1bp)  .. (8);
  \draw [] (3) ..controls (349.14bp,190.69bp) and (320.0bp,235.1bp)  .. (9);
  \draw [] (3) ..controls (443.71bp,190.76bp) and (515.46bp,235.35bp)  .. (12);
  \draw [] (3) ..controls (407.78bp,194.41bp) and (438.71bp,251.25bp)  .. (418.5bp,296.0bp) .. controls (399.24bp,338.64bp) and (354.6bp,371.45bp)  .. (14);
  \draw [] (4) ..controls (94.488bp,190.69bp) and (64.652bp,235.1bp)  .. (10);
  \draw [] (5) ..controls (498.5bp,190.69bp) and (498.5bp,235.1bp)  .. (11);
  \draw [] (5) ..controls (525.86bp,190.69bp) and (555.0bp,235.1bp)  .. (12);
  \draw [] (6) ..controls (208.5bp,320.69bp) and (208.5bp,365.1bp)  .. (13);
  \draw [] (6) ..controls (235.86bp,320.69bp) and (265.0bp,365.1bp)  .. (14);
  \draw [] (7) ..controls (151.86bp,320.69bp) and (181.0bp,365.1bp)  .. (13);
  \draw [] (8) ..controls (321.69bp,320.76bp) and (263.17bp,365.35bp)  .. (13);
  \draw [] (8) ..controls (396.37bp,320.69bp) and (417.53bp,365.1bp)  .. (15);
  \draw [] (9) ..controls (265.14bp,320.69bp) and (236.0bp,365.1bp)  .. (13);
  \draw [] (10) ..controls (89.858bp,343.33bp) and (197.28bp,472.83bp)  .. (16);
  \draw [] (11) ..controls (478.63bp,320.69bp) and (457.47bp,365.1bp)  .. (15);
  \draw [] (12) ..controls (535.27bp,320.69bp) and (484.97bp,365.1bp)  .. (15);
  \draw [] (13) ..controls (222.18bp,450.69bp) and (236.75bp,495.1bp)  .. (16);
  \draw [] (14) ..controls (278.82bp,450.69bp) and (264.25bp,495.1bp)  .. (16);
  \draw [] (15) ..controls (376.49bp,450.76bp) and (311.35bp,495.35bp)  .. (16);
\end{tikzpicture}}
    \caption{Subgroup lattice of $\PSL_2(\FF_{13})$ up to conjugacy. The small numbers in gray represent the number of conjugate copies of a subgroup.}
    \label{fig:subgroups}
\end{figure}

The following computations will be based on the subgroups of $\PSL_2(\FF_{13})$ and the curves appearing as intermediate quotients of $C$. The subgroup lattice of $\PSL_2(\FF_{13})$ is recorded in Figure~\ref{fig:subgroups}. $\mathrm{Out}(\PSL_2(\FF_{13})) \cong C_2$ acts trivially on this diagram except that it exchanges the two copies of $S_3$. For each subgroup $H \subset G := \PSL_2(\FF_{13})$ we can compute the intermediate curve $C \to C / H \to \P^1$ as follows. The curve $C / H \to \P^1$ is a \Belyi curve whose monodromy is given by the action of $g_1, g_2, g_3$ on the fiber $G/H$. The cycle decomposition of this action determines the ramification structure over $0,1,\infty$. Using Riemann--Hurwitz, we get a diagram of quotient curves whose genera are labeled in Figure~\ref{fig:subgroups_genera}. Notice that the genus $1$ quotients are via the two copies of $S_3$ and are hence exchanged under the outer automorphism. This suggests, as we will see is indeed the case in Appendix~\ref{appendix:arithmetic}, that the elliptic curve factor is not defined over $K_+$ but rather over the field of definition $K$ of $\Aut(C)$ whose relative Galois group $\Gal(K / K_+) \cong C_2$ acts on conjugacy classes of subgroups of $\Aut(C)$ by the full group of outer automorphisms. Additionally, as is verified in the \textsc{Magma} script \cite[\texttt{subgroup\_schemes.m}]{github}, $\Aut(C)$, as a form of $\PSL_2(\FF_{13})$ over $K_+$, admits subgroups $H \subset \Aut(C)$ which recover a copy of each of $D_7, D_6, A_4$ geometrically. Hence, the surfaces in Theorem~\ref{intro:thm:example} are defined over $K_+$ (we can also use Proposition~\ref{prop:arithmetic_resolution} of Appendix~\ref{appendix:computations} to obtain the Hirzebruch--Jung resolution over $K_+$).

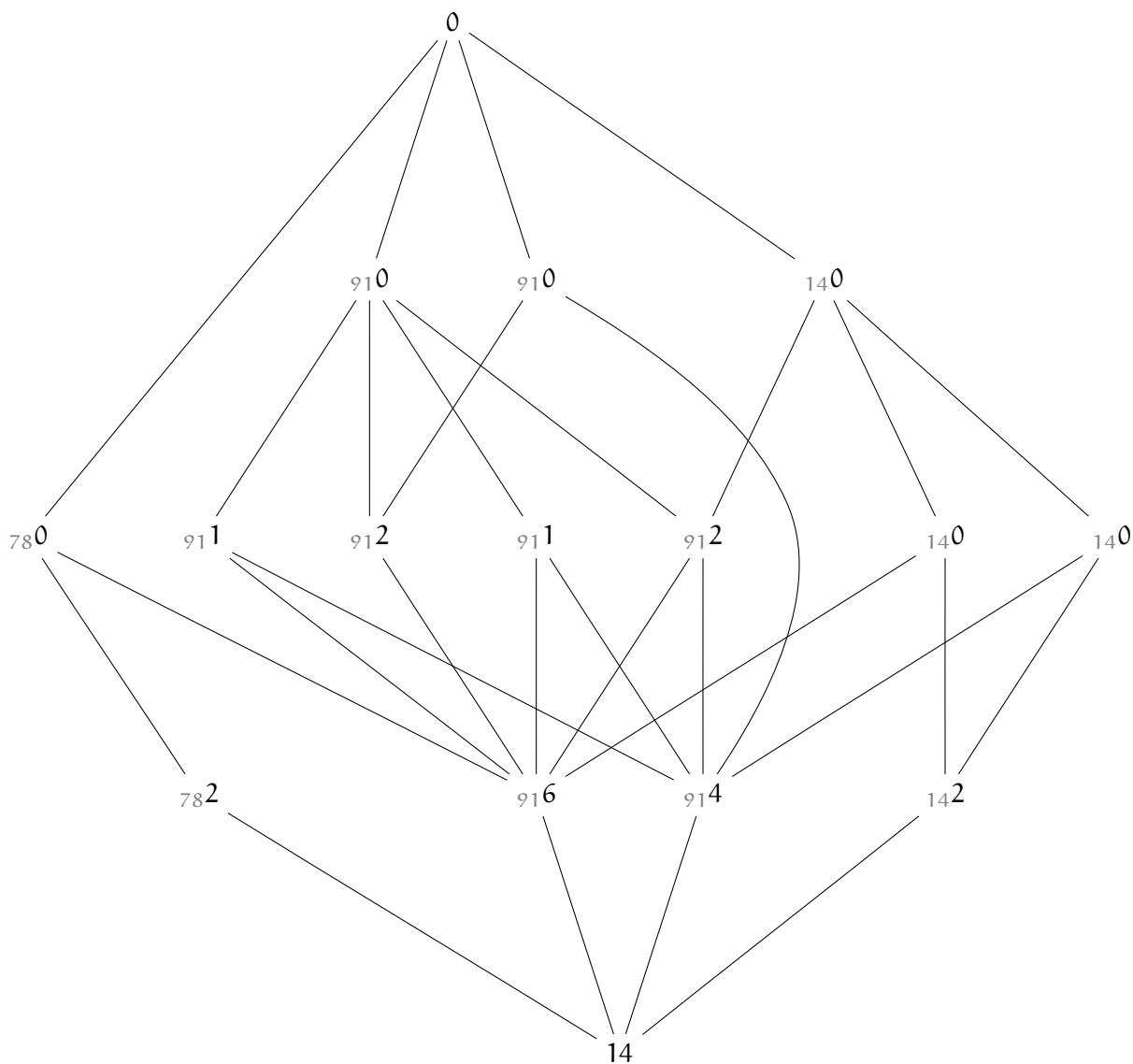
\begin{figure}
\makebox[\textwidth]{
\begin{tikzpicture}[scale = 0.8,line join=bevel]
\node (1) at (334.5bp,18.0bp) [draw,draw=none] {$14$};
  \node (2) at (292.5bp,148.0bp) [draw,draw=none] {${}_{\textcolor{gray}{91}} 6$};
  \node (3) at (376.5bp,148.0bp) [draw,draw=none] {${}_{\textcolor{gray}{91}} 4$};
  \node (4) at (122.5bp,148.0bp) [draw,draw=none] {${}_{\textcolor{gray}{78}} 2$};
  \node (5) at (498.5bp,148.0bp) [draw,draw=none] {${}_{\textcolor{gray}{14}} 2$};
  \node (6) at (208.5bp,278.0bp) [draw,draw=none] {${}_{\textcolor{gray}{91}} 2$};
  \node (7) at (124.5bp,278.0bp) [draw,draw=none] {${}_{\textcolor{gray}{91}} 1$};
  \node (8) at (376.5bp,278.0bp) [draw,draw=none] {${}_{\textcolor{gray}{91}} 2$};
  \node (9) at (292.5bp,278.0bp) [draw,draw=none] {${}_{\textcolor{gray}{91}} 1$};
  \node (10) at (36.5bp,278.0bp) [draw,draw=none] {${}_{\textcolor{gray}{78}} 0$};
  \node (11) at (498.5bp,278.0bp) [draw,draw=none] {${}_{\textcolor{gray}{14}} 0$};
  \node (12) at (582.5bp,278.0bp) [draw,draw=none] {${}_{\textcolor{gray}{14}} 0$};
  \node (14) at (292.5bp,408.0bp) [draw,draw=none] {${}_{\textcolor{gray}{91}} 0$};
  \node (13) at (208.5bp,408.0bp) [draw,draw=none] {${}_{\textcolor{gray}{91}} 0$};
  \node (15) at (437.5bp,408.0bp) [draw,draw=none] {${}_{\textcolor{gray}{14}} 0$};
  \node (16) at (250.5bp,538.0bp) [draw,draw=none] {$0$};
  \draw [] (1) ..controls (320.82bp,60.692bp) and (306.25bp,105.1bp)  .. (2);
  \draw [] (1) ..controls (348.18bp,60.692bp) and (362.75bp,105.1bp)  .. (3);
  \draw [] (1) ..controls (265.33bp,60.761bp) and (191.49bp,105.35bp)  .. (4);
  \draw [] (1) ..controls (388.01bp,60.761bp) and (445.13bp,105.35bp)  .. (5);
  \draw [] (2) ..controls (265.14bp,190.69bp) and (236.0bp,235.1bp)  .. (6);
  \draw [] (2) ..controls (237.69bp,190.76bp) and (179.17bp,235.35bp)  .. (7);
  \draw [] (2) ..controls (319.86bp,190.69bp) and (349.0bp,235.1bp)  .. (8);
  \draw [] (2) ..controls (292.5bp,190.69bp) and (292.5bp,235.1bp)  .. (9);
  \draw [] (2) ..controls (210.59bp,189.95bp) and (120.23bp,235.13bp)  .. (10);
  \draw [] (2) ..controls (359.71bp,190.76bp) and (431.46bp,235.35bp)  .. (11);
  \draw [] (3) ..controls (295.05bp,190.37bp) and (206.07bp,235.57bp)  .. (7);
  \draw [] (3) ..controls (376.5bp,190.69bp) and (376.5bp,235.1bp)  .. (8);
  \draw [] (3) ..controls (349.14bp,190.69bp) and (320.0bp,235.1bp)  .. (9);
  \draw [] (3) ..controls (443.71bp,190.76bp) and (515.46bp,235.35bp)  .. (12);
  \draw [] (3) ..controls (407.78bp,194.41bp) and (438.71bp,251.25bp)  .. (418.5bp,296.0bp) .. controls (399.24bp,338.64bp) and (354.6bp,371.45bp)  .. (14);
  \draw [] (4) ..controls (94.488bp,190.69bp) and (64.652bp,235.1bp)  .. (10);
  \draw [] (5) ..controls (498.5bp,190.69bp) and (498.5bp,235.1bp)  .. (11);
  \draw [] (5) ..controls (525.86bp,190.69bp) and (555.0bp,235.1bp)  .. (12);
  \draw [] (6) ..controls (208.5bp,320.69bp) and (208.5bp,365.1bp)  .. (13);
  \draw [] (6) ..controls (235.86bp,320.69bp) and (265.0bp,365.1bp)  .. (14);
  \draw [] (7) ..controls (151.86bp,320.69bp) and (181.0bp,365.1bp)  .. (13);
  \draw [] (8) ..controls (321.69bp,320.76bp) and (263.17bp,365.35bp)  .. (13);
  \draw [] (8) ..controls (396.37bp,320.69bp) and (417.53bp,365.1bp)  .. (15);
  \draw [] (9) ..controls (265.14bp,320.69bp) and (236.0bp,365.1bp)  .. (13);
  \draw [] (10) ..controls (89.858bp,343.33bp) and (197.28bp,472.83bp)  .. (16);
  \draw [] (11) ..controls (478.63bp,320.69bp) and (457.47bp,365.1bp)  .. (15);
  \draw [] (12) ..controls (535.27bp,320.69bp) and (484.97bp,365.1bp)  .. (15);
  \draw [] (13) ..controls (222.18bp,450.69bp) and (236.75bp,495.1bp)  .. (16);
  \draw [] (14) ..controls (278.82bp,450.69bp) and (264.25bp,495.1bp)  .. (16);
  \draw [] (15) ..controls (376.49bp,450.76bp) and (311.35bp,495.35bp)  .. (16);
\end{tikzpicture}}
    \caption{Lattice of intermediate curves of the cover $C \to \P^1$ up to conjugacy with their genera marked. The small numbers in gray represent the number of isomorphic conjugate copies of a curve represented as a particular quotient.}
    \label{fig:subgroups_genera}
\end{figure}

In order for the product-quotient surface $X \to (C^{\nu_0} \times C^{\nu_0})/H$ to have zero irregularity, we must choose $H \subset G$ so that $C / H$ has genus $0$. Among these, the minimal groups having rational quotients are most promising since the larger $\# H$ the more difficult it becomes for $(K_X - E)^2 > 0$ to be satisfied. From these figures, the minimal groups with genus $0$ quotient are $D_7, D_6, A_4$, and $C_6 \rtimes C_3$. It turns out, the first three, but not the last, will produce simply-connected examples with $(K_X - E)^2 > 0$.

\subsubsection{Supersingularity} 

The Hurwitz curve $C$ of genus $14$ has the amiable property that its Jacobian is isogenous to $14$ copies of some elliptic curve $E$. Moreover, this curve $E$ is defined over the field $K = (K_+)^{\p_{13} \infty} = K_+(\sqrt{-3 \eta - 2})$ by Theorem~\ref{app:thm:isogeny_decomp_defined_over_K} in Appendix~\ref{appendix:arithmetic}. Since $K$ has a real place, we may apply Elkies' theorem \cite{Elkies} to conclude that there are infinitely many places of $K$ at which $E$ has supersingular reduction. Since supersingularity is a geometric property, at each prime $\p \subset \cO_{K_+}$ lying under a prime in $\cO_K$ of supersingular reduction for $E$ where $X_{\p}$ has good reduction, we conclude that $X_{\p}$ is supersingular by Corollary~\ref{cor:curves_supersingular}.  

\subsection{Proof of Theorem~\ref{intro:thm:example}}

For any product-quotient surface $X$ dominated by a product of Hurwitz curves of genus $14$, we saw in the previous section that $X$ has infinitely many primes of supersingular reduction. 
In order for $X$ to violate Shioda's conjecture, we must show that $\pi_1(X) = 1$ and, to apply Theorem~\ref{thm:obstruction}, that its asymmetric Frobenius twists all carry diagonal symmetric forms.
\par 
For the minimal resolution $\pi : X \to Y := (C \times C)/H$ of each surface defined by a subgroup $H \subset \PSL_2(\FF_{13})$ for $H$ among the copies of $D_7, D_6, A_4$ in $\PSL_2(\FF_{13})$, we will verify these claims. Fix the following notation: for $\sigma, \tau \in \Gal(K'/\Q)$, denote by $\pi^{\sigma, \tau} : X^{\sigma, \tau} \to ((C^{\nu_0})^{\sigma} \times (C^{\nu_0})^{\tau}) / H$ the minimal resolution, where $H$ acts diagonally using the fixed inclusion $H \subset G$ and the induced actions under the Galois twists of $G$. We will need to check
\begin{enumerate}
    \item $\pi_1(X_{\CC}) = 1$
    \item for all $\sigma, \tau \in \Gal(K'/\Q)$ we have $(K_{X^{\sigma, \tau}} - \Exc{\pi^{\sigma, \tau}})^2 > 0$
\end{enumerate}

Let us first see why this is sufficient to prove Theorem~\ref{intro:thm:example} by an application of Theorem~\ref{thm:obstruction} and Theorem~\ref{thm:obstruction_elliptic}. Let $f : \X \to \Spec{\cO_{K',S}}$ be a spreading out as a smooth projective morphism over the ring of $S$-integers of $K'$ away from some finite set of primes $S$. An explicit set $S$ is computed in Appendix~\ref{appendix:computations}. For any prime $\frP \subset \cO_{K',S}$, by Artin comparison \cite[Expos\'{e} XII, Cor. 5.2]{SGA1} and surjectivity of Grothendieck's specialization maps \cite[Expos\'{e} X, Cor. 2.3]{SGA1}, the geometric fiber of $\X$ is simply-connected, $\pi_1^{\et}(\X_{\ol{\frP}}) = 1$. By Corollary~\ref{cor:numerical_criterion}, $X^{\sigma, \tau}$ carries a diagonal symmetric form whose zero locus contains a big divisor. Likewise, spread out $X^{\sigma, \tau}$ to $f^{\sigma, \tau} : \X^{\sigma, \tau} \to \Spec{\cO_{K', S}}$. For any prime $\frP \subset \cO_{K',S}$ there exists $\Frob_{\frP} \in \Gal(K'/\Q)$ in the decomposition group of $\frP$ lifting Frobenius of the residue field $\cO_{K'} / \frP$. Therefore, setting $\sigma := (\Frob_{\frP})^a$ and $\tau := (\Frob_{\frP})^b$ we obtain that the fiber $(\X^{\sigma, \tau})_{\frP}$ is naturally isomorphic to $(\X_{\frP})^{a,b}$ in the notation of \S\ref{sec:obstruction}. By upper semi-continuity, $(\X^{\sigma, \tau})_{\frP}$ also carries a diagonal symmetric form whose zero locus contains a big divisor. Hence, Theorem~\ref{thm:obstruction} and Theorem~\ref{thm:obstruction_elliptic} apply to the geometric fibers $\X_{\ol{\frP}}$.
\par 
It remains to check the numerical claims, which are computable over any field extension. For the remainder of this section, everything is base-changed to $\CC$ along $\nu_0 : K \embed \CC$ without comment.
To perform the calculations in each case, we need to know the spherical generators for the monodromy of an intermediate cover $C \to C / H \cong \P^1$. This is computed as follows. Consider the factorization $C \to C / H \to C / G \cong \P^1$. Above each ramification point $p_i \in \P^1$ of $f : C \to \P^1$, the fiber of $C / H \to \P^1$ is in bijection with the cycle decomposition of the monodromy $g_i$ action on $G/H$. There is a loop $\gamma_{ij}$ around each point of the fiber corresponding to a cycle $\sigma_{ij}$ in this decomposition. To compute the monodromy of $C \to C / H$, we need to know the monodromy around each $\gamma_{ij}$. Let $\ell$ be the order of $\sigma_{ij}$; by definition $g_i^{\ell}$ fixes any coset $H r$ appearing in the cycle $\sigma_{ij}$. This means $r g_i^{\ell}$ is conjugate to $r$ by some element $h_{ij} \in H$ well-defined up to $H$-conjugacy. The conjugacy class of $h_{ij}$ is the conjugacy class of the monodromy action of $\gamma_{ij}$ on the fiber $H$ of $C \to C / H$. Note that this procedure only determines each monodromy element up to conjugacy and additional care must be taken to extract spherical generators, i.e., a well-defined map $\pi_1( (C/H)^\circ ) \to H$ since there may be multiple inequivalent covers defined by spherical generators with the same conjugacy classes. Luckily for the purposes of computing baskets of singularities, any two choices are equivalent (see Remark~\ref{rmk:baseket_depends_only_conjugacy_multiset}). 
\par 
From here, we can verify $\pi_1(Y) = 1$ and compute the basket of singularities of $Y$. The explicit computation of the basket $\cB$ is done as in \S\ref{section:basket}. 
Using Lemma~\ref{lemma:fund_group_torsion}, it suffices to show that $\pi_1(Y^{\circ})$ is normally generated by torsion. For each example, we will apply Proposition~\ref{prop:extends_good_presentation_normally_generated} to conclude that $\pi_1(X) = 1$. Finally, we will directly verify $(K_{X^{\sigma, \tau}} - \Exc{\pi^{\sigma, \tau}})^2 > 0$ by applying Proposition~\ref{prop:numerical_volume}.

\subsubsection{Case: $H \cong D_{7}$} \label{section:case_D7}

Let $Y = (C \times C)/H$ with $H$ acting diagonally and $\pi : X \to Y$ the minimal resolution. 
Using the method described in the preceding section, we compute that $g : C \to C / H$ is ramified at $7$ points with conjugacy classes
\[ \Cl(s), \Cl(s), \Cl(s), \Cl(s), \Cl(s), \Cl(s), \Cl(r). \]
This means $g : C \to C / H \cong \P^1$ corresponds to the refined passport 
\cite[\href{https://www.lmfdb.org/HigherGenus/C/Aut/14.14-1.0.2-2-2-2-2-2-7.1}{Higher Genus Family: 14.14-1.0.2-2-2-2-2-2-7.1}]{lmfdb} with conjugacy classes $(2, 2, 2, 2, 2, 2, 3)$ using the labeling scheme of the LMFDB. The spherical generators of $g : C \to \P^1$ viewed as a $D_{7}$-cover fit the description of Example~\ref{example:odd_dihedral}. Therefore, applying Proposition~\ref{prop:extends_good_presentation_normally_generated}, we see that $\pi_1(X) = 1$. 
\par 
Using the method described in \S\ref{section:basket}, $Y$ has the following basket of singularities
\[ \cB(Y) := \{ 36 \times A_1, 1 \times A_6, 1 \times \tfrac{1}{7} (1,1) \}. \]
Plugging into Proposition~\ref{prop:chern_numerics} gives Chern numbers:
\[ K_X^2 = 93 \quad \quad c_2(X) = 111. \]
Using Proposition~\ref{cor:numerical_criterion}:
\[ (K_X - E)^2 = 93 - 2 \cdot 36 - 2 \cdot 1 - 17 \cdot 1 = 2 \]
is positive and hence $\omega_X(-E)$ is big. This surface has Hodge diamond:
\begin{center}
\begin{tabular}{ccccc}
 &  & $1$ &  &  \\
 & $0$ &  & $0$ &  \\
$16$ &  & $77$ &  & $16$ \\
 & $0$ &  & $0$ &  \\
 &  & $1$ &  &  \\
\end{tabular}    
\end{center}
and is of general type. It is a routine matter to repeat these calculations for $X^{\sigma, \tau}$. The \textsc{Magma} script \texttt{verify\textunderscore examples.m}, available on \texttt{github} at \cite{github}, performs the necessary checks.

\begin{example} \label{example:twists_have_different_numerics}
A surprising phenomenon appears in this example: $X$ and $X^{\sigma, \tau}$ do not even have the same numerical invariants. For example, among the $X^{\sigma, \tau}$ for $H \cong D_7$ is the product-quotient surface $X'$ whose defining pair of spherical generators for $D_7$ have conjugacy classes,
\[ \Cl(s), \Cl(s), \Cl(s), \Cl(s), \Cl(s), \Cl(s), \Cl(r). \]
and
\[ \Cl(s), \Cl(s), \Cl(s), \Cl(s), \Cl(s), \Cl(s), \Cl(r^2) \]
which leads to a basket of singularities:
\[ \cB(Y') := \{ 36 \times A_1, 1 \times \frac{1}{7}(1,3), 2 \times \tfrac{1}{7} (1,4) \}. \]
which is different from $X$, replacing the $\frac{1}{7}(1,1)$ and $\frac{1}{7}(1,6)$ singularities by the pair $\frac{1}{7}(1,3)$ and $\frac{1}{7}(1,4)$. This leads to different numerics. It has Chern numbers,
\[ K_{X'}^2 = 95 \quad \quad c_2(X') = 109 \]
and thus a Hodge diamond
\begin{center}
\begin{tabular}{ccccc}
 &  & $1$ &  &  \\
 & $0$ &  & $0$ &  \\
$17$ &  & $73$ &  & $17$ \\
 & $0$ &  & $0$ &  \\
 &  & $1$ &  &  \\
\end{tabular}    
\end{center}
Finally, using Corollary~\ref{cor:numerical_criterion},
\[ (K_{X'} - E')^2 = 95 - 2 \cdot 36 - 5 \cdot 1 - 8 \cdot 1 = 10. \]
\end{example}

\subsubsection{Case: $H \cong D_{6}$} \label{section:case_D6}

Let $Y = (C \times C)/H$ with $H$ acting diagonally and $\pi : X \to Y$ the minimal resolution. 
Using the method described in the preceding section, we compute that $g : C \to C / H$ is ramified at $8$ points with conjugacy classes
\[ \Cl(r^3), \Cl(s), \Cl(s), \Cl(s), \Cl(rs), \Cl(rs), \Cl(rs), \Cl(r^2). \]
This means $g : C \to C / H \cong \P^1$ corresponds to the refined passport 
\cite[\href{https://www.lmfdb.org/HigherGenus/C/Aut/14.12-4.0.2-2-2-2-2-2-2-3.5}{Higher Genus Family: 14.12-4.0.2-2-2-2-2-2-2-3.5}]{lmfdb} with conjugacy classes $(2, 3, 3, 3, 4, 4, 4, 5)$ using the labeling scheme of the LMFDB. We can use \textsc{Magma} to enumerate the possible spherical generators with these prescribed conjugacy classes and for each check explicitly that they extend a good presentation. Therefore, applying Proposition~\ref{prop:extends_good_presentation_normally_generated}, we see that $\pi_1(X) = 1$. 
\par 
Using the method described in \S\ref{section:basket}, $Y$ has the following basket of singularities
\[ \cB(Y) := \{ 42 \times A_1, 2 \times A_2, 2 \times \tfrac{1}{3} (1,1) \}. \]
Plugging into Proposition~\ref{prop:chern_numerics} gives Chern numbers:
\[ K_X^2 = 112 \quad \quad c_2(X) = 128 \]
and using Corollary~\ref{cor:numerical_criterion}:
\[ (K_X - E)^2 = 112 - 2 \cdot 42 - 2 \cdot 2 - 5 \cdot 2 = 14 \]
is positive and hence $\omega_X(-E)$ is big. This surface has Hodge diamond:
\begin{center}
\begin{tabular}{ccccc}
 &  & $1$ &  &  \\
 & $0$ &  & $0$ &  \\
$19$ &  & $88$ &  & $19$ \\
 & $0$ &  & $0$ &  \\
 &  & $1$ &  &  \\
\end{tabular}    
\end{center}
and is of general type. Unlike the $D_7$ case, all partial twists have the same Hodge diamond and $(K_X - E)^2$, checked in \texttt{verify\textunderscore examples.m}.

\subsubsection{Case: $H \cong A_4$} \label{section:case_A4}

Let $Y = (C \times C)/H$ with $H$ acting diagonally and $\pi : X \to Y$ the minimal resolution. Fix notation $a = (1 \, 2)(3 \, 4)$ and $b = (1 \, 2 \, 3)$ and $c = (1 \, 3 \, 2)$.
Using the method described in the preceding section, we compute that $g : C \to C / H$ is ramified at $7$ points with conjugacy classes,
\[ \Cl(a), \Cl(a), \Cl(a), \Cl(b), \Cl(b), \Cl(c), \Cl(c). \]
This means $g : C \to C / H \cong \P^1$ corresponds to the refined passport 
\cite[\href{https://www.lmfdb.org/HigherGenus/C/Aut/14.12-3.0.2-2-2-3-3-3-3.1}{Higher Genus Family: 14.12-3.0.2-2-2-3-3-3-3.1}]{lmfdb} with conjugacy classes $(2, 2, 2, 3, 3, 4, 4)$ using the labeling scheme of the LMFDB. We can use \textsc{Magma} to enumerate the possible spherical generators with these prescribed conjugacy classes and for each check explicitly that they extend a good presentation. Therefore, applying Proposition~\ref{prop:extends_good_presentation_normally_generated}, we see that $\pi_1(X) = 1$. 
\par 
Using the method described in \S\ref{section:basket}, $Y$ has the following basket of singularities:
\[ \cB(Y) := \{ 18 \times A_1, 8 \times A_2, 8 \times \tfrac{1}{3} (1,1) \}. \]
Plugging into Proposition~\ref{prop:chern_numerics} gives Chern numbers,
\[ K_X^2 = 110 \quad \quad c_2(X) = 118. \]
and using Corollary~\ref{cor:numerical_criterion}:
\[ (K_X - E)^2 = 110 - 2 \cdot 18 - 2 \cdot 8 - 5 \cdot 8 = 18 \]
is positive and hence $\omega_X(-E)$ is big. This surface has Hodge diamond:
\begin{center}
\begin{tabular}{ccccc}
 &  & $1$ &  &  \\
 & $0$ &  & $0$ &  \\
$18$ &  & $80$ &  & $18$ \\
 & $0$ &  & $0$ &  \\
 &  & $1$ &  &  \\
\end{tabular}    
\end{center}
and is of general type. Unlike the $D_7$ case, all partial twists have the same Hodge diamond and $(K_X - E)^2$, checked in \texttt{verify\textunderscore examples.m}.

\appendix

\section{Arithmetic of the first Hurwitz triplet} \label{appendix:arithmetic}

We survey Shimura's groundbreaking work \cite{shimura:curves} constructing canonical models of Shimura curves with an eye toward the arithmetic of certain Hurwitz curves. The main purpose of this appendix is to prove that the elliptic curve $E$ appearing as an isogeny factor in the Jacobian of a genus $14$ Hurwitz curve $C$ is defined over a number field with a real place; see Theorem~\ref{app:thm:isogeny_decomp_defined_over_K}. This is necessary to apply Elkies' theorem \cite{Elkies} to show that $C$ has infinitely many primes of supersingular reduction.

\subsection{Canonical models of Shimura curves}

The exposition follows \S2--3 of \cite{shimura:curves}. Let $F$ be a totally real number field and $B$ a quaternion algebra over $F$. Write $g = [F : \Q]$ and $\frD(B/F)$ for the product of the non-archimedean primes of $F$ at which $B$ is \textit{non-split}, meaning $B \ot_F F_{\p}$ is a division algebra over $F_{\p}$. At all other finite places $B$ is \textit{split} (or \textit{ramified}) meaning $B \ot_F F_{\p} \cong M_2(F_{\p})$. Let $\nu_1, \dots, \nu_g : F \embed \RR$ be the archimedean places ordered such that $B$ is split at $\nu_1, \dots, \nu_r$ and non-split at $\nu_{r+1}, \dots, \nu_{g}$. We also choose an isomorphism compatible with these fixed places
\[ B \ot_{\QQ} \RR \iso M_2(\RR)^r \times \HH^{g-r} \]
where $\HH$ are the Hamiltonian quaternions. Assuming $r > 0$, this defines an embedding
\[ \iota_{\infty} : B \embed M_2(\RR)^r. \]
Let $B^+$ denote the elements of $B$ with totally positive reduced norm (note that Shimura writes $N_{B/F} : B \to F$ for the reduced norm whereas often this denotes the full norm $(\Nm_{B/F})^2$ where $\Nm_{B/F}$, often written as $\mathrm{nrd}_{B/F}$, is the reduced norm).  Equivalently, since the definite quaternion algebra $\HH$ has only elements of positive norm, these are the elements $\gamma \in B$ so that $\iota_{\infty}(\gamma)$ is a sequence of real matrices with positive determinant.
\par 
Via the embedding $\iota_{\infty} : B^+ \to \GL_2^+(\RR)^r$ we get an action $B^+ \acts \frH$ where $\frH \subset \CC$ is the upper half-plane. Quotients by special subgroups of $B^+$ will lead to Shimura varieties. The associated Shimura datum is for the algebraic group $G := \Res_{F/\Q} B^\times$, the functor $A \mapsto (B \times_{\Q} A)^\times$ on $\Q$-algebras, with the conjugacy class $X$ generated by the map $h : \SS \to G_{\RR}$ defined by 
\[ \CC^\times \to \GL_2(\RR)^r \times \HH^{g-r} \quad \quad x + i y \mapsto \left( \begin{pmatrix}
    x & y 
    \\
    -y & x 
\end{pmatrix}, \dots, \begin{pmatrix}
    x & y 
    \\
    -y & x 
\end{pmatrix}, 1, \dots, 1 \right).  \]
In general, the Shimura datum $(G, X)$ is of abelian type but not necessarily of Hodge type \cite{milne}. Fix a maximal order $\cQ \subset B$ and set $\Gamma(\cQ, 1) := \cQ^\times \cap B^+$. The action $\Gamma(\cQ, 1) \acts \frH^r$ is totally discontinuous and cocompact as long as $B$ is not split over $F$. For a finite index subgroup $\Gamma \subset \Gamma(\cQ, 1)$ the quotients $Y = \frH^r / \Gamma$ are the complex points of the associated Shimura variety. For example:
\begin{enumerate}
    \item when $B = M_2(F)$ then $\cQ = M_2(\cO_F)$ and $Y$ is a Hilbert-Blumenthal modular variety which parametrizes abelian $2g$-folds with real multiplication by $\cO_F$
    \item $F = \Q$ and $B$ is a non-split indefinite quaternion algebra with discriminant $d$ then $Y$ is a Shimura curve parameterizing abelian surfaces with QM by $B$
    \item likewise, when $B$ is non-split over $F$ and $r = g$ then $Y$ is an honest moduli space parameterizing abelian $2g$-folds with QM by $\cQ$
    \item when $r = 1$ these $Y$ are the \textit{Shimura curves} but for $g > 1$ these are not honest moduli spaces of abelian varieties.
\end{enumerate}

These moduli interpretations make sense over a number field, leading to \textit{canonical models} for the modular Shimura varieties over this field. As in Shimura's original article, we focus here on the case $r = 1$. Even when $g > 1$ and the Shimura curve is not an honest moduli space, Shimura was able to construct canonical models over number fields by using isogenies to other moduli spaces. For an ideal $\frJ \subset \cO_F$, Shimura defines \textit{congruence subgroups}
\[ \Gamma(\cQ, \frJ) := \{ \gamma \in \cQ^\times \cap B^+ \mid \gamma - 1 \in \frJ \cQ \} \]
and he produces a canonical model for $\frH / \Gamma(\cQ, \frJ)$ as follows.

\begin{theorem}\cite[(3.2) Main Theorem I and (3.3)]{shimura:curves} \label{thm:canonical_models}
Let $F^{(\frJ \infty)}$ be the narrow ray class field of $\frJ$. Fix an embedding $F^{(\frJ \infty)} \embed \CC$ compatible with $\nu_1 : F \embed \RR$. Then there exists a curve $C$ defined over $F^{(\frJ \infty)}$ and a holomorphic map $\varphi : \frH \to C_{\CC}$ such that
\begin{enumerate}
    \item $\varphi$ descends to a biholomorphism $\frH  /\Gamma(\cQ, \frJ) \iso C_{\CC}$
    \item For any $K/F$ totally imaginary quadratic extension so that $B \ot_F K \cong M_2(K)$ and embedding $K \embed \CC$ compatible with $\nu_1 : F \embed \RR$ and any $F$-linear embedding $f : K \embed B$ so that $f(\cO_K) \subset \cQ$, let $z$ be the fixed point of $f(K) \acts \frH$; then $K^{(\frJ' \infty)}$ equals the compositum $K \cdot F^{(\frJ \infty)}(\varphi(z)) = K^{(\frJ' \infty)}$ where $\frJ' := \frJ \cO_K$ and $F^{(\frJ \infty)}(\varphi(z))$ is the field of definition of the point $\varphi(z) \in C(\ol{\Q})$.  
\end{enumerate}
Furthermore, let $(C, \varphi)$ and $(C', \varphi')$ be two such pairs. Then there exists an isomorphism $j : C \to C'$ defined over $F^{(\frJ \infty)}$ such that $\varphi' = j_{\CC} \circ \varphi$. 
\end{theorem}

Due to the uniqueness property, we call such a pair $(C, \varphi)$ a \textit{canonical model} and denote $C$ by $\X(\cQ, \frJ)$. This uniqueness property has the following consequence for automorphisms of $\X(\cQ, \frJ)$. Note that $\Gamma(\cQ, \frJ) \triangleleft \Gamma(\cQ, 1)$ so there is a well-defined action $\Gamma(\cQ, 1) \acts \frH/\Gamma(\cQ, \frJ)$ which induces an action on $\X(\cQ, \frJ)_{\CC}$ via $\varphi$.

\begin{lemma} \label{lem:automorphisms_defined_over_reflex}
$\Gamma(\cQ, 1)$ acts on $\X(\cQ, \frJ)$ via automorphisms \textit{defined over} $F^{(\frJ \infty)}$.
\end{lemma}

\begin{proof} 
For any $\alpha \in \cQ^\times \cap B^+$, we can conjugate $f : K \to B$ to get $f^{\alpha} : K \to B$ with image $\alpha f(K) \alpha^{-1}$ and $f^{\alpha}(\cO_K) = \alpha f(\cO_K) \alpha^{-1} \subset \cQ$ since $\alpha \in \cQ^\times$ and $f(\cO_K) \subset \cQ$. Thus, if $z \in \frH$ is the fixed point of $f(K)$ then $\alpha \cdot z$ is fixed by $f^{\alpha}(K)$, so the hypothesis says that $K \cdot F^{(\frJ \infty)}(\varphi(\alpha \cdot z)) = K^{(\frJ \infty')}$. Let $\varphi' := \varphi \circ \alpha$ viewing $\alpha$ as its action on $\frH$. Then the pair $(C, \varphi')$ is a canonical model because $\varphi(z') = \varphi(\alpha \cdot z)$ also satisfies $K \cdot F^{(\frJ \infty)}(\varphi(z)) = K^{(\frJ' \infty)}$ for any choice of $K/F$. Therefore, there exists $j : C \to C$ defined over $F^{(\frJ \infty)}$ so that $\varphi' = j_{\CC} \circ \varphi$, meaning exactly that the automorphism of $\X(\cQ, \frJ)$ induced by $\alpha$ is defined over $F^{(\frJ \infty)}$.
\end{proof}

\subsection{Congruence subgroups}

At split places of $B$ we can also define other congruence subgroups in analogy with the modular curves $X_0(p)$ and $X_1(p)$. Then we can fix an isomorphism:
\[ \iota_{\p} : B \ot_F F_{\p} \iso M_2(F_{\p}) \]
with the property that it restricts to an isomorphism $\iota_{\p} : \cQ \ot_{\cO_F} \cO_{F_\p} \iso M_2(\cO_{\p})$ (see \cite[Lemma~10.4.3]{quaternion_algebras}). 
Hence, reduction mod $\p^{\ell}$ gives an isomorphism
\[ \iota_{\p} : \cQ / \p \cQ \iso M_2(\kappa_{\p}) \]
where $\kappa_{\p} = \cO_F / \p$ is the residue field. This defines a reduction map
\[ \iota_{\p} : \Gamma(\cQ, 1) \to \GL_2(\kappa_{\p}) \]
whose kernel is exactly $\Gamma(\cQ, \p)$. More generally, if $\frJ \subset \cO_F$ is any ideal whose prime factorization consists entirely of primes splitting $B$, then we can fix a reduction map
\[ \iota_{\frJ} : \Gamma(\cQ, 1) \to \GL_2(\cO_F / \frJ).  \]
Taking the preimage of the upper Borel (resp.\ its unipotent radical, the upper triangular matrices) gives congruence subgroups $\Gamma_0(\cQ, \frJ)$ (resp.\ $\Gamma_1(\cQ, \frJ)$). Beware: this is not consistent with Shimura's use of the notation $\Gamma_1$. The groups $\Gamma_0(\cQ, \frJ)$ are the totally positive units in an \textit{Eichler order} $\cQ_0(\frJ) \subset \cQ$ of level $\frJ$ cf.\ \cite[\S3]{Voight:explicit_methods}.
\par 
However, there are two caveats with these definitions: first $\iota_{\frJ} : \Gamma(\cQ, 1) \to \GL_2(\cO_F / \frJ)$ may not be surjective because of the restriction that elements live in $B^+$. Second, $\iota_{\infty}$ does not give a \textit{faithful} action of $\Gamma(\cQ, 1)$ on $\frH$ because the action factors through $\GL_2^+(\RR) \to \PGL_2^+(\RR)$. In order to understand the relationship between quotients by the various congruence subgroups, we need to precisely understand these failings. 
\par 
For any $\Gamma \subset \Gamma(\cQ, 1)$ denote by $\mathrm{P} \Gamma$ the quotient 
\[ \mathrm{P} \Gamma := \frac{\Gamma}{\cO_F^\times \cap \Gamma} = \frac{\cO_F^{\times} \cdot \Gamma}{\cO_F^{\times}} \]
where $\cO_F^\times \embed \Gamma(\cQ, 1)$ are the scalar matrices. Since $\cQ^{\Nm = 1} \subset \Gamma(\cQ, 1)$ generates $\SL_2(\cO_{\p}) \subset \GL_2(\cO_{\p})$ it is clear that $\im{\iota_{\frJ}}$ contains $\SL_2(\cO_F / \frJ)$. Let 
\[ \Gamma^1(\cQ, \frJ) := \Gamma(\cQ, \frJ)^{\Nm = 1} = \{ \gamma \in \cQ^{\Nm = 1} \mid \gamma - 1 \in \frJ \cQ \} \] 
be the subgroup consisting of elements of $\Gamma(\cQ, \frJ)$ with reduced norm $1$. This is the kernel of $\iota_{\frJ}$ restricted to $\Gamma^1(\cQ, 1) := \cQ^{\Nm = 1}$.
We get a diagram of exact sequences:
\begin{center}
    \begin{tikzcd}
        & 0 \arrow[d] & 0 \arrow[d] & 0 \arrow[d] & 0 \arrow[d]
        \\
        0 \arrow[r] & L \arrow[d, hook] \arrow[r] & \Gamma^1(\cQ, \frJ) \arrow[d, hook] \arrow[r] & \mathrm{P} \Gamma(\cQ, \frJ) \arrow[r, "\Nm"] \arrow[d, hook] & U_{\p}^+ / U_\p^2\arrow[d, hook] &
        \\
        0 \arrow[r] & \{ \pm 1 \} \arrow[r] \arrow[d] & \Gamma^1(\cQ, 1) \arrow[r] \arrow[d, "\iota_{\frJ}"] & \mathrm{P} \Gamma(\cQ, 1) \arrow[r, "\Nm"] \arrow[d] & (\cO_F^\times)^+ / (\cO_F^{\times})^2 \arrow[d]
        \\
        0 \arrow[r] & \mu_2 \arrow[r] & \SL_2(\cO_F / \frJ) \arrow[r] & \PGL_2(\cO_F / \frJ) \arrow[r, "\det"] & (\cO_F / \frJ)^\times / (\cO_F / \frJ)^{\times 2} \arrow[r] & 0  
    \end{tikzcd}
\end{center}
where $U_{\p} \subset \cO^\times_F$ is the group of $u \in \cO_F^{\times}$ such that $u \equiv 1 \mod \p$ and $L$ is $\{ \pm 1 \}$ if $2 \in \p$ and trivial otherwise. Therefore, when $(\cO_F^\times)^+ = (\cO_F^\times)^2$, for example, when the narrow class number of $F$ equals its class number, then the diagram shows that 
\[ \Gamma^1(\cQ, \frJ) / L = \mathrm{P} \Gamma(\cQ, \frJ) \text{ and } \Gamma^1(\cQ, 1) / \{ \pm 1 \} = \mathrm{P} \Gamma(\cQ, 1) \text{ and } \im{\iota_{\frJ}} = \PSL_2(\cO_F / \frJ). \]
Therefore, in this case, $\X(\cQ, \frJ) \to \X(\cQ, 1)$ is a $\PSL_2(\cO_F / \frJ)$-cover. It has intermediate covers 
\[ \X(\cQ, \frJ) \to \X_1(\cQ, \frJ) \to \X_0(\cQ, \frJ) \to \X(\cQ, 1) \]
which are quotients of $\X(\cQ, \frJ)$ by the Borel and its unipotent radical respectively. 

Let $H \subset \GL_2(\cO_{\p})$ be an open subgroup for a prime $\p$ at which $B$ is split. The \textit{level} of $H$ is the smallest $\ell$ such that $H$ is the preimage of a subgroup $\ol{H} \subset \GL_2(\cO_F / \p^{\ell})$. Then we can define the Shimura curve $\X_H(\cQ)$ with $H$-level structure as the quotient of $\X(\cQ, \p^{\ell})$ by the subgroup $\Gamma_H(\cQ) := \Gamma(\cQ, 1) \cap H = \iota_{\p^{\ell}}^{-1}(\ol{H})$. By Lemma~\ref{lem:automorphisms_defined_over_reflex}, $\X_H(\cQ)$ also has a canonical model over $F^{(\frJ \infty)}$. 

\subsection{Certain Hurwitz curves as Shimura curves}

\newcommand{\Hur}{\mathrm{Hur}}
\newcommand{\rmP}{\mathrm{P}}

Much of the favorable arithmetic of the genus 14 Hurwitz curves is due to their realization as Shimura curves for a particular quaternion algebra $B$ over $K_+ = \Q(\eta)$ where $\eta = \zeta_7 + \zeta_7^{-1} = 2 \cos{\tfrac{2\pi}{7}}$. The exposition here closely follows \S3.19 of \cite{shimura:curves} and Elkies' article \cite[\S4.4]{klein_quartic}. Choose as $B$ the following quaternion algebra:
\[ B = \left( \frac{\eta, \eta}{F} \right) := \frac{K_+ \left< i, j \right>}{(i^2 - \eta, j^2 - \eta, ij + ji)}. \]
This is the unique quaternion algebra over $K_+$ that is non-split exactly over the two real places of $K_+$ where $\eta$ embeds negatively and split over all non-archimedean places. Let $v_{\infty} : K_+ \embed \RR$ denote the obvious real embedding sending $\eta \mapsto 2 \cos{\tfrac{2\pi}{7}}$ which splits $B$. We can thus fix an embedding $\iota_{\infty} : B \embed M_2(\RR)$ compatible with $v_{\infty}$ arising from a choice of isomorphism $B \ot_{K_+, v_{\infty}} \RR \iso M_2(\RR)$. We apply Shimura's construction from the previous section to this quaternion algebra. Define the following maximal order:
\[ \cQ_{\mathrm{Hur}} = \cO_{K_+}[i,j,j'] \subset B \]
where $j' := \tfrac{1}{2}(1 + \eta i + \tau j)$ and $\tau := 1 + \eta + \eta^2$. Since $K_+$ has narrow class number $1$, the previous section shows $\Gamma(\cQ,1) \onto \PSL_2(\cO_{K_+} / \frJ) \subset \PGL_2(\cO_{K_+} / \frJ)$.
\par 
Miraculously, 
\[ \mathrm{P} \Gamma(\cQ_{\Hur}, 1) = (\cQ_{\Hur})^{\Nm = 1} / \{ \pm 1 \} \cong \TT(2,3,7) = \left< g_1, g_2, g_3 \mid g_1^2, g_2^3, g_3^7, g_1 g_2 g_3 \right> \]
recovers the Hurwitz triangle group. Therefore, any finite-index normal subgroup $\Gamma \subset \Gamma(\cQ_{\Hur}, 1)$ without elliptic elements produces a cover $\X_{\Gamma} \to \X(\cQ_{\Hur},1) \cong \P^1$ ramified at exactly three points with orders $2,3,7$ and hence a Hurwitz curve with automorphism group $\rmP \Gamma(\cQ_{\Hur},1) / \rmP \Gamma$. Conversely, any Hurwitz curve is uniformized by $\frH$ with deck transformation group $\Gamma \subset \SL_2(\RR)$ (i.e., $\Gamma$ is a Fuchsian subgroup so that $\frH / \Gamma$ has $84(g-1)$ automorphisms) is conjugate in $\SL_2(\RR)$ to a normal subgroup of $\Gamma(\cQ_{\Hur}, 1)$ modulo the center.
\par 
Because $B$ is non-split only at $\infty$ we know that $\cQ_{\Hur} / \p \cong \SL_2(\kappa_{\p})$ for all primes $\p \subset \cO_{K_+}$. Therefore, the Shimura curves $\X(\cO, \p)$ are Hurwitz curves with automorphism groups $\PSL_2(\kappa_{\p})$. By Theorem~\ref{thm:canonical_models}, this curve is defined over the ray class field $K_+^{(\p \infty)}$. For any prime $\p_{13}$ lying over $13$ this field $(K_+)^{(\p_{13} \infty)}$ is isomorphic to the field $K$ described in the main text (\cite[\href{https://www.lmfdb.org/NumberField/6.2.31213.1}{Number field 6.2.31213.1}]{lmfdb}). In fact, when $p$ is inert ($p \equiv \pm 2, \pm 3 \mod 7$) or ramified ($p = 7$) in $K_+$, the curves $\X^{\Hur}(\p)$ for the unique prime $\p$ over $p$ are defined over $\Q$ and when $p$ splits ($p \equiv \pm 1 \mod 7$) in $K_+$ the three curves $\X^{\Hur}(\p)$ for $\p \mid p$ are defined over $K_+$ and arise from the three embeddings $K_+ \embed \CC$ by \cite[Theorem~B]{voight:triangle_groups}
(cf. \cite[Proposition~5.1.2]{voight:thesis}). From now on, we make the abbreviation $\X^{\Hur}(\frJ) := \X(\cQ_{\Hur}, \frJ)$.

\begin{figure}
\begin{center}
    \begin{tabular}{c|c|c|c|c}
        $N(\p)$ & $\p$ & $m$ & $g$ & $G$
         \\
         \hline
         7 & $\p_7$ & 1 & 3 & $\PSL_2(\FF_7)$
         \\
         8 & $(2)$ & 1 & 7 & $\SL_2(\FF_8)$
         \\
         13 & $\p_{13}$ & 3 & 14 & $\PSL_2(\FF_{13})$
         \\
         27 & $(3)$ & 1 & 118 & $\PSL_2(\FF_{3^3})$
         \\
         29 & $\p_{29}$ & 3 & 146 & $\PSL_2(\FF_{29})$
    \end{tabular}
\end{center}
\caption{Prime ideals of $\cO_F$ ordered by norm. Here $m$ is the number of Galois conjugate prime ideals which all give distinct Hurwitz curves (see \cite[3.19.2]{shimura:curves}) each having genus $g$ and automorphism group $G$.}
\end{figure}

\begin{theorem}
All three Hurwitz curves of genus $14$ occur as $\X^{\Hur}(\p_{13})$ for the three primes $\p_{13}$ lying over $p = 13$. In particular, these three curves are Galois conjugate.
\end{theorem} 

\begin{proof}
As tabulated in \cite[\href{https://www.lmfdb.org/HigherGenus/C/Aut/14.1092-25.0.2-3-7}{Higher Genus Family: 14.1092-25.0.2-3-7}]{lmfdb}, there are exactly three isomorphism classes of curves over $\CC$ of genus $14$ with automorphism group $\PSL_2(\FF_{13})$ and no other Hurwitz groups with order $1092$ exist. Furthermore, by \cite[3.19.2]{shimura:curves}, no two $\X^{\Hur}(\p_{13})$ are isomorphic over $\CC$ (base changed along the fixed embedding $K_+ \embed \RR$) for different choices of the prime $\p_{13}$ over $p$. Therefore, all three isomorphism classes of Hurwitz curves over $\CC$ appear. 
\end{proof}

\begin{rmk}
Transitivity of the Galois action also follows from \cite[Theorem B (b)]{voight:triangle_groups} where it is shown that the $K_+$ is the field of moduli for $\PSL_2(\FF_{13})$-Galois $(2,3,7)$ \Belyi maps forgetting the isomorphism $\PSL_2(\FF_{13}) \iso \Aut(C)$. Note: given a Hurwitz curve, its $(2,3,7)$ \Belyi map is unique and canonically determined as the group scheme quotient $C \to C / \Aut(C)$.
\end{rmk}

\begin{theorem} \label{app:thm:isogeny_decomp_defined_over_K}
Let $\X^{\Hur}(\p_{13}) / K_+$ be a genus $14$ Hurwitz curve. Then there is an elliptic curve $E$ defined over $K = (K_+)^{(\p_{13}\infty)}$ such that $\Jac(\X(\p_{13}))$ is isogenous to $E^{14}$ over $K$. 
\end{theorem}

\begin{proof}
Let $C = \X^{\Hur}(\p_{13})$. Computing fixed points and comparing to genera of curves in the quotient lattice, we determine $H^1(C(\CC), \Q)$ as a $G$-representation to be $V_{14} \oplus V_{14}$ where $V_{14}$ is the $14$-dimensional irreducible $\Q$-representation of $G$ corresponding to the character \cite[\href{https://beta.lmfdb.org/Groups/Abstract/Qchar_table/1092.25?char_highlight=1092.25.14b}{1092.25.14b}]{lmfdb}. Moreover, the action of $\Q[G]$ on $V_{14}$ gives a surjection $\Q[G] \onto M_{14}(\Q)$ through which $\Q[G] \acts H^1(C(\CC), \Q)$ factors. Since $G \acts C$ is defined over $K$, the group algebra $\Q[G]$ acts by $K$-isogenies on $\Jac(C)$ so the idempotents in $M_{14}(\Q)$ give an isogeny decomposition
\[ \Jac(C) \sim E^{14} \]
defined over $K$. 
\end{proof}

\begin{cor} \label{app:cor:supersingular_reduction}
$\X^{\Hur}(\p_{13})$ is supersingular at infinitely many primes of $K_+$.  
\end{cor}

\begin{proof}
Since $K$ has a real place, $E$ has infinitely many primes of supersingular reduction by \cite{Elkies}. Therefore, for infinitely many primes $\p \subset \cO_{K_+}$ there is a prime $\mathfrak{P} \subset \cO_{K}$ over $\p$ such that $E_{\mathfrak{P}}$ is supersingular. Therefore, whenever $C$ and $E$ have good reduction, we conclude that $C_{\mathfrak{P}}$ is supersingular since supersingularity is a  property of the Tate module of $\Jac(C)_{\mathfrak{P}}$ which is isomorphic to the Tate module of $E_{\mathfrak{P}}^{14}$. However, supersingularity is a geometric property, so $C_{\p}$ is supersingular as well. 
\end{proof}

\section{Explicit computations} \label{appendix:computations}

At this point, we have deduced that the elliptic curve $E$ appearing in the isogeny decomposition of $\Jac(\X^{\Hur}(\p_{13}))$ has infinitely many primes of supersingular reduction. This suffices to prove the main results of the paper. In this section, we will actually compute explicit primes ($1091, 2339, 6551, 7643, \dots$) of non-unirational supersingular reduction for the surfaces in Theorem~\ref{intro:thm:example}. To do this, we will compute the $j$-invariant of the elliptic curve $E$ appearing in the isogeny decomposition of $\Jac(\X^{\Hur}(\p_{13}))$. To do this, we rely on the Eichler--Shimizu--Jacquet--Langlands correspondence to relate \textit{quaternionic cusp forms} for Shimura curves to Hilbert cusp forms and the Eichler--Shimura relation to compute traces of Frobenius from the Hecke eigenvalues. We will then be able to identify the elliptic curve by a finite number of Frobenius traces since it is defined over $K = (K_+)^{(\p_{13} \infty)}$ with bounded conductor. Unfortunately, algorithms to compute the spaces of Hilbert modular forms as Hecke modules are implemented and their results tabulated in the LMFDB only for level $\Gamma_0$ congruence subgroups. In our case, $\X_0^{\Hur}(\p_{13})$ is rational (see Figure~\ref{fig:subgroups_genera}) so there are no corresponding cusp forms. Luckily, conjugation isomorphisms realize the intermediate curve $\X_1^{\Hur}(\p_{13})$ of genus $2$ as another Shimura curve $\X_0^{\Hur}(\p_{13}^2)$ for $\Gamma_0$ at the higher level $\p_{13}^2$. This allows us to read off the Galois representation of a $2$-dimensional isogeny factor $\Jac(\X_1^{\Hur}(\p_{13})) \sim \Res^{K}_{K_+}{E}$ inside $\Jac(\X^{\Hur}(\p_{13}))$.
\par 
This determines a list of supersingular primes for $E$. Next, we show that the product-quotient surface $X$ has good reduction at all primes away from $\# G$ and $N(\p_{13}) = 13$. To do this, we use the canonical integral models of Shimura curves which have good reduction away from the level and some results on simultaneous resolution for cyclic quotient singularities in families. In total, we verify the claim in the introduction that our examples violate the Shioda conjecture in characteristics $1091, 2339, 6551, 7643, \dots$.

\subsection{Quaternionic modular forms}

In this section, we discuss modular forms invariant under action of congruence subgroups for a quaternion order $\cQ$. The relevance for us is that the Eichler--Shimura relation computes Frobenius traces on \etale cohomology in terms of Hecke eigenvalues. 

\begin{defn}
For a matrix 
\[ \gamma = 
\begin{pmatrix}
a & b 
\\
c & d
\end{pmatrix} \in \GL_2(\RR) \]
and $\tau \in \frH$ define the \textit{factor of automorphy}
\[ j(\gamma, \tau) := c \tau + d \in \CC \]
For a \textit{weight} vector $k = (k_1, \dots, k_n) \in (\Z^+)^r$, we define a right \textit{weight-k action} of $\Gamma(\cQ, 1)$ on functions $f : \frH^r \to \CC$ via
\[ (f|_k \gamma)(\tau) := f(\gamma \cdot \tau) \prod_{i = 1}^r (\det{\nu_i(\gamma)})^{k_i/2} j(\nu_i(\gamma), \tau_i)^{-k_i} \]
For $\Gamma \subset \Gamma(\cQ, 1)$ the space $M_k^B(\Gamma)$ of \textit{weight-k modular forms at level} $\Gamma$ is the $\CC$-vector space of holomorphic functions $f : \frH^r \to \CC$ invariant under the weight-$k$ action of $\Gamma$ and bounded at the cusps. The space $S_k^B(\Gamma)$ of \textit{weight-k cusp forms at level} $\Gamma$ is the subspace of functions vanishing at the cusps of $\frH^r / \Gamma$. 
\end{defn}

In the case $B = M_2(F)$, $\cQ = M_2(\cO_F)$, and  $r = g$, we recover the spaces of Hilbert modular forms for the totally real field $F$. In that example, we drop the quaternion algebra ``$B$'' from the notation. We focus on the case $\Gamma = \Gamma_0(\cQ, \frJ)$ because these spaces have been computed in \cite{lmfdb} for small level. Write $S_k^B(\frJ) := S_k^B(\Gamma_0(\cQ, \frJ))$ and ditto for $M_k^B(\frJ)$. Parallel to the usual story, there is a Hecke algebra action on $S_k(\Gamma)$ now with Hecke operators $T_{\p}$ indexed by prime ideals $\p \subset \cO_F$. In the case when $B$ is non-split over $F$ it turns out there are no cusps (the corresponding subgroup is cocompact) so the condition at cusps is vacuous. In this case we say $S_k^B(\frJ) = M_k^B(\frJ)$ is the space of \textit{quaternionic cusp forms}.
\par 
The case of parallel weight $2$, meaning $k = (2,\dots,2)$, is particularly important because such cusp forms define holomorphic symmetric $r$-forms on the Shimura variety $\frH^r / \Gamma$ (really the associated analytic quotient stack). In the case $r = 1$, this induces the \textit{Eichler--Shimura isomorphism}
\[ S_2^B(\frJ) \oplus \ol{S_2^B(\frJ)} \iso H^1(\X_0(\frJ), \CC). \]
Let $\TT_0(\frJ)$ be the quotient of the Hecke algebra acting faithfully on $S_2^B(\frJ)$. Then $H^1(\X_0(\frJ), \CC)$ becomes a free $\TT_0(\frJ)_{\CC}$-module of rank $2$. A Galois orbit of \textit{eigenforms} $f \in S_2^B(\frJ)$ for the Hecke operators is equivalent to an evaluation map $\ev_f : S_2^B(\frJ) \onto K_f$ sending $T_{\p} \mapsto a_{\p}(f)$ where $a_{\p}(f)$ is the eigenvalue of $T_{\p} f = a_{\p}(f) f$ and $K_f$ is the trace field generated by the $a_{\p}(f)$. The kernel defines a prime $\p_f = \ker{\ev_f}$ of the Hecke algebra. Since $\TT_0(\frJ)$ acts by correspondences, it acts on $\Jac(\X_0(\frJ))$ by isogenies so there is a quotient $A_f$ by $\p_f$ well-defined up to isogeny. Here $A_f$ is an abelian variety of dimension $[K_f : \Q]$ defined over $F$ \cite{zhang}. By above, its Tate module $V_{\ell}(A_f)$ is a $\Q_{\ell} \ot_{\Q} K_f$-module (whose action commutes with the Galois action) of rank $2$. The \textit{Eichler--Shimura relation} then says that for the Galois representation
\[ \rho_{A_f, \ell} : \Gal(\ol{F}/F) \to \GL_2(\Q_{\ell} \ot_{\Q} K_f) \]
and any prime $\p \subset \cO_F$ with $\p \nmid \ell \frJ$, the characteristic polynomial of $\rho_{A_f, \ell}(\Frob_{\p})$ is
\[ T^2 - a_{\p}(f) \, T + N(\p) \]
viewed as a polynomial over $\Q_{\ell} \ot_{\Q} K_f$ \cite[\S10.3]{carayol}.
\par 
The Jacquet--Langlands correspondence now relates the Hecke module structure for parallel weight modular forms in the quaternionic and Hilbert cases.

\begin{theorem}[Eichler--Shimizu--Jacquet--Langlands correspondence] \label{thm:jacquet_langlands}
Let $\frD \subset \cO_F$ be the discriminant of $B$ and $\frJ \subset \cO_F$ an ideal coprime to $\frD$. Then there is a Hecke module embedding
\[ S_2^B(\frJ) \embed S_2(\frD \frJ) \]
whose image consists of Hilbert cusp forms which are new at all $\p \mid \frD$. 
\end{theorem}

See \cite[Proposition~2.12]{Hida} (cf.\ \cite[Theorem~2.9]{Voight:computing_hecke} and \cite[Theorem~3.9]{Voight:explicit_methods}). In particular, when $\frD = (1)$, Theorem~\ref{thm:jacquet_langlands} yields an isomorphism $S_2^B(\frJ) \cong S_2(\frJ)$ as Hecke modules. 

\subsection{Conjugation isomorphisms between Shimura curves}

Using conjugacy of certain Fuchsian groups, we can relate different Shimura curves at various levels. This will be important for explicit computations using the Jacquet-Langlands correspondence. The exposition closely follows the case of modular curves recorded in the appendix to \cite{images_of_galois} by John Voight.
\par 
Let $H \subset \GL_2(\cO_{\p})$ be an open subgroup for a prime $\p$ at which $B$ is split. Assume $H$ is contained in the preimage of the upper Borel of $\GL_2(\kappa_{\p})$. Also assume that $\p$ has a totally positive generator $\pi \in \p$ (such as occurs when $F$ has narrow class number $1$). Consider the matrix
\[ \varpi := \begin{pmatrix}
0 & -1
\\
\pi & 0
\end{pmatrix} \in \GL_2^+(\cO_F) \]
This acts by conjugation in $\GL_2(F_{\p})$ as follows:
\[ 
\varpi 
\begin{pmatrix}
a & b
\\
c & d
\end{pmatrix}
\varpi^{-1} = 
\begin{pmatrix}
d & - \pi^{-1} c
\\
- \pi b & a
\end{pmatrix}
\]
Hence, if $c \in \p$ this produces a matrix in $\GL_2(\cO_{\p})$. 
Therefore, there is a conjugation operation on the subgroups $H$ as follows: 
\[ H^{\varpi} := \varpi H \varpi^{-1} = \left \{ 
\begin{pmatrix}
d & - c
\\
- \pi b & a
\end{pmatrix} \: \middle| \: 
\begin{pmatrix}
a & b
\\
\pi c & d
\end{pmatrix} 
\in H 
\right \} \]
This conjugation action on subgroups induces nontrivial isomorphisms between Shimura curves at different levels. To prove this, we need an important result of Eichler. 

\begin{lemma}[Eichler Approximation] \cite[Lemma~2.15]{shimura:curves}
Let $\cQ \subset B$ be a maximal order in a quaternion algebra over $F$ and $\mathfrak{I} \subset \cQ$ a two-sided ideal. Let $\beta \in \cO_F$ and $\gamma \in \cQ$ such that $\beta$ is positive at the non-split non-archimedean places of $B$ and $\Nm(\gamma) \equiv \beta \mod (\mathfrak{I} \cap \cO_F)$. Then there exists $\gamma' \in \cQ$ such that $\gamma' \equiv \gamma \mod \mathfrak{I}$ and $\Nm(\gamma') = \beta$. 
\end{lemma}

\begin{prop} \label{prop:conjugation_isomorphism}
Let $H \subset \GL_2(\cO_{\p})$ be an open subgroup for a prime $\p$ at which $B$ is split contained in the preimage of the upper Borel of $\GL_2(\kappa_{\p})$. Then the Shimura curves $\X_H(\cQ)$ and $\X_{H^{\varpi}}(\cQ)$ are isomorphic as curves over $F^{(\frJ \infty)}$.
\end{prop}

\begin{proof}
We need to show that there is an automorphism of $\frH$ taking the action of $\Gamma_H(\cQ)$ to $\Gamma_{H^{\varpi}}(\cQ)$. A natural candidate is $\nu_1(\varpi)$ but it is not clear that this element actually does the job since conjugation by it may not preserve $B$. Instead, we will $\pi$-adically approximate $\varpi$ by an element of $\cQ$. Using the isomorphism $\iota_{\frJ} : \cQ / \p^{\ell} \cQ \iso M_2(\cO_F / \p^{\ell})$ we can find an element $\varpi_{\ell}' \in \cQ$ so that $\varpi_{\ell}' \equiv \varpi \mod \p^{\ell}$. Now we use Eichler's approximation theorem to find $\varpi_{\ell} \in \cQ$ such that $\Nm(\varpi_{\ell}) = \pi$ and $\varpi_{\ell} \equiv \varpi \mod \p^{\ell}$. It is clear that $(-)^{\varpi_{\ell}}$ preserves $B^+$. Since $\pi$ is totally positive, $\iota_{\infty}(\varpi_{\ell})$ acts on $\frH$. Now, if $H$ has level $\p^{\ell}$ then notice
\[ H^{\varpi} = H^{\varpi_{\ell}} \]
because $\varpi_{\ell} \gamma \varpi^{-1}_{\ell} \equiv \varpi \gamma \varpi^{-1} \mod \p^{\ell}$ and $H$ contains every element which is $1 \mod \p^{\ell}$. Now, $B^+$ is stable under $(-)^{\varpi_{\ell}}$ but $\cQ^{\times}$ is not, so we need to be careful in checking that $(-)^{\varpi_{\ell}}$ takes $\Gamma_H(\cQ)$ to $\Gamma_{H^{\varpi}}(\cQ)$. We do this at each local place using that 
\[ \cQ = \bigcap_{\p' \subset \cO_F} (\cQ_{\p'} \cap B) \]
with the intersection taking place in $B$. Since $\Nm(\varpi_{\ell}) = \pi$ is a unit for $\p' \neq \p$, it preserves $\cQ_{\p'}$. At $\p$, an element of $\Gamma_H(\cQ)$ lands inside the maximal order $M_2(\cO_{\p})$. Therefore $(-)^{\varpi_{\ell}}$ sends $\Gamma_H(\cQ)$ inside $\cQ$. Its image must land inside $\cQ^{\times}$ since it is multiplicative and $\Gamma_H(\cQ)$ is a subgroup.  
\end{proof}

For modular curves, Deligne--Rapoport prove that this isomorphism is defined over $\Q$ \cite[Propositions IV-3.16 and IV-3.19]{deligne_rapoport}. It is likely that the above isomorphism is defined over the reflex field $F^{(\p \infty)}$ but we will not need this. 

\begin{example} \label{example:conjugation_p_squared_level}
For example, let $\p \subset \cO_F$ be a prime ideal generated by a totally positive element $\pi$ at which $B$ is split. The full congruence subgroup $\Gamma(\cQ, \p)$ is associated to the subgroup $H := \ker{(\GL_2(\cO_{\p}) \to \GL_2(\cO_F / \p)}$ which satisfies
\[ H^{\varpi} = \left \{ 
\begin{pmatrix}
d & - c
\\
- \pi b & a
\end{pmatrix} \: \middle| \: 
\begin{pmatrix}
a & b
\\
\pi c & d
\end{pmatrix} 
\in H 
\right \} = \left \{ \begin{pmatrix}
1 + \pi d & - c
\\
- \pi^2 b & 1 + \pi a
\end{pmatrix} \: \middle| \: 
a, b, c, d \in \cO_{\p}
\right\}  \]
which is the preimage of the intersection of the Borel in $\GL_2(\cO_F / \p^2)$ and the unipotent radical in $\GL_2(\cO_F / \p)$. Call this subgroup $H_{\mathrm{ub}}$ and $\X_{\mathrm{ub}}(\cQ, \p^2)$ its associated Shimura curve. Likewise, we can ask what larger subgroup $H'$ (hence also of level $\p$) corresponds to the full Borel $B(\p^2) \subset \GL_2(\cO_F / \p^2)$. This turns out to be the subgroup
\[ H_{\mathrm{sp}} = \left \{
\begin{pmatrix}
a & \pi b
\\
\pi c & d
\end{pmatrix} 
\: \middle| \: 
a, b, c, d \in \cO_{\p}
\right\}  \]
which is the preimage of the split Cartan subgroup of diagonal matrices in $\GL_2(\cO_F / \p)$. Indeed,
\[ H_{\mathrm{sp}}^{\varpi} = \left \{ \begin{pmatrix}
d & - c
\\
- \pi^2 b & a
\end{pmatrix} \: \middle| \: 
a, b, c, d \in \cO_{\p}
\right\}  \] 
is exactly the Borel in $\GL_2(\cO_F/\p^2)$. Therefore, Proposition~\ref{prop:conjugation_isomorphism} produces a diagram
\begin{center}
    \begin{tikzcd}
        \X_{\mathrm{ub}}(\cQ, \p^2) \arrow[d] \arrow[r, "\sim"] & \X(\cQ, \p) \arrow[d]
        \\
        \X_0(\cQ, \p^2) \arrow[r, "\sim"] & \X_{\mathrm{sp}}(\cQ, \p) 
    \end{tikzcd}
\end{center}
\end{example}

\subsection{$\Jac(\X^{\Hur}(\p_{13})$}

Here we compute the $j$-invariant of the elliptic curve $E$ appearing in the isogeny decomposition of $\Jac(\X^{\Hur}(\p_{13}))$. From Figure~\ref{fig:subgroups_genera} we know that $\X_{\mathrm{sp}}^{\Hur}(\p_{13})$, the quotient of $\X^{\Hur}(\p_{13})$ by the maximal torus $C_6$, is a genus $2$ curve. Since $\Jac(\X^{\Hur}(\p_{13})) \sim E^{14}$ over $K$ it suffices to compute $\Jac(\X_1^{\Hur}(\p_{13}))$. The conjugation isomorphism of Example~\ref{example:conjugation_p_squared_level} exhibits $\X_{\mathrm{sp}}^{\Hur}(\p_{13}) \iso \X_0^{\Hur}(\p_{13}^2)$. The Eichler--Shimizu--Jacquet--Langlands correspondence,
\[ S^B_2(\p_{13}^2) \cong S_2(\p_{13}^2) \]
expresses the Hecke action on $H^1(\X_0^{\Hur}(\p_{13}^2)_{\CC}, \QQ)$ in terms of Hilbert cusp forms of parallel weight $2$ for level $\Gamma_0(\p_{13}^2)$. Fixing the prime $\p_{13}$ over $13$, the space $S_2(\p_{13}^2)$ consists of a unique Galois orbit $f$ of Hilbert eigenforms \cite[\href{https://www.lmfdb.org/ModularForm/GL2/TotallyReal/3.3.49.1/holomorphic/3.3.49.1-169.4-a}{3.3.49.1-169.4-a}]{lmfdb}. The Hecke trace field is $K_f = \Q(\sqrt{3})$, hence $f$ corresponds to the abelian surface $A_f \cong \Jac(\X_0^{\Hur}(\p_{13}^2))$. We expect an elliptic curve $E$ in the isogeny decomposition over $K$ but the Hecke eigenvalues are not rational so the eigenform cannot correspond to an elliptic curve over $K_+$. Indeed, $K$ is the field of definition of $E$ and there is an isogeny $A_f \iso \Res^{K}_{K_+} E$. To find $E$ we analyze the Galois representation $\rho_{A_f, \ell}$ attached to the Tate module $V_{\ell}(A_f)$. Eichler-Shimura gives
\[ \tr{\rho_{A_f, \ell}(\Frob_{\p})} = a_{\p}(f) \] 
for $\p \subset \cO_{K_+}$ not dividing $\ell \p_{13}$. 
If $V_{\ell}(A_f)$ splits when restricted to $\Gal(\ol{K} / K)$, we can compute its traces of Frobenius directly. Consider $\frP \subset \cO_{K}$ lying over $\p \subset \cO_{K_+}$. If $\p$ is split or ramified then $\Frob_{\frP}$ and $\Frob_{\p}$ are conjugate in $\Gal(\ol{K_+} / K_+)$, so $\tr{\rho_{A_f, \ell}(\Frob_{\frP})} = a_{\p}(f)$. If $\p$ is inert, then $\Frob_{\frP}$ is conjugate to $\Frob_{\p}^2$ in $\Gal(\ol{K_+} / K_+)$ and therefore, by Cayley--Hamilton,
\[ \tr{\rho_{A_f, \ell}(\Frob_{\frP})} = (\tr{\rho_{A_f, \ell}(\Frob_{\p})})^2 - 2 \det{\rho_{A_f, \ell}(\Frob_{\p})} = a_{\p}(f)^2 - 2 N(\p). \] 
As expected, every irrational value of $a_{\p}(f)$ appears for $\p$ inert in $K$ and $a_{\p}(f)^2 \in \ZZ$ meaning, in particular, the traces are the same for the two eigenforms exchanged by Galois. Hence, we get a unique list of rational traces $\tr{\rho_{A_f, \ell}(\Frob_{\frP})}$ for the primes $\frP \subset \cO_K$ corresponding to the decomposition $(A_f)_K \sim E^2$ also showing that $E$ is isogenous to its conjugate under $\Gal(K / K_+)$ (which also follows from the fact that $A_f$ has an isogeny action of $K_f = \Q(\sqrt{3})$). Since we know $E$ is defined over $K$ and has conductor dividing $\p_{13}^2$, since $\X_0^{\Hur}(\p_{13})$ has good reduction away from $\p_{13}$ by \cite{carayol,kisin}, there are finitely many potential elliptic curves. The built in \textsc{Magma} function \texttt{EllipticCurveSearch} identifies $E$ and its conjugate from finitely many $\Frob_{\frP}$ traces (see the script \texttt{hilbert\_modular\_forms.m}). The minimal polynomial of $j(E) \in K$ is 
\begin{align*}
    x^6 & - 14475 x^5 + 66476374 x^4 + 86360537888 x^3  - 1766214565832855 x^2 
    \\
    & + 5634850735711464943 x - 5057260487992776789167. 
\end{align*}
This elliptic curve has bad reduction exactly at the unique prime of $K$ lying over $\p_{13}$ which is also where $K / K_+$ is ramified. The first few supersingular primes of $E$ have norms $1091, 2339, 6551, 7643, \dots$.

\subsection{Cyclic quotient singularities in mixed characteristic}

In this section, we show that the product-quotient surfaces $X$ have good reduction away from the order of $G$ and the bad primes of the curves $C_1$ and $C_2$. Of course, $X$ has only finitely many primes of bad reduction so this section is only necessary to get an explicit set of primes in Theorem~\ref{intro:thm:example}. To do this, we study cyclic quotient singularities and their resolutions in mixed characteristic. For simplicity, we only study the case of \textit{tame} group actions, meaning those for which $\# G$ is coprime to the residue characteristic.

\begin{lemma} \label{lemma:etale_fixed_locus}
Let $\cC \to S$ be a smooth proper relative curve of genus $g \ge 2$. Let $G$ be a finite group acting on $\cC$ as an $S$-scheme whose action is faithful on the generic fiber. If $\# G$ is coprime to all residue characteristics of $S$, then for any $g \in G$ the fixed point scheme $\Fix_g(\cC)$ is finite \etale over $S$. 
\end{lemma}

\begin{proof}
First we show that $\Aut(\cC/S) \to S$ is finite and unramified. Since $g \ge 2$, the infinitesimal automorphisms $H^0(\cC_s, \T_{\cC_s}) = 0$ vanish on each fiber. Therefore, infinitesimal deformation theory shows that $\Aut(\cC/S)$ is formally unramified. The theory of regular proper models over a DVR shows that $\Aut(\cC/S)$ satisfies the valuative criterion of properness. 
\par 
Properness of $\Fix_g(\cC) \to S$ is automatic from the definition. It suffices to show it is \etale. This locus is the intersection of two flat divisors on $\cC \times_S \cC$, namely the diagonal $\Delta$ and the graph $\Gamma_g$ of $g$. Their intersection is \etale over $S$ if and only if they are transverse on each geometric fiber. To prove this, we must show that if $\varphi \in \Aut(C)$ is a nontrivial automorphism of a smooth proper curve $C$ over an algebraically closed field $k$ of characteristic $p$ and $\varphi$ has finite order coprime to $p$ then $\d{\varphi}$ is nonzero at every fixed point. Since $\varphi$ is nontrivial, $\Fix_{\varphi}(C)$ is a finite collection of points. For $x \in \Fix_{\varphi}(C)$, we will show that if $\d{\varphi}|_x = 0$ then $\varphi$ acts trivially on the formal local ring $\wh{\cO}_{C,x}$ which implies it is trivial on an open $x \in U \subset C$ and hence trivial on $C$ since $C$ is separated. Using induction and the sequences
\[ 0 \to \m_x^n / \m_x^{n+1} \to \cO_{C,x} / \m_x^{n+1} \to \cO_{C,x} / \m_x^{n} \to 0 \]
we will show $\varphi : \wh{\cO}_{C,x} \to \wh{\cO}_{C,x}$ is trivial. The case $n = 1$ is by assumption. Regularity gives $\m_x^n / \m_x^{n+1} = \Sym^n (\m_x / \m_x^2)$ and we assume $\varphi$ is the identity on $\m_x / \m_x^2$. Therefore, assuming towards induction that $\varphi_n := \varphi \ot \cO / \m_x^n = \id$ then $\varphi_{n+1} - \id$ defines a $k$-linear map,
\[ \delta : \cO_{C,x} / \m_x^{n} \to \Sym^n (\m_x / \m_x^2). \]
Now if we iterate the automorphism, $\varphi^k_{n+1} = (\id + \delta)^k = \id + k \delta$ since $\delta^2 = 0$ because it lands in a square zero ideal. But $\varphi$ has finite order $m$, so we must have $m \delta = 0$ and we assumed $m$ is prime to $p$ so $\delta = 0$. Hence $\varphi$ is trivial by induction. 
\end{proof}

\begin{lemma} \label{lemmma:is_cyclic_quotient}
Let $n$ be coprime to the characteristic of $\bar{k}$ and $C_n \acts X$ be the action of a cyclic group on a smooth surface over $\bar{k}$ with isolated fixed points. Then each singularity of $X / C_n$ is a $\frac{1}{n'}(1,a)$ singularity for some $n' \mid n$ and $a$ coprime to $n'$.
\end{lemma}

\begin{proof}
We need to show this singularity is formally equivalent to 
\[ \A^2 / \left< 
\begin{pmatrix}
\zeta & 0
\\
0 & \zeta^a
\end{pmatrix} \right> \]
for some choice of $n$-th root of unity $\zeta \in \bar{k}$. This follows from the combination of \cite[Proposition~2.4.8]{wild_quotient} and \cite[Proposition~2.3.6]{wild_quotient} to pass to the completion.
\end{proof}

\begin{prop} \label{prop:arithmetic_resolution}
Let $\cC_1 \to S$ and $\cC_2 \to S$ be smooth proper relative curves equipped with an action, over $S$, of a finite group $G$ satisfying the assumptions of Lemma~\ref{lemma:etale_fixed_locus}. Then there exists a smooth proper $S$-scheme $\X \to S$ equipped with a $S$-morphism $\pi : \X \to (\cC_1 \times_S \cC_2) / G$ which is identified with the Hirzebruch--Jung resolution on each geometric fiber. 
\end{prop}

\begin{proof}
This is a consequence of the fact that the algorithm to produce Hirzebruch--Jung resolutions can be run globally and that (weighted) blowups along finite \etale multi-sections (indeed along any subscheme flat over $S$) are compatible with base change. This follows from the ideal sheaf sequence for a closed subscheme $Z \subset X$,
\[ 0 \to \mathcal{I}_Z \to \struct{X} \to \struct{Z} \to 0 \]
being compatible with base change along $S' \to S$ when $\struct{Z}$ is flat over $\cO_S$.
\par 
Over an algebraically closed field $\bar{k}$, the Hirzebruch--Jung resolution of a cyclic $\frac{1}{n}(1, a)$ singularity is obtained by a sequence of weighted blowups constructed from the ideals of the two torus invariant divisors  (see \cite{weighted_blowups}). This sequence of weighted blowups is determined entirely from the Hirzebruch--Jung continued fraction representation $\frac{n}{a} = \dbrac{b_1, \dots, b_{\ell}}$. 
\par 
The singular locus of $\mathcal{Y} = (\cC_1 \times_S \cC_2)/G$ consists of a collection of finite \etale multi-sections whose geometric fibers are isolated cyclic quotient singularities. Furthermore, the images in $\mathcal{Y}$ of $\cC_1 \times_S \Fix_g(\cC_2)$ and $\Fix_g(\cC_1) \times_S \cC_2$ as $g$ ranges over nontrivial elements of $G$ produce a collection of Weil divisors flat over $S$ whose complement is a toroidal embedding. In particular, near each singularity on a geometric fiber of $\mathcal{Y}$, there is a pair of such divisors formally equivalent to the torus-invariant coordinate axes in the standard quotient form of the singularity. Hence running the weighted-blowups with the Hirzebruch--Jung weights along these divisors produces a relative flat proper surface $\X \to \mathcal{Y}$ over $S$. Since the blowups commute with base change, on each fiber $\X$ is the Hirzebruch--Jung resolution. In particular, $\X \to S$ is smooth.
\end{proof}

\begin{rmk}
It is not necessary to use weighted blowups if we only ask that some smooth projective birational model of $X_{\eta}$ has good reduction. Indeed, the previous results show that the singular locus of $(\cC_1 \times_S \cC_2) / G$ is finite \etale over $S$. Blowing up these loci repeatedly does give a simultaneous resolution (this runs Lipman's algorithm fiberwise as no non-normal singularities are introduced as can be seen by considering the subdivision of the fan associated to blowing up a cyclic quotient singularity) but it is not minimal. Since the existence of diagonal symmetric forms is a birational invariant of smooth proper models, this is sufficient for our purposes. We can check existence on the minimal resolution and pull back these forms to any other model before specializing by upper semi-continuity. 
\end{rmk}

Now we use this to justify the claims made in the main text about the places of good reduction of product quotient surfaces. It follows from the existence of integral canonical models of Shimura varieties of hyperspecial level \cite{carayol, kisin} that Shimura curves have good reduction away from the level and the discriminant of $B$. In particular, $\X^{\Hur}(\p_{13})$ has good reduction away from $\p_{13}$. Therefore, the product quotient surfaces in Theorem~\ref{intro:thm:example} have good reduction away from primes over $13$ and $\# G$, which is divisible only by $2,3,7,13$. In particular, the supersingular primes of $E$, which are all $>1000$, are primes of good supersingular reduction for the surfaces $X$ in Theorem~\ref{intro:thm:example}.

\bibliographystyle{alpha}
\bibliography{refs}

\begin{bibdiv}
\begin{biblist}

\bib{Arm65}{article}{
      author={Armstrong, M.~A.},
       title={On the fundamental group of an orbit space},
        date={1965},
        ISSN={0008-1981},
     journal={Proc. Cambridge Philos. Soc.},
      volume={61},
       pages={639\ndash 646},
         url={https://doi.org/10.1017/s0305004100038974},
      review={\MR{187244}},
}

\bib{Arm68}{article}{
      author={Armstrong, M.~A.},
       title={The fundamental group of the orbit space of a discontinuous group},
        date={1968},
        ISSN={0008-1981},
     journal={Proc. Cambridge Philos. Soc.},
      volume={64},
       pages={299\ndash 301},
         url={https://doi.org/10.1017/s0305004100042845},
      review={\MR{221488}},
}

\bib{bauer:fundamental_groups}{article}{
      author={Bauer, Ingrid},
      author={Catanese, Fabrizio},
      author={Grunewald, Fritz},
      author={Pignatelli, Roberto},
       title={Quotients of products of curves, new surfaces with {$p_g=0$} and their fundamental groups},
        date={2012},
        ISSN={0002-9327,1080-6377},
     journal={Amer. J. Math.},
      volume={134},
      number={4},
       pages={993\ndash 1049},
         url={https://doi.org/10.1353/ajm.2012.0029},
      review={\MR{2956256}},
}

\bib{bruin:an-singularities}{article}{
      author={Bruin, Nils},
      author={Ilten, Nathan},
      author={Xu, Zhe},
       title={Local {E}uler characteristics of {$A_n$}-singularities and their application to hyperbolicity},
        date={2025},
        ISSN={2491-6765},
     journal={\'Epijournal G\'eom. Alg\'ebrique},
      volume={9},
       pages={Art. 2, 28},
      review={\MR{4864400}},
}

\bib{enriques_unirational}{article}{
      author={Blass, Piotr},
       title={Unirationality of {E}nriques surfaces in characteristic two},
        date={1982},
        ISSN={0010-437X,1570-5846},
     journal={Compositio Math.},
      volume={45},
      number={3},
       pages={393\ndash 398},
         url={http://www.numdam.org/item?id=CM_1982__45_3_393_0},
      review={\MR{656613}},
}

\bib{Bog}{article}{
      author={Bogomolov, F.~A.},
       title={Families of curves on a surface of general type},
        date={1977},
        ISSN={0002-3264},
     journal={Dokl. Akad. Nauk SSSR},
      volume={236},
      number={5},
       pages={1041\ndash 1044},
      review={\MR{457450}},
}

\bib{bauer:classification}{article}{
      author={Bauer, I.},
      author={Pignatelli, R.},
       title={The classification of minimal product-quotient surfaces with {$p_g=0$}},
        date={2012},
        ISSN={0025-5718,1088-6842},
     journal={Math. Comp.},
      volume={81},
      number={280},
       pages={2389\ndash 2418},
         url={https://doi.org/10.1090/S0025-5718-2012-02604-4},
      review={\MR{2945163}},
}

\bib{carayol}{article}{
      author={Carayol, Henri},
       title={Sur la mauvaise r\'eduction des courbes de {S}himura},
        date={1986},
        ISSN={0010-437X,1570-5846},
     journal={Compositio Math.},
      volume={59},
      number={2},
       pages={151\ndash 230},
         url={http://www.numdam.org/item?id=CM_1986__59_2_151_0},
      review={\MR{860139}},
}

\bib{enriques}{book}{
      author={Cossec, Fran\c~cois},
      author={Dolgachev, Igor},
      author={Liedtke, Christian},
       title={Enriques {S}urfaces {I}},
     edition={},
   publisher={Springer, Singapore},
        date={2025},
        ISBN={978-981-96-1213-0; 978-981-96-1214-7},
         url={https://doi.org/10.1007/978-981-96-1214-7},
      review={\MR{4911529}},
}

\bib{github}{unpublished}{
      author={Church, Benjamin},
       title={{obstructions-to-unirationality-code}},
        date={2025},
         url={https://github.com/benvchurch/obstructions-to-unirationality-code},
        note={\url{{https://github.com/benvchurch/obstructions-to-unirationality-code}}},
}

\bib{voight:triangle_groups}{article}{
      author={Clark, Pete~L.},
      author={Voight, John},
       title={Algebraic curves uniformized by congruence subgroups of triangle groups},
        date={2019},
        ISSN={0002-9947,1088-6850},
     journal={Trans. Amer. Math. Soc.},
      volume={371},
      number={1},
       pages={33\ndash 82},
         url={https://doi.org/10.1090/tran/7139},
      review={\MR{3885137}},
}

\bib{deligne_rapoport}{incollection}{
      author={Deligne, P.},
      author={Rapoport, M.},
       title={Les sch\'emas de modules de courbes elliptiques},
        date={1973},
   booktitle={Modular functions of one variable, {II} ({P}roc. {I}nternat. {S}ummer {S}chool, {U}niv. {A}ntwerp, {A}ntwerp, 1972)},
      series={Lecture Notes in Math.},
      volume={Vol. 349},
   publisher={Springer, Berlin-New York},
       pages={143\ndash 316},
      review={\MR{337993}},
}

\bib{Voight:explicit_methods}{incollection}{
      author={Demb\'el\'e, Lassina},
      author={Voight, John},
       title={Explicit methods for {H}ilbert modular forms},
        date={2013},
   booktitle={Elliptic curves, {H}ilbert modular forms and {G}alois deformations},
      series={Adv. Courses Math. CRM Barcelona},
   publisher={Birkh\"auser/Springer, Basel},
       pages={135\ndash 198},
         url={https://doi.org/10.1007/978-3-0348-0618-3_4},
      review={\MR{3184337}},
}

\bib{Ek87}{incollection}{
      author={Ekedahl, Torsten},
       title={Foliations and inseparable morphisms},
        date={1987},
   booktitle={Algebraic geometry, {B}owdoin, 1985 ({B}runswick, {M}aine, 1985)},
      series={Proc. Sympos. Pure Math.},
      volume={46, Part 2},
   publisher={Amer. Math. Soc., Providence, RI},
       pages={139\ndash 149},
         url={https://doi.org/10.1090/pspum/046.2/927978},
      review={\MR{927978}},
}

\bib{Elkies}{article}{
      author={Elkies, Noam~D.},
       title={Supersingular primes for elliptic curves over real number fields},
        date={1989},
        ISSN={0010-437X,1570-5846},
     journal={Compositio Math.},
      volume={72},
      number={2},
       pages={165\ndash 172},
         url={http://www.numdam.org/item?id=CM_1989__72_2_165_0},
      review={\MR{1030140}},
}

\bib{klein_quartic}{incollection}{
      author={Elkies, Noam~D.},
       title={The {K}lein quartic in number theory},
        date={1999},
   booktitle={The eightfold way},
      series={Math. Sci. Res. Inst. Publ.},
      volume={35},
   publisher={Cambridge Univ. Press, Cambridge},
       pages={51\ndash 101},
      review={\MR{1722413}},
}

\bib{Fulton:toric_varieties}{book}{
      author={Fulton, William},
       title={Introduction to toric varieties},
      series={Annals of Mathematics Studies},
   publisher={Princeton University Press, Princeton, NJ},
        date={1993},
      volume={131},
        ISBN={0-691-00049-2},
         url={https://doi.org/10.1515/9781400882526},
        note={The William H. Roever Lectures in Geometry},
      review={\MR{1234037}},
}

\bib{GKK10}{article}{
      author={Greb, Daniel},
      author={Kebekus, Stefan},
      author={Kov\'acs, S\'andor~J.},
       title={Extension theorems for differential forms and {B}ogomolov-{S}ommese vanishing on log canonical varieties},
        date={2010},
        ISSN={0010-437X,1570-5846},
     journal={Compos. Math.},
      volume={146},
      number={1},
       pages={193\ndash 219},
         url={https://doi.org/10.1112/S0010437X09004321},
      review={\MR{2581247}},
}

\bib{GKKP}{article}{
      author={Greb, Daniel},
      author={Kebekus, Stefan},
      author={Kov\'acs, S\'andor~J.},
      author={Peternell, Thomas},
       title={Differential forms on log canonical spaces},
        date={2011},
        ISSN={0073-8301,1618-1913},
     journal={Publ. Math. Inst. Hautes \'Etudes Sci.},
      number={114},
       pages={87\ndash 169},
         url={https://doi.org/10.1007/s10240-011-0036-0},
      review={\MR{2854859}},
}

\bib{SGA1}{book}{
      author={Grothendieck, Alexander},
       title={Rev\^etements \'etales et groupe fondamental (sga 1)},
      series={Lecture notes in mathematics},
   publisher={Springer-Verlag},
        date={1971},
      volume={224},
}

\bib{lang_product_quotient}{article}{
      author={Grivaux, Julien},
      author={Restrepo~Velasquez, Juliana},
      author={Rousseau, Erwan},
       title={On {L}ang's conjecture for some product-quotient surfaces},
        date={2018},
        ISSN={0391-173X,2036-2145},
     journal={Ann. Sc. Norm. Super. Pisa Cl. Sci. (5)},
      volume={18},
      number={4},
       pages={1483\ndash 1501},
      review={\MR{3829753}},
}

\bib{Voight:computing_hecke}{article}{
      author={Greenberg, Matthew},
      author={Voight, John},
       title={Computing systems of {H}ecke eigenvalues associated to {H}ilbert modular forms},
        date={2011},
        ISSN={0025-5718,1088-6842},
     journal={Math. Comp.},
      volume={80},
      number={274},
       pages={1071\ndash 1092},
         url={https://doi.org/10.1090/S0025-5718-2010-02423-8},
      review={\MR{2772112}},
}

\bib{Hida}{article}{
      author={Hida, Haruzo},
       title={On abelian varieties with complex multiplication as factors of the {J}acobians of {S}himura curves},
        date={1981},
        ISSN={0002-9327,1080-6377},
     journal={Amer. J. Math.},
      volume={103},
      number={4},
       pages={727\ndash 776},
         url={https://doi.org/10.2307/2374049},
      review={\MR{623136}},
}

\bib{wild_quotient}{thesis}{
      author={Kiraly, Franz},
       title={Wild quotient singularities of arithmetic surfaces and their regular models},
        type={Ph.D. Thesis},
        date={2010},
}

\bib{kisin}{article}{
      author={Kisin, Mark},
       title={Integral models for {S}himura varieties of abelian type},
        date={2010},
        ISSN={0894-0347,1088-6834},
     journal={J. Amer. Math. Soc.},
      volume={23},
      number={4},
       pages={967\ndash 1012},
         url={https://doi.org/10.1090/S0894-0347-10-00667-3},
      review={\MR{2669706}},
}

\bib{Kol93}{article}{
      author={Koll\'ar, J\'anos},
       title={Shafarevich maps and plurigenera of algebraic varieties},
        date={1993},
        ISSN={0020-9910,1432-1297},
     journal={Invent. Math.},
      volume={113},
      number={1},
       pages={177\ndash 215},
         url={https://doi.org/10.1007/BF01244307},
      review={\MR{1223229}},
}

\bib{kovacs_rational}{article}{
      author={Kov\'acs, S\'andor~J.},
       title={A characterization of rational singularities},
        date={2000},
        ISSN={0012-7094,1547-7398},
     journal={Duke Math. J.},
      volume={102},
      number={2},
       pages={187\ndash 191},
         url={https://doi.org/10.1215/S0012-7094-00-10221-9},
      review={\MR{1749436}},
}

\bib{lmfdb}{misc}{
      author={{LMFDB Collaboration}, The},
       title={The {L}-functions and modular forms database},
         how={\url{https://www.lmfdb.org}},
        date={2025},
        note={[Online; accessed 22 June 2025]},
}

\bib{milne}{incollection}{
      author={Milne, J.~S.},
       title={Introduction to {S}himura varieties},
        date={2005},
   booktitle={Harmonic analysis, the trace formula, and {S}himura varieties},
      series={Clay Math. Proc.},
      volume={4},
   publisher={Amer. Math. Soc., Providence, RI},
       pages={265\ndash 378},
      review={\MR{2192012}},
}

\bib{miyaoka}{article}{
      author={Miyaoka, Yoichi},
       title={The maximal number of quotient singularities on surfaces with given numerical invariants},
        date={1984},
        ISSN={0025-5831,1432-1807},
     journal={Math. Ann.},
      volume={268},
      number={2},
       pages={159\ndash 171},
         url={https://doi.org/10.1007/BF01456083},
      review={\MR{744605}},
}

\bib{local_vanishing}{article}{
      author={Musta\c~t\u a, Mircea},
      author={Olano, Sebasti\'an},
      author={Popa, Mihnea},
       title={Local vanishing and {H}odge filtration for rational singularities},
        date={2020},
        ISSN={1474-7480,1475-3030},
     journal={J. Inst. Math. Jussieu},
      volume={19},
      number={3},
       pages={801\ndash 819},
         url={https://doi.org/10.1017/s1474748018000208},
      review={\MR{4094707}},
}

\bib{pathologies}{article}{
      author={Mumford, David},
       title={Pathologies of modular algebraic surfaces},
        date={1961},
        ISSN={0002-9327,1080-6377},
     journal={Amer. J. Math.},
      volume={83},
       pages={339\ndash 342},
         url={https://doi.org/10.2307/2372959},
      review={\MR{124328}},
}

\bib{weighted_blowups}{article}{
      author={Quek, Ming~Hao},
      author={Rydh, David},
       title={Weighted blow-ups},
        date={2021},
     journal={preparation. https:},
        ISSN={/people. },
}

\bib{RR13}{article}{
      author={Roulleau, Xavier},
      author={Rousseau, Erwan},
       title={Canonical surfaces with big cotangent bundle},
        date={2014},
        ISSN={0012-7094,1547-7398},
     journal={Duke Math. J.},
      volume={163},
      number={7},
       pages={1337\ndash 1351},
         url={https://doi.org/10.1215/00127094-2681496},
      review={\MR{3205727}},
}

\bib{images_of_galois}{article}{
      author={Rouse, Jeremy},
      author={Sutherland, Andrew~V.},
      author={Zureick-Brown, David},
       title={{$\ell$}-adic images of {G}alois for elliptic curves over {$\Bbb{Q}$} (and an appendix with {J}ohn {V}oight)},
        date={2022},
        ISSN={2050-5094},
     journal={Forum Math. Sigma},
      volume={10},
       pages={Paper No. e62, 63},
         url={https://doi.org/10.1017/fms.2022.38},
        note={With an appendix with John Voight},
      review={\MR{4468989}},
}

\bib{shimura:curves}{article}{
      author={Shimura, Goro},
       title={Construction of class fields and zeta functions of algebraic curves},
        date={1967},
        ISSN={0003-486X},
     journal={Ann. of Math. (2)},
      volume={85},
       pages={58\ndash 159},
         url={https://doi.org/10.2307/1970526},
      review={\MR{204426}},
}

\bib{Shioda:unirational}{article}{
      author={Shioda, Tetsuji},
       title={Some results on unirationality of algebraic surfaces},
        date={1977},
        ISSN={0025-5831,1432-1807},
     journal={Math. Ann.},
      volume={230},
      number={2},
       pages={153\ndash 168},
         url={https://doi.org/10.1007/BF01370660},
      review={\MR{572983}},
}

\bib{Tate}{incollection}{
      author={Tate, John},
       title={Conjectures on algebraic cycles in {$l$}-adic cohomology},
        date={1994},
   booktitle={Motives ({S}eattle, {WA}, 1991)},
      series={Proc. Sympos. Pure Math.},
      volume={55, Part 1},
   publisher={Amer. Math. Soc., Providence, RI},
       pages={71\ndash 83},
         url={https://doi.org/10.1090/pspum/055.1/1265523},
      review={\MR{1265523}},
}

\bib{voight:thesis}{book}{
      author={Voight, John~Michael},
       title={Quadratic forms and quaternion algebras: {A}lgorithms and arithmetic},
   publisher={ProQuest LLC, Ann Arbor, MI},
        date={2005},
        ISBN={978-0542-29154-8},
         url={http://gateway.proquest.com/openurl?url_ver=Z39.88-2004&rft_val_fmt=info:ofi/fmt:kev:mtx:dissertation&res_dat=xri:pqdiss&rft_dat=xri:pqdiss:3187174},
        note={Thesis (Ph.D.)--University of California, Berkeley},
      review={\MR{2707756}},
}

\bib{quaternion_algebras}{book}{
      author={Voight, John},
       title={Quaternion algebras},
      series={Graduate Texts in Mathematics},
   publisher={Springer, Cham},
        date={2021},
      volume={288},
        ISBN={978-3-030-56692-0; 978-3-030-56694-4},
         url={https://doi.org/10.1007/978-3-030-56694-4},
      review={\MR{4279905}},
}

\bib{zhang}{article}{
      author={Zhang, Shouwu},
       title={Heights of {H}eegner points on {S}himura curves},
        date={2001},
        ISSN={0003-486X,1939-8980},
     journal={Ann. of Math. (2)},
      volume={153},
      number={1},
       pages={27\ndash 147},
         url={https://doi.org/10.2307/2661372},
      review={\MR{1826411}},
}

\end{biblist}
\end{bibdiv}

\end{document}